%% file: main.tex
\documentclass[12pt,english]{article}
\usepackage[a4paper,left=0.75in,right=0.75in,top=1in,bottom=1in]{geometry}

\usepackage[ngerman, main=english]{babel}
\usepackage[maxbibnames=5,backend=biber]{biblatex}
\usepackage[autostyle]{csquotes}

\makeatletter
\def\blx@err@patch#1{}
\makeatother

\usepackage{graphicx}
\usepackage{subcaption}

\usepackage{tabulary}
\usepackage{booktabs}   % Verbesserte Möglichkeiten für Tabellenlayout über horizontale Linien

\usepackage{amsmath,amssymb,amsthm}
\usepackage{mathtools}
\usepackage[shortlabels]{enumitem}
\usepackage{amsfonts}
\usepackage{stmaryrd}

\usepackage{subcaption}

\usepackage{pifont}% Zapf-Dingbats Symbole

\usepackage{xcolor}

\usepackage{tikz}
\usetikzlibrary{arrows.meta}

\usepackage{multicol}

\usepackage{pgfplots}
\pgfplotsset{width=10cm,compat=1.9}

\theoremstyle{plain}

\newtheorem{Definition}{Definition}[section]
\newtheorem{Remark}[Definition]{Remark}
\newtheorem{Theorem}[Definition]{Theorem}
\newtheorem{Lemma}[Definition]{Lemma}
\newtheorem{Proposition}[Definition]{Proposition}
\newtheorem{Algorithm}[Definition]{Algorithm}
\newtheorem{Corollary}[Definition]{Corollary}

\usepackage{hyperref}
\usepackage{hyperxmp}
\usepackage{authblk}

\hypersetup{pdflang={en-GB}}

\begin{document}

\title{A posteriori existence for the Keller-Segel model via a finite volume - finite element scheme}
\author{Marc Hoffmann\thanks{corresponding author; hoffmann@mathematik.tu-darmstadt.de}, Jan Giesselmann\thanks{giesselmann@mathematik.tu-darmstadt.de}} %optionales Argument ist die Signatur,
\affil{Department of Mathematics, Technical University of Darmstadt, \\ Dolivostr. 15, 64293 Darmstadt, Germany}
\date{}

\maketitle
\begin{abstract}
    \noindent We derive two forms of conditional \emph{a posteriori} error estimates for a finite volume scheme approximating the parabolic-elliptic Keller-Segel system. 
    The estimates control the error in the $L^\infty(0,T, L^2(\mathbb{T}^d))$-norm and exhibit linear convergence in the mesh size, as observed in numerical experiments, for $d \in \{2,3\}$.
    Crucially, we show that, as long as the condition of the error estimate is satisfied, a  weak solution exists. 
    This means, as long as the numerical solution has good properties, we can rigorously infer existence of an exact solution.
    %Crucially, we show that these error bounds can be used to rigorously verify the existence of weak solutions—up to a prescribed time horizon—based solely on numerical simulations.
\end{abstract}


%\tableofcontents
%\listoffigures

\input{notation.tex}

\input{introduction.tex}

\input{kellersegelsystem.tex}
\input{stability.tex}
\input{numericalmethod.tex}

\input{reconstruction.tex}
\input{residualestimates.tex}

\input{numericalsimulations.tex}
% %\include{parabolic-paraboliC_{\text{S}}tability.tex}
% \include{outlook.tex}

% \include{überdiesedatei.tex}
% \include{verwendung.tex}
\paragraph{Acknowledgements:} The authors thank Sebastian Franz for helpful discussions on $L^\infty$-error estimates for elliptic problems and Michael Plum for valuable exchange that led to the results in Subsection~\ref{subsec:sobolev}.
The authors are grateful for valuable feedback by the reviewers which contributed significantly to the quality of this manuscript.
The research of J.G. was supported by Deutsche Forschungsgemeinschaft (DFG, German Research Foundation) - SPP 2410 Hyperbolic Balance Laws in Fluid Mechanics: Complexity, Scales, Randomness (CoScaRa), within the Project “A posteriori error estimators for statistical solutions of barotropic Navier-Stokes equations” 525877563. The work of J.G. is also supported by the Graduate School CE within Computational Engineering at Technische Universität Darmstadt.

\input{appendix.tex}

% \cleardoublepage
% \phantomsection
% \addcontentsline{toc}{chapter}{Bibliography}
%\bibliography{mybib}

% \nocite{Hang2015}
\printbibliography

\end{document}

%% file: notation.tex
% -- Abkuerzungen neuer mathematischer Befehle --
\newcommand{\D}{\mathcal{D}}               % Definitionsbereich
\newcommand{\K}{\mathbb{K}}                % Koerper
\newcommand{\R}{\mathbb{R}}                % reelle Zahlen
\newcommand{\Q}{\mathbb{Q}}                % rationale Zahlen
\newcommand{\Z}{\mathbb{Z}}                % ganze Zahlen
\newcommand{\N}{\mathbb{N}}                % natuerliche Zahlen
\newcommand{\F}{\mathcal{F}}               % Fourier Transformation
\renewcommand{\S}{\mathcal{S}}             % Schwartz Raum
\newcommand{\SO}{\mathrm{SO}}              % special orthogonal group
\newcommand{\SE}{\mathrm{SE}}              % special euclidean group
\newcommand{\GL}{\mathrm{GL}}              % general linear group
\newcommand{\so}{\mathfrak{so}}            % special orthogonal group
\newcommand{\se}{\mathfrak{se}}            % special euclidean group
\renewcommand{\Re}{\mathrm{Re}\,}          % Realteil
\renewcommand{\Im}{\mathrm{Im}\,}          % Imaginaerteil
\providecommand{\trace}{\mathrm{tr}}           % Trace/Spur
\renewcommand{\L}{\mathcal{L}}             % L Operator
\newcommand{\A}{\mathcal{A}}               % L Operator
\newcommand{\Null}{\mathcal{N}}            % kernel, null space
\newcommand{\Range}{\mathcal{R}}           % range
\newcommand{\sech}{\mathrm{sech}}          % secans hyperbolicus
\newcommand{\erf}{\mathrm{erf}}            % Error function 
\newcommand{\diag}{\mathrm{diag}}          % Diagonalmatrix
\newcommand{\E}{\mbox{I\negthinspace E}}   % Erwartungswert
\newcommand{\supp}{\mathrm{supp}}          % Traeger, support
\newcommand{\one}{\mathbbm{1}}             % Indikatorfunktion
\newcommand{\n}{\mathrm{\textbf{n}}}       % Normalenvektor,
\newcommand{\orbit}{\mathcal{O}}           % Orbit
\newcommand{\rank}{\mathrm{rank}}          % Rank of a matrix
\newcommand{\amin}{a_{\mathrm{min}}} 
\newcommand{\amax}{a_{\mathrm{max}}} 
\newcommand{\azero}{a_0} 
\newcommand{\czero}{c_0} 
\newcommand{\dzero}{d_0}
\newcommand{\Imax}{I_{\mathrm{max}}}

\providecommand{\rd}{{\operatorname{rd}}}
\providecommand{\cond}{{\operatorname{cond}}}
\providecommand{\weight}{{\psi}} 
\providecommand{\RRm}{\ensuremath{\setR^m}}
\providecommand{\setR}{\ensuremath{\mathcal{R}}}
\providecommand{\Rmn}{\setR^{m \times n}}
\providecommand{\Rnn}{\setR^{n \times n}}
\providecommand{\Rmm}{\setR^{m \times m}}
\providecommand{\Cn}{\setC^n}
\providecommand{\Cm}{\setC^m}
\providecommand{\Cmn}{\setC^{m \times n}}
\providecommand{\Cnn}{\setC^{n \times n}}
\providecommand{\Cmm}{\setC^{m \times m}}
\providecommand{\Kn}{\ensuremath{\setK^n}}
\providecommand{\Km}{\ensuremath{\setK^m}}
\providecommand{\Kmn}{\ensuremath{\setK^{m \times n}}}
\providecommand{\Kmm}{\ensuremath{\setK^{m \times m}}}
\providecommand{\Knn}{\setK^{n \times n}}
\providecommand{\rank}{\textup{rang}}
\providecommand{\notexam}{\faMagic}
\providecommand{\notfivecp}{\faGear}
\def\RR{\mathbb{R}}
\def\CC{\mathbb{C}}
\def\NN{\mathbb{N}}
\def\ZZ{\mathbb{Z}}

\def\dt{\,\mathrm{d}t}
\def\dx{\,\mathrm{d}x}
\def\ds{\,\mathrm{d}s}
\def\dr{\,\mathrm{d}r}

% \def\y{\mathbf{y}}
% \def\z{\mathbf{z}}
% \def\f{\mathbf{f}}
% \def\g{\mathbf{g}}
% \def\r{\mathbf{r}}
% \def\b{\mathbf{b}}
% \def\c{\mathbf{c}}
% \def\d{\mathbf{d}}
% \def\k{\mathbf{k}}
% \def\e{\mathbf{e}}
% \def\u{\mathbf{u}}
% \def\v{\mathbf{v}}
% \def\w{\mathbf{w}}
% \def\x{\mathbf{x}}
% \def\p{\mathbf{p}}
% \def\h{\mathbf{h}}
% \def\A{\mathbf{A}}
% \def\B{\mathbf{B}}
% \def\C{\mathbf{C}}
% \def\D{\mathbf{D}}
% \def\E{\mathbf{E}}
% %\def\F{\mathbf{F}}
% \def\J{\mathbf{J}}
% \def\I{\mathbf{I}}
% %\def\R{\mathbf{R}}
% %\def\S{\mathbf{S}}
% \def\U{\mathbf{U}}
% \def\V{\mathbf{V}}
% \def\W{\mathbf{W}}
% \def\Y{\mathbf{Y}}
% \def\Lam{\boldsymbol{\Lambda}}
% \def\b{\mathbf{b}}
% \def\zero{\mathbf{0}}

\def\bspsi{\boldsymbol{\psi}}
\def\bbeta{\boldsymbol{\beta}}
\def\bgamma{\boldsymbol{\gamma}}
\def\M{\mathcal{M}}
\def\L{\mathcal{L}}
\def\T{\mathbb{T}}
\def\Z{\mathcal{Z}}

\def\eps{\epsilon}
\def\Re{\text{Re}}
\def\Im{\text{Im}}
\def\tint{{\textstyle \int}}
\def\tnorm{|\!|\!|}

%% differential operators 
\def\divergence{\mathrm{div}}

%% norms 

  \providecommand{\normtmp}[2]{{#1\lVert #2 #1\rVert}}
\providecommand{\norm}[1]{\normtmp{}{#1}}
\providecommand{\bignorm}[1]{\normtmp{\big}{#1}}
\providecommand{\Bignorm}[1]{\normtmp{\Big}{#1}}
\providecommand{\biggnorm}[1]{\normtmp{\bigg}{#1}}
\providecommand{\Biggnorm}[1]{\normtmp{\Bigg}{#1}}

\providecommand{\normmtmp}[2]{{#1\lvert\hspace{-0.07em}#1\lvert\hspace{-0.07em}#1\lvert{#2}
		#1\rvert\hspace{-0.07em}#1\rvert\hspace{-0.07em}#1\rvert}}  
\providecommand{\normm}[1]{\normmtmp{}{#1}}
\providecommand{\bignormm}[1]{\normmtmp{\big}{#1}}
\providecommand{\Bignormm}[1]{\normmtmp{\Big}{#1}}
\providecommand{\biggnormm}[1]{\normmtmp{\bigg}{#1}}
\providecommand{\Biggnormm}[1]{\normmtmp{\Bigg}{#1}}

\providecommand{\abstmp}[2]{{#1\lvert #2 #1\rvert}}
\providecommand{\abs}[1]{\abstmp{}{#1}}
\providecommand{\bigabs}[1]{\abstmp{\big}{#1}}
\providecommand{\Bigabs}[1]{\abstmp{\Big}{#1}}
\providecommand{\biggabs}[1]{\abstmp{\bigg}{#1}}
\providecommand{\Biggabs}[1]{\abstmp{\Bigg}{#1}}

\providecommand{\skptmp}[3]{{\ensuremath{#1\langle {#2}, {#3} #1\rangle}}}
\providecommand{\skp}[2]{\skptmp{}{#1}{#2}}
\providecommand{\bigskp}[2]{\skptmp{\big}{#1}{#2}}
\providecommand{\Bigskp}[2]{\skptmp{\Big}{#1}{#2}}
\providecommand{\biggskp}[2]{\skptmp{\bigg}{#1}{#2}}
\providecommand{\Biggskp}[2]{\skptmp{\Bigg}{#1}{#2}}

\providecommand{\hskptmp}[3]{{\ensuremath{#1( {#2}, {#3} #1)}}}
\providecommand{\hskp}[2]{\hskptmp{}{#1}{#2}}
\providecommand{\bighskp}[2]{\hskptmp{\big}{#1}{#2}}
\providecommand{\Bighskp}[2]{\hskptmp{\Big}{#1}{#2}}
\providecommand{\bigghskp}[2]{\hskptmp{\bigg}{#1}{#2}}
\providecommand{\Bigghskp}[2]{\hskptmp{\Bigg}{#1}{#2}}

\providecommand{\settmp}[2]{{#1\{{#2}#1\}}}
\providecommand{\set}[1]{\settmp{}{#1}}
\providecommand{\bigset}[1]{\settmp{\big}{#1}}
\providecommand{\Bigset}[1]{\settmp{\Big}{#1}}
\providecommand{\biggset}[1]{\settmp{\bigg}{#1}}
\providecommand{\Biggset}[1]{\settmp{\Bigg}{#1}}

%% file: introduction.tex
\section{Introduction}
\noindent The (parabolic-elliptic) Keller-Segel system is one of the iconic models in mathematical biology. 
It has received a lot of attention in analysis and numerics. 
One reason for this is that for certain initial and boundary data a weak solution exists for all times, while for others, weak solutions break down in finite time, usually by mass concentrating in a single point. 
Thus, analysts have investigated criteria on the initial data that allow one to distinguish these cases \cite{Biler1993,Nagai1995}.
In two space dimensions sharp criteria have been found \cite{Nagai1995} whereas in three space dimensions there are initial data for which long time existence of a solution is unclear.
The role of numerics in this endeavor has been to build intuition based on numerical simulations. 
While it was usually believed that the behavior displayed by numerical solutions corresponded to the behavior of the exact solution \cite{Acosta2023,Chertock2008,Huang2024} there was no way to positively assert this. 
The aim of the paper at hand is to (at least partially) change this:
We will provide conditional \emph{a posteriori} error estimates for a numerical scheme that are not only valid as long as a weak solution exists, but that allow us to infer existence of a weak solution (up to a certain time). 
This is similar in spirit to the results of \cite{Mizuguchi2017verif,Ortner2009,Plum2008} and references therein that provide similar statements for semi-linear parabolic and semi-linear elliptic equations.
See \cite{BrunkPrep,Marin2013,Morosi2012} for similar results for the incompressible Navier-Stokes equation.
The goal of our paper is to show how numerical simulations can be used to prove existence (for certain time horizons) of weak solutions to specific initial boundary value problems for the Keller-Segel system and to, hopefully, initiate further investigations addressing this question for other nonlinear evolutionary PDEs.

\noindent Conditional \emph{a posteriori} error estimates are \emph{a posteriori} error estimates that only provide an upper bound for the error if the numerical solution satisfies certain criteria -- in particular the residuals need to be small and the numerical solution is not allowed to display blow-up.
Such error estimates are quite common in the context of  nonlinear PDEs \cite{Bartels2011,Kyza2016,Kolbe2023,Kwon2023,Kyza2020,Ortner2009}, in particular if the stability theories of these equations are intricate. 
Here, by stability theory we mean a mathematical framework that provides control of the difference between (approximate) solutions. 
This is the case for e.g. the Allen-Cahn and the Cahn-Hilliard system as well as semi-linear parabolic equations with (super-)critical nonlinearity \cite{Bartels2011,Brunk2024,Kyza2016}. 
Conditional \emph{a posteriori} estimates can be grouped into three classes: those based on global-in-time continuation arguments, e.g. using the Generalized Gronwall Lemma \cite{Bartels2011,Kwon2023}, those based on local-in-time continuation arguments \cite{Kyza2016} and those based on fixed point arguments \cite{Kyza2011,Kyza2020}.
%See citations on Kyza paper

\noindent The literature on numerical methods for the Keller-Segel system was, until now, focused on the development of efficient numerical schemes and \emph{a priori} error estimates. 
Concerning the development of numerical schemes, we would like to mention finite element (FE) schemes derived in \cite{Saito2007,Strehl2013}. 
Quite recently, an \emph{a priori} error analysis and an analysis of the blow-up behavior for a FE method was provided in \cite{Chen2022}.
A local discontinuous Galerkin (dG) method was considered in \cite{Li2017} and proven to be energy dissipative in \cite{Guo2019}. 
An \emph{a priori} convergence analysis for a semi-discrete family of interior penalty dG methods was performed in \cite{Epshteyn2009}. Fully discrete versions of the schemes in this family were analyzed  in \cite{Epshteyn2009fully}.
Finite volume (FV) schemes were studied in \cite{Chertock2008,Filbet2006}. While \cite{Chertock2008} presents many numerical experiments regarding the blow-up behavior, \cite{Filbet2006} provides an \emph{a priori} convergence analysis provided the initial data are sufficiently small.
Furthermore, there exist heuristic adaptive mesh refinement strategies applied to chemotactic evolution equations. 
For example, in \cite{Dudley2011} dimensionless quantities such as the Damköhler numbers for growth and decay and the Peclet number are utilized to identify regions, where the mesh is subsequently refined. 
A moving mesh FE method was employed in \cite{Sulman2019}.
There the mesh is chosen adaptively by detecting steep gradients of the numerical solution and using the Monge-Ampère method to redistribute a fixed number of mesh nodes, such that they are concentrated in regions of large solution variations.

\noindent Conditional \emph{a posteriori} error estimates have been derived by one of the authors in \cite{Kolbe2023} for a FV-FE method restricted to Cartesian meshes, and in \cite{Kwon2023} for a dG scheme.
All these \emph{a posteriori} estimates can be used to drive mesh and time step adaptation.

\noindent A serious concern for conditional \emph{a posteriori} error estimates is, obviously, whether one can expect the condition to be satisfied in relevant scenarios. 
In \cite{Kyza2020} it was argued that conditions of local nature are satisfied more often than global ones. 
We could not find any evidence for this claim in the literature and decided to derive two different \emph{a posteriori} error estimates, one using a Generalized Gronwall Lemma, see Proposition \ref{adj_gengronwall}, and one using a local-in-time continuation argument as in \cite{Kyza2016}.
As predicted in \cite{Kyza2020}, the local conditions are satisfied far more often, see Section~\ref{num_sim}.

\noindent The conditions in both our results, Theorem~\ref{stab_thm} and Theorem~\ref{stab_cont}, satisfy the minimal requirement that the results hold for sufficiently fine meshes if the scheme converges, 
however with our computational resources, the resulting time horizons of availability of the stability estimates are rather short, see Subsection~\ref{subsec:availability}.
There are several ideas for overcoming this in future research. 
One is to use more memory and compute power, one is to derive higher order convergent \emph{a posteriori} error estimates for higher order schemes (which seems feasible although quite technical), and finally one could try to improve our stability results.

\noindent The remainder of this paper is structured as follows: 
In Section~\ref{section-KS}, the parabolic-elliptic Keller-Segel system is introduced and its most important features are stated. Additionally, the \emph{a posteriori} verifiable existence argument is explained in detail.
Then, Section~\ref{stability} is concerned with the derivation of two distinct stability frameworks, one based on a Generalized Gronwall Lemma and one using a local-in-time continuation argument.
Further, the numerical method used to approximate the Keller-Segel system is explained in Section~\ref{chapter-Num-method}.
After that, Section~\ref{Morley} contains a suitable reconstruction, in our case a \emph{Morley-type} interpolation. 
Furthermore, the \emph{a posteriori} analysis of the Keller-Segel model can be found in Section~\ref{aposterirorierror}. 
This includes the derivation of residual estimates in Subsection~\ref{residual_esti} and the construction of a conditional \emph{a posteriori} error estimator in Subsection~\ref{aposterirorierroresti}, in which the central result Theorem~\ref{main_thm} is stated.
Lastly, in Section~\ref{num_sim} numerical simulations are presented: One is concerned with the convergence rate of the \emph{a posteriori} residual estimator for $d \in \{2,3\}$.
A second numerical test compares the time horizons of rigorous availability, i.e. in which the respective condition holds true, of the two \emph{a posteriori} error estimators derived in Subsection~\ref{aposterirorierroresti}, for $d\in \{2,3\}$.

%% file: kellersegelsystem.tex
\section{The Keller-Segel system}
\label{section-KS}

The Keller-Segel system is a paradigmatic model in mathematical biology that describes chemotaxis, see \cite{Keller1970}.
A striking feature of this model is that, under certain circumstances, solutions blow up in finite time in two or more space dimensions, see \cite{Nagai1995}.

\noindent As the spatial domain, we consider the flat torus $\mathbb{T}^d := \R^d / \mathbb{Z}^d$ with $d \in \left\{2, 3\right\}$. Any strong solution to 
$\bar{c} - \Delta \bar{c} = \bar{\rho}$ in $\mathbb{T}^d$ enjoys elliptic regularity, i.e.
\begin{align}
    \label{ell_reg}
    \norm{\bar{c}}_{H^{k+2}(\mathbb{T}^d)} \leq C_\text{ell} \norm{\bar{ \rho}}_{H^k(\mathbb{T}^d)},
\end{align}
whenever $\bar{ \rho} \in H^k(\mathbb{T}^d)$, for $k \geq 0$ and some constant $C_\text{ell} >0$. 
This is due to the operator $(I - \Delta)^{-1}$ being an isomorphism $H^k(\mathbb{T}^d) \to H^{k+2}(\mathbb{T}^d)$. 
By comparing Fourier multipliers, we get $C_\text{ell}=1$ for all $k\geq0$ on the flat torus $\mathbb{T}^d$.

\noindent To abbreviate notation, we write ${L}^p, H^k$, and $W^{k,p}$ instead of $L^p(\mathbb{T}^d), H^k(\mathbb{T}^d)$, and $W^{k,p}(\mathbb{T}^d)$, respectively, and use $|\cdot|_{H^{k}}$ and $|\cdot|_{W^{k,p}}$ to denote the usual Sobolev seminorms. 
We denote by $(H^1(\mathbb{T}^d))^\prime$ the dual space of $H^1(\mathbb{T}^d)$.

\noindent Let a time interval $[0,T]$ for some $T>0$ be given.
The dynamics of the bacterial density $\rho$ is governed by chemotaxis and diffusion. 
The parabolic-elliptic version of the Keller-Segel system reads as
\begin{align}
    \partial_t \rho + \divergence\left( \rho \nabla c \right) - \Delta \rho &= 0 \quad \text{ in } (0,T) \times \mathbb{T}^d, \label{KS1}\tag{KS.1} \\
    c - \Delta c &= \rho \quad \text{ in } (0,T) \times \mathbb{T}^d, \label{KS2}\tag{KS.2} \\
    \rho(0,\cdot) &= \rho_0 \ \ \  \text{ in } \mathbb{T}^d, \label{KScond}\tag{KS.3}
\end{align}
for an initial condition $\rho_0 \in L^2$.
% \begin{align}
%     \label{KScond}\tag{KS.3}
%     \begin{cases}
%         \ \nabla \rho \cdot n &= 0 \quad \text{ on } (0,T) \times \partial\Omega, \\ %\label{KS.bdr.1}\\
%         \ \nabla c \cdot n &= 0 \quad \text{ on } (0,T) \times \partial\Omega, \\ %\label{KS.bdr.2}\\
%         \ \rho(0,\cdot) &= \rho_0 \ \ \  \text{ in } \Omega, %\label{KS.init.1}
%     \end{cases}
% \end{align}
% where $n$ denotes the unit outward normal vector along $\partial \Omega$.

% \begin{Remark}
%     The a posteriori analysis that follows is equally applicable for periodic boundary conditions as well as mixed boundary conditions, where $\partial \Omega$ is decomposed into two parts $\Gamma_P$ and $\Gamma_N$, i.e. $\Gamma_P \stackrel{.}{\cup} \Gamma_N = \partial \Omega$. 
%     Then, a homogeneous Neumann boundary condition is imposed on $\Gamma_N$ and a periodic boundary condition on $\Gamma_P$, where $\Gamma_P$ is appropriately chosen to grant well-posedness of our initial-boundary value problem.
%     This observation is important for our numerical simulations, which are performed for mixed boundary conditions due to difficulties in decomposing the unit cube into a well-centered tetrahedral mesh, i.e. a mesh in which the circumcenter of each element lies within the respective element, see \cite{Hirani2008} and \autoref{num_sim}. 
% \end{Remark}
\ \\
\noindent For the definition of weak solutions and their well-posedness we refer to \cite{Biler1993}. The following well-posedness result applies. 
\newpage
\begin{Theorem}[Existence and regularity of weak solutions]
    \label{weakexistence}
    Let $d \in \{2,3\}$. Then
        \begin{enumerate}[i)]
            \item If $0 \leq \rho_0 \in L^2$, then there exists $T = T(\norm{\rho_0}_{L^2}) > 0$ such that the system (\ref{KS1}), (\ref{KS2}), (\ref{KScond}) has a unique weak solution $\rho \in L^\infty(0,T;L^2)\cap L^2(0,T;H^1)$.
            Moreover, $\partial_t \rho \in L^2(0,T;(H^1(\mathbb{T}^d))^\prime)$, $\rho(t,x) \geq 0$ for almost every $x \in \mathbb{T}^d$ and $t \geq 0$, and $\int_{\mathbb{T}^d} \rho(t,\cdot) \ dx = \int_{\mathbb{T}^d} \rho_0 \ dx$.
            \item If $0 \leq \rho_0 \in L^p$, $p>\frac{d}{2}$, then there exists $T = T(p,\norm{\rho_0}_{L^p}) > 0$ and a weak solution $\rho$ such that $\rho \in L^\infty(0,T;L^p)$ and $\rho^{p/2} \in L^2(0,T;H^1)$. 
            This solution is regular, i.e. $\rho \in L^\infty_{\text{loc}}(0,T;L^\infty)$, and if additionally $p > d$, it is unique.
        \end{enumerate}
\end{Theorem}
\begin{proof}
    The estimates to invoke the fixed point argument as done in \cite{Biler1992} on bounded domains with $C^{1+\delta}$-boundary, for some $\delta > 0$, equipped with nonlinear no-flux boundary conditions, work analogously on the flat torus $\mathbb{T}^d$, whenever $d \in \{2,3\}$.
\end{proof}
\begin{Remark}
    \label{blowup-remark}
    Statement $ii)$ means that blow-up of the norm $\norm{\rho(t,\cdot)}_{L^\infty}$ at time $T>0$ implies blow-up of the norms $\norm{\rho(t,\cdot)}_{L^p}$ for all $p \in (\frac{d}{2},\infty]$, at the same moment $T$.
\end{Remark}

\noindent Additionally, one can show that $\rho$ is continuous in time \cite{Biler1998}, which justifies considering an initial value problem.
Another important insight regarding the blow-up time of weak solutions of the Keller-Segel system is its identification.
\begin{Lemma}[Blow-up criterion for weak solutions, \protect{\cite[Lemma 2.5]{Kwon2023}}]
    \label{LemmaKwon}
    Let $\rho_0 \in L^2$ and let $T_{\text{max}} \in (0,\infty]$ be the maximal existence time of the weak solution $\rho$ to (\ref{KS1}), (\ref{KS2}), (\ref{KScond}). 
    If $T_{\text{max}} < \infty$, we have
    \begin{align*}
        \lim_{t \nearrow T_{\text{max}}} \norm{\rho(t,\cdot)}_{L^\infty} = \infty.
    \end{align*}
\end{Lemma}

\noindent These observations together with some finite \emph{a posteriori} error estimator can be exploited to infer \emph{a posteriori} verifiable existence of weak solutions. 
This means for an explicit time $T>0$, under certain conditions, which will be specified below, we can rigorously show the existence of a weak solution on $[0,T]$ via numerical simulations.
This strengthens the \emph{a posteriori} error estimates derived in Section~\ref{aposterirorierror}, as the existence of weak solutions on a time interval $[0,T]$ is a necessary assumption in many stability estimates leading to \emph{a posteriori} error estimates, see Section~\ref{stability}.
Whereas before, one could only hope for a weak solution to exist on $[0,T]$, we can now verify this in an \emph{a posteriori} manner.
The statement reads as
\begin{Proposition}[A posteriori verifiable existence]
    \label{apostverif}
    Let $d \in \{2,3\}$.
    Let $0 \leq \rho_0 \in L^2$ and let $T_{\text{max}}>0$ be the maximal existence time of the weak solution $(\rho,c)$ to (\ref{KS1}), (\ref{KS2}), (\ref{KScond}).
    Let $\tilde{\rho} \in L^\infty(0,T;L^2)$ be a reconstruction of a numerical approximation to $\rho$.
    Suppose that for some $T>0$, one has $\| \tilde{\rho}(t,\cdot)\|_{L^\infty(0,T;L^2)} < \infty$ and suppose that there exists an error estimator $\theta < \infty$ such that%the condition (\ref{stab_cond}) is satisfied, $\norm{\bar{\rho}}_{L^\infty(0,T;L^2(\Omega))}$ is finite and the right-hand side of the stability estimate in Theorem~\autoref{stab_thm} is finite.
    \begin{align}
        \label{unif_bound}
        \sup_{t \in [0,\min\{T,T_{\text{max}}\}]} \| \rho(t,\cdot) - \tilde{\rho}(t,\cdot)\|_{L^2}^2 \leq \theta.
    \end{align}
    Then $T < T_{\text{max}}$, i.e. the weak solution exists beyond time $T$.
\end{Proposition}
\begin{proof}
    The proof is similar to the proof of \cite[Corollary 3.5]{Kwon2023}, but more general. 
    %The error estimator $\theta_\Omega$ is finite and bounds the error in the time interval $[0,\min\{T,T_{\text{max}}\}>0]$. 
    Suppose $T_{\text{max}} \leq T$. 
    Then, via triangle inequality, (\ref{unif_bound}) yields a uniform-in-time upper bound for $\| \rho(t,\cdot)\|_{L^2}$, as $\| \tilde{\rho}(t,\cdot)\|_{L^2}$ is finite and uniformly bounded for almost all $t \in [0,T_{\text{max}}]$.
    In contrast, Lemma~\ref{LemmaKwon} together with Remark~\ref{blowup-remark} imply that $\| \rho(t,\cdot)\|_{L^2} \rightarrow \infty$ for $t \nearrow T_{\text{max}}$ due to $d\leq3$. 
    This is a contradiction and hence  $T < T_{\text{max}}$.
\end{proof}
%\noindent If a solution exists beyond time $T>0$, we especially know that it did not blow-up until this moment. 
\noindent The required error estimator $\theta$ is constructed in Section~\ref{aposterirorierror}.

%% file: stability.tex
\section{Stability results}
    \label{stability}
    %This is the main section of this paper. We aim at deriving fully computable a posteriori error estimates for the Keller-Segel system (\ref{KS1}) and (\ref{KS2}) equipped with (\ref{KS.bdr.1}) - (\ref{KS.init.1}). 
    
    In order to derive \emph{a posteriori} error estimates, we introduce a reconstruction of the numerical approximation, which we then compare once to the numerical approximation and once to the exact weak solution, bounding the error using the triangle inequality.
    This approach is abstractly explained in \cite{Makridakis2007} and applied to linear parabolic equations in \cite{Makridakis2003}.

    \noindent In this context, a stability framework is used to bound the difference of a weak solution of the Keller-Segel system and a solution of a corresponding perturbed system, i.e. (\ref{res_bdrKS}), by only the initial datum, the perturbation of the right-hand side and terms, for which we can find computable upper bounds.
    %Then \emph{a posteriori} residual estimates are required to bound the error of the reconstruction. 

    \noindent Let $(\rho,c)$ denote the weak solution to (\ref{KS1}), (\ref{KS2}), (\ref{KScond}). 
    Further, let $(\bar{\rho},\bar{c})$ with $\bar{\rho} \in C(0,T;H^1)$, $\partial_t \bar{\rho} \in L^2(0,T;(H^1(\mathbb{T}^d))^\prime)$ and $\bar{c} \in C(0,T;H^3)$ be a solution to the perturbed system
    \begin{align}
        \label{res_bdrKS}
        \begin{cases}
            \ \ \ \partial_t \bar{\rho} + \divergence\left( \bar{\rho} \nabla \bar{c} \right) - \Delta \bar{\rho} &= R_{\bar{\rho}} \ \text{ in } (0,T) \times \mathbb{T}^d, \\
            \quad \quad \quad \quad \quad \quad \quad \bar{c} - \Delta \bar{c} &= \bar{\rho} \quad \text{ in } (0,T) \times \mathbb{T}^d,
        \end{cases}
    \end{align}
    for some (residual) function $R_{\bar{\rho}} \in L^2(0,T;(H^1(\mathbb{T}^d))^\prime)$.
    Note that the additional regularity $\bar{c}(t,\cdot) \in H^3(\mathbb{T}^d)$ immediately follows from $\bar{\rho}(t,\cdot) \in H^1(\mathbb{T}^d)$, for all $t \in [0,T]$, due to elliptic regularity (\ref{ell_reg}).
    This ensures $\bar{c}(t,\cdot) \in W^{1,\infty}(\mathbb{T}^d)$, for all $t \in [0,T]$, which we will exploit in the sequel.
    %We assume throughout the whole paper that $\Omega$ is such that $\bar{c}$ enjoys elliptic regularity, i.e. $\norm{\bar{c}(t,\cdot)}_{H^2} \leq C_{\text{ell}} \norm{\bar{\rho}(t,\cdot)}_{L^2}$, for some constant $C_{\text{ell}} > 0$, which can also be estimated explicitly \cite{Bebendorf2003,Grisvard2011}.
    Later, a suitable reconstruction of the numerical solution will play the role of the solution $(\bar{\rho},\bar{c})$ and $R_{\bar{\rho}}$ will be defined based on $(\bar{\rho},\bar{c})$.
    We consider two different stability approaches.
    
    % For our purposes a suitable reconstruction for the bacterial density is the interpolation of \emph{Morley-type}.
    % As in (\ref{res_bdrKS}) only the first equation is perturbed, we need to reconstruct the chemical density elliptically.

    \subsection{Stability via a Generalized Gronwall Lemma}
    \label{gen_gronwall}

    One possible stability framework is proposed in \cite{Kolbe2023,Kwon2023}. 
    There, the Generalized Gronwall Lemma \cite[Proposition 6.2]{Bartels2015nonlinear} is prominently used. 
    We use an adjusted version thereof, which reads as
    \begin{Proposition}[Adjusted Generalized Gronwall Lemma]
        \label{adj_gengronwall}
        Suppose that the non-negative functions $y \in C([0, T]), a \in L^{\infty}([0, T])$, the real numbers $A, B_1, B_2 \geq 0$ and $\beta_1 , \beta_2>0$ satisfy
        \begin{align}
            \label{nonlin_term}
            y\left(T^{\prime}\right) \leq A+\int_0^{T^{\prime}} a(t) y(t)\ dt+\sup _{t \in\left[0, T^{\prime}\right]} \left(B_1 y^{\beta_1}(t) + B_2 y^{\beta_2}(t) \right)\int_0^{T^{\prime}}y(t) \ d t
        \end{align}
        for all $T^{\prime} \in[0, T]$. 
        Set $E:=\exp \left(\int_0^T a(t) \ d t\right)$ and suppose there exists a $\delta > 1$ such that the condition 
        \begin{align}
            \label{cond_prop31}
            B_1 \left(\delta A E\right)^{\beta_1} + B_2 \left(\delta A E\right)^{\beta_2} < \frac{\delta -1}{\delta T E}
        \end{align}
        is satisfied. We then have
        \begin{align*}
           \sup _{t \in[0, T]} y(t) \leq \delta A E .
        \end{align*}
    \end{Proposition}
    \begin{proof}
            We assume first that $A>0$, set $\theta:=\delta A E$, and define
            \begin{align*}
                \mathcal{I}_\theta:=\left\{T^{\prime} \in[0, T]: \Upsilon\left(T^{\prime}\right):=\sup _{t \in\left[0, T^{\prime}\right]} y(t) \leq \theta\right\}.
            \end{align*}
            Since $y(0) \leq A<\theta$ and since $\Upsilon$ is continuous and increasing, we have $\mathcal{I}_\theta=$ $\left[0, T_M\right]$ for some $0<T_M \leq T$. For every $T^{\prime} \in\left[0, T_M\right]$ we have
            \begin{align*}
                y\left(T^{\prime}\right) & \leq A+\int_0^{T^{\prime}} a(t) y(t)\ d t+\sup _{t \in\left[0, T^{\prime}\right]}\left(B_1 y^{\beta_1}(t) + B_2 y^{\beta_2}(t) \right) \int_0^{T^{\prime}} y(t)\ dt \\
                & \leq A+\int_0^{T^{\prime}} a(t) y(t) \ d t+ T \left(B_1 \theta^{1+\beta_1} +  B_2\theta^{1+\beta_2}\right) 
            \end{align*}
            An application of the classical Gronwall Lemma, the condition (\ref{cond_prop31}), and the choice of $\theta$ yield that for all $T^{\prime} \in\left[0, T_M\right]$, we have
            \begin{align*}
            y\left(T^{\prime}\right)\leq\left(A+ T \left(B_1 \theta^{1+\beta_1} +  B_2\theta^{1+\beta_2}\right)\right) E < \theta .
            \end{align*}
            This implies $\Upsilon\left(T_M\right)<\theta$, hence $T_M=T$, and thus proves the Lemma if $A>0$. If $A=0$, then the above argument holds for every $\theta>0$, and we deduce that $y(t)=0$ for all $t \in[0, T]$.
    \end{proof}
    \begin{Remark}
        
        The main advantage of the adjusted version, in addition to the fact that it is exactly tailored for the type of estimate that will result in the following, is that we can replace the factor $(1+T)$ in the condition of \cite[Proposition 6.2]{Bartels2015nonlinear} by the factor $T$.
        Similar results also work for more general functions in $y$ other than $B_1 y^{\beta_1} + B_2 y^{\beta_2}$, e.g. all polynomials.
        
        % In comparison to \cite[Proposition 6.2]{Bartels2015nonlinear}, we have made three adjustments: 
        % Firstly, by bounding the nonlinear term on the right-hand side of (\ref{nonlin_term}) only in terms of $y_1(t)$, in place of $y_1(t)$ and $y_2(t)$, we replace the factor $(1+T)$ in the condition by the factor $T$.
        % Secondly, choosing $\delta = 8$, we see that we improved the upper bound of the condition at least by a factor of $7^{1/\beta}$. 
        % Additionally, other choices of $\delta$ might lead to even (slightly) weaker conditions. 
    \end{Remark}
    
    \noindent Let $C_{\text{S}}$ denote the constant of the Sobolev embedding $H^1(\mathbb{T}^d) \hookrightarrow L^6(\mathbb{T}^d)$.
    Let $(\rho,c)$ and $(\bar{\rho},\bar{c})$ be as above. 
    In order to apply Proposition \ref{adj_gengronwall} to the Keller-Segel system, we perform energy estimates similar to \cite[Section 3]{Kwon2023}, that is subtracting the first line in (\ref{res_bdrKS}) from (\ref{KS1}), testing with $e := (\rho - \bar{\rho})$ and using the Cauchy-Schwarz inequality. 
    We obtain
    \begin{align*}
        \frac{d}{dt} \bigg(\frac12 \| e \|_{L^2}^2\bigg) + |e|_{H^1}^2 
        \leq \underbrace{|e|_{H^1} \| \bar{\rho} \|_{L^3} |c-\bar{c}|_{W^{1,6}} + |e|_{H^1} \| e \|_{L^2} \| \nabla \bar{c}\|_{L^\infty} + \| R_{\bar{\rho}} \|_{(H^1(\mathbb{T}^d))^\prime} \| e \|_{H^1}}_{=:\  I} \quad &\\
            + \underbrace{|e|_{H^1}\| e \|_{L^3} |c-\bar{c}|_{W^{1,6}}}_{=: \ II} &.
    \end{align*}
    We treat $I$ similar to \cite[Section 3]{Kwon2023}, only slightly adjusting the constants in Young's inequality, which gives
    \begin{align*}
        I \leq (2C_{\text{S}}^2 C_{ell}^2 \| \bar{\rho} \|_{L^3}^2 + 2\| \nabla \bar{c}\|_{L^\infty}^2 + \frac{1}{16})\|e\|_{L^2}^2 + 6\| R_{\bar{\rho}} \|_{(H^1(\mathbb{T}^d))^\prime}^2 + \frac38 |e|_{H^1}^2,
    \end{align*}

    \noindent In contrast to \cite[Section 3]{Kwon2023}, we further estimate $II$ by
    \begin{align*}
        |e|_{H^1}\| e \|_{L^3} |c-\bar{c}|_{W^{1,6}} &\leq C_{\text{S}} C_{\text{ell}} |e|_{H^1}\| e \|_{L^3} \|e\|_{L^2} \\
        &\leq C_{\text{S}}^{3/2} C_{\text{ell}} |e|_{H^1}\|e\|_{H^1}^{1/2}\|e\|_{L^2}^{3/2} 
         \ = C_{\text{S}}^{3/2} C_{\text{ell}} |e|_{H^1}(\|e\|_{L^2}^2+|e|_{H^1}^2)^{1/4}\|e\|_{L^2}^{3/2} \\
         & \leq  C_{\text{S}}^{3/2} C_{\text{ell}} \left( \|e\|_{L^2}^2 |e|_{H^1} + \|e\|_{L^2}^{3/2} |e|_{H^1}^{3/2} \right)\\
         & \leq   C_{\text{S}}^3 C_{\text{ell}}^2 \|e\|_{L^2}^4 + \frac{1}{4}|e|_{H^1}^2 + 2C_{\text{S}}^6 C_{\text{ell}}^4\|e\|_{L^2}^{6} + \frac{3}{8} |e|_{H^1}^{2} \\
         &=  C_{\text{S}}^3 C_\text{ell}^2\|e\|_{L^2}^4 + 2 C_{\text{S}}^6 C_{\text{ell}}^4\|e\|_{L^2}^6 + \frac58 |e|_{H^1}^2,
        % &\leq  \frac45 C_{\text{S}}^3 C_\text{ell}^2\|e\|_{L^2}^4 + \frac{432}{125} C_{\text{S}}^6 C_{\text{ell}}^4\|e\|_{L^2}^6  + \frac58 |e|_{H^1}^2,
    \end{align*}
    where we used the Sobolev embedding $H^1 \hookrightarrow L^6$ and elliptic regularity in the first step, Hölder's inequality and again the Sobolev embedding $H^1 \hookrightarrow L^6$ in the second step and twice Young's inequality, once with exponents both equal to two and once with exponents $p= 4$ and $q = \frac{4}{3}$, in the fifth step.
    Combining the estimates of $I$ and $II$, we obtain in total
    \begin{align}
        \label{energy_esti}
        \frac{d}{dt} \| e \|_{L^2}^2 &\leq (4 C_{\text{S}}^2 C_{ell}^2 \| \bar{\rho} \|_{L^3}^2 + 4\| \nabla \bar{c}\|_{L^\infty}^2 + B_1 \|e\|_{L^2}^2 + B_2\|e\|_{L^2}^4 + \frac18 )\|e\|_{L^2}^2 + 12\| R_{\bar{\rho}} \|_{(H^1(\mathbb{T}^d))^\prime}^2,
    \end{align}
    with $B_1 := 2 C_{\text{S}}^3 C_\text{ell}^2$ and $B_2 := 4C_{\text{S}}^6 C_{\text{ell}}^4$.
    Integrating in time and setting
    \begin{align*}
        y(t) & :=\|\rho(t, \cdot)-\bar{\rho}(t, \cdot)\|_{L^2}^2,\quad A := \norm{\rho(0,\cdot) - \bar{\rho}(0,\cdot)}_{L^2}^2 + 12\int_0^T \norm{R_{\bar{\rho}}(t,\cdot)}^2_{(H^1(\mathbb{T}^d))^\prime} \ dt \\
        a(t) & :=4 C_{\text{S}}^2 C_{\text {ell }}^2\|\bar{\rho}(t, \cdot)\|_{L^3}^2+4\|\nabla \bar{c}(t, \cdot)\|_{L^{\infty}}^2+\frac{1}{8}, \quad E:=\exp\left( \int_0^T a(t)\ dt \right)
        ,
    \end{align*}
    as well as $\beta_1 := 1$, $\beta_2 := 2$ and $B_1$, $B_2$ as above, we apply Proposition \ref{adj_gengronwall} to get
    \begin{Theorem}
        \label{stab_thm}
        Let $d \in \left\{2,3\right\}$. Suppose $(\rho,c)$ is a weak solution of the Keller-Segel system (\ref{KS1}), (\ref{KS2}), (\ref{KScond}) and $(\bar{\rho},\bar{c})$ be a solution to (\ref{res_bdrKS}), as above.
        Then, for any $\delta>1$, such that
        \begin{align}
            \label{stab_cond}
            B_1 \delta A E + B_2 \left(\delta A E\right)^{2} < \frac{\delta -1}{\delta T E}
        \end{align}
        is satisfied, it holds
        \begin{align*}
            &\sup_{t \in [0,T]}\norm{\rho(t,\cdot) - \bar{\rho}(t,\cdot)}_{L^2}^2 \leq \delta AE. % \\
            % &\quad\quad\quad\quad\quad\quad\quad\quad\quad\quad\leq 8 \left( \frac{1}{2} \norm{\rho(0,\cdot) - \bar{\rho}(0,\cdot)}_{L^2(\Omega)}^2 + \int_0^T \norm{R_{\bar{\rho}}}^2_{(H^1(\mathbb{T}^d))^\prime(\Omega)} dt \right) \times \\
            % & \quad\quad\quad\quad\quad\quad\quad\quad\quad\quad\quad\quad \exp\left( \int_0^T 3 C_{\text{S}}^2 C_{\text{ell}}^2 \norm{\bar{\rho}(t,\cdot)}_{L^3(\Omega)}^2 + 3\norm{\nabla \bar{c}(t,\cdot)}^2_{L^\infty(\Omega)}+\frac{1}{3} \ dt \right).
        \end{align*}
    \end{Theorem}

    \noindent By slightly modifying the arguments, analogous to \cite{Kolbe2023,Kwon2023}, we could use the generalized Gronwall Lemma to control the $L^2(0,T;H^1)$-norm in addition to the $L^\infty(0,T;L^2)$-norm. This would involve a slight worsening of the constants. 
    For our purpose of a posteriori existence, control of the $L^\infty(0,T;L^2)$-norm is sufficient, see Proposition \ref{apostverif}.
    
    % Note that for $d=2$, the term $\norm{\nabla \bar{c}(t,\cdot)}^2_{L^\infty(\Omega)}$ can be bounded up to a constant by $\norm{\bar{\rho}(t,\cdot)}^2_{L^{3}(\Omega)}$, due to elliptic regularity and Sobolev embedding $H^2(\Omega) \hookrightarrow L^\infty(\Omega)$ with constant $C_{\text{S}}^{\prime \prime}$, i.e.
    % \begin{align*}
    %     \norm{\nabla \bar{c}(t,\cdot)}_{L^\infty(\Omega)} \leq C_{\text{S}}^{\prime \prime} \norm{\nabla \bar{c}(t,\cdot)}_{H^2(\Omega)} \leq C_{\text{S}}^{\prime \prime} C_{\text{ell}} \norm{\nabla \bar{\rho}(t,\cdot)}_{L^2(\Omega)} \leq  C_{\text{S}}^{\prime \prime} C_{\text{ell}} C_{\text{S}}^{\prime} \norm{\bar{\rho}(t,\cdot)}_{L^3(\Omega)}.
    % \end{align*}

    \begin{Remark} 
        \label{rem:gen}
        Due to the residual estimates in the $(H^1(\mathbb{T}^d))^\prime$-norm, that will follow in Section~\ref{aposterirorierror}, and their scaling behavior, see Section~\ref{num_sim}, we anticipate the factor $A$ to be small, whenever $(\bar{\rho},\bar{c})$ results from a convergent numerical scheme, and to go to zero under mesh refinement.
        Consequently, the condition (\ref{stab_cond}) can be expected to hold true for sufficiently fine meshes. 
        Furthermore, small time-intervals $[0,T]$ make it easier to satisfy condition (\ref{stab_cond}). 
        % However, the large constant $C_{\text{S}}$ especially for $d=3$, pose a challenge in practice, see Remark \ref{bad_constants}.
    \end{Remark}

    \noindent We can formulate an \emph{a posteriori} verifiable existence statement as a direct consequence of Proposition~\ref{apostverif}, similar to \cite[Corollary 3.5]{Kwon2023}.
        \begin{Corollary}
            Let $0 \leq \rho_0 \in L^2$ and let $T_{\text{max}}>0$ be the maximal existence time of the weak solution $(\rho,c)$ to (\ref{KS1}), (\ref{KS2}), (\ref{KScond}).
            Let $\tilde{\rho} \in L^\infty(0,T;L^2) \cap L^2(0,T;H^1)$ be a reconstruction of a numerical approximation to $\rho$.
            Suppose that for some $T>0$, one has $\| \tilde{\rho}(t,\cdot)\|_{L^\infty(0,T;L^2)} < \infty$. Suppose further that (\ref{stab_cond}) is satisfied and that $A$ as well as $E$ are finite.
            Then $T < T_{\text{max}}$, i.e. the weak solution exists beyond time $T$.
        \end{Corollary}

    % The statement reads as
    % \begin{Proposition}[A posteriori verifiable regularity, see Corollary 3.4 in \cite{Kwon2023}]
    %     \label{apostverif}
    %     Let $\rho_0 \in L^2(\Omega)$ and let $T_{\text{max}}>0$ be the maximal existence time of the weak solution $(\rho,c)$ to (\ref{KS1}) and (\ref{KS2}) equipped with (\ref{KS.bdr.1}) - (\ref{KS.init.1}).
    %     Suppose that there exists an approximate strong solution $(\bar{\rho},\bar{c})$ of (\ref{res_bdrKS}) such that at some time $T>0$ the condition (\ref{stab_cond}) is satisfied, $\norm{\bar{\rho}}_{L^\infty(0,T;L^2(\Omega))}$ is finite and the right-hand side of the stability estimate in Theorem~\autoref{stab_thm} is finite.
    %     Then $T_{\text{max}}>T$, i.e. the weak solution exists at least until time $T$.
    % \end{Proposition}
    % \noindent This means that given a reconstruction $\bar{\rho}$ of our numerical solution and given a posteriori residual bounds such that the conditions of Proposition~\autoref{apostverif} hold true up to some time $T$, we can verify a posteriori that the weak solution exists at least up to the time $T$. 
    % Especially, the solution did not blow up until this moment.

    \subsection{Stability through a local-in-time continuation argument}
    \label{cont_arg}
        This stability framework is inspired by the local-in-time continuation argument used in \cite{Kyza2016}.
        As a time discretization, we choose the potentially non-uniform time steps $t^n \in [0,T]$ for $n = 0, 1, \dots, N_t$ with the time step sizes $\Delta t^n := t^{n+1}-t^n \geq 0$ such that $\sum_{n = 0}^{N_t-1} \Delta t^n = T$. 
        
        % \noindent Let $(\rho,c)$ and $(\bar{\rho},\bar{c})$ be as in Theorem~\ref{stab_thm}. 
        \begin{Theorem}[Stability through local-in-time continuation]
            \label{stab_cont}
            Let $d \in \left\{2,3\right\}$. Suppose $(\rho,c)$ is a weak solution of the Keller-Segel system (\ref{KS1}), (\ref{KS2}), (\ref{KScond}) and $(\bar{\rho},\bar{c})$ be a solution of (\ref{res_bdrKS}), as above.
            Then, provided for each $n = 0, 1, \ldots, N_t-1$, there is a $\delta_n > 1$ such that
            \begin{align}
                \label{root}
                \Delta t^n \left( B_1 \delta_n \tilde{A}_n E_n + B_2 \delta_n^2 \tilde{A}_n^2 E_n^2\right) - \log(\delta_n) < 0,
            \end{align}
            it holds
            \begin{align*}
                &\sup_{t \in [0,T]}\norm{\rho(t,\cdot) - \bar{\rho}(t,\cdot)}_{L^2}^2  \leq \delta_{N_t -1} \tilde{A}_{N_t -1} E_{N_t -1},
            \end{align*}
            where 
            \begin{align}
                \label{iter_def}
                E_n &:= \exp\bigg(\int_{t^n}^{t^{n+1}} 4 C_{\text{S}}^2 C_{ell}^2 \| \bar{\rho}(s,\cdot) \|_{L^3}^2 + 4\| \nabla \bar{c}(s,\cdot)\|_{L^\infty}^2 + \frac18 \  ds \bigg) \ , \notag \\
                \tilde{A}_0 &:= \|\rho(0,\cdot) - \bar{\rho}(0,\cdot)\|_{L^2}^2 + 12\int_{0}^{t^{1}} \| R_{\bar{\rho}}(s,\cdot) \|_{(H^1(\mathbb{T}^d))^\prime}^2 \ ds, \notag \\
                \tilde{A}_{n+1} &:= \delta_n \tilde{A}_n E_n + 12\int_{t^{n+1}}^{t^{n+2}} \| R_{\bar{\rho}}(s,\cdot) \|_{(H^1(\mathbb{T}^d))^\prime}^2 \ ds.
            \end{align}
        \end{Theorem}
        \begin{proof}
            Integrating (\ref{energy_esti}) locally in time and applying the classical Gronwall Lemma leads to
            \begin{align}
                \label{noncomp_esti}
                \| e(t,\cdot)\|_{L^2}^2 \leq A_n E_n \mathcal{H}_n(t),
            \end{align}
            for any $t \in [t^n,t^{n+1}]$, where
            \begin{align*}
            A_n &:= \|e(t^n,\cdot)\|_{L^2}^2 + 12\int_{t^n}^{t^{n+1}} \| R_{\bar{\rho}}(s,\cdot) \|_{(H^1(\mathbb{T}^d))^\prime}^2 \ ds \ , \\
            % E_n &:= \exp\bigg(\int_{t^n}^{t^{n+1}} 4 C_{\text{S}}^2 C_{ell}^2 \| \bar{\rho}(s,\cdot) \|_{L^3}^2 + 4\| \nabla \bar{c}(s,\cdot)\|_{L^\infty}^2 + \frac18 \  ds \bigg) \ ,\\
            \mathcal{H}_n(t) &:= \exp\bigg(  \int_{t^n}^{t} B_1\|e(s,\cdot)\|_{L^2}^2 + B_2\|e(s,\cdot)\|_{L^2}^4 \ ds \bigg).
        \end{align*}
        In a next step, we bound $\mathcal{H}_n(t)$ using the following continuation argument.
        We define
        \begin{align*}
            \mathcal{I}_n := \left\{t \in [t^n,t^{n+1}] : \sup_{s \in [t^n,t]}\norm{e(s,\cdot)}_{L^2}^2 \leq \delta_n A_n E_n \right\},
        \end{align*}
        for some constant $\delta_n > 1$, which will later be chosen as close to $1$ as possible.
        Note that $ t^n \in \mathcal{I}_n$, hence $\mathcal{I}_n$ is non-empty and due to its compactness attains a maximal value. 
        We want to show by contradiction that $\mathcal{I}_n = [t^n,t^{n+1}]$.
        To this end, assume $t^* := \max \mathcal{I}_n$ satisfies $t^* < t^{n+1}$.
        Then (\ref{noncomp_esti}) implies
        \begin{align*}
            \sup_{t \in [t^n,t^*]}\| e(t,\cdot)\|_{L^2}^2 &\leq A_n E_n \mathcal{H}_n(t^*) \\
            & \leq A_n E_n \exp\bigg(  \Delta t^n \sup_{t \in [t^n,t^*]} \left( B_1 \|e(t,\cdot)\|_{L^2}^2 + B_2 \|e(t,\cdot)\|_{L^2}^4\right) \bigg) \\
            & \leq A_n E_n \exp\bigg( \Delta t^n \left( B_1 \delta_n A_n E_n + B_2 (\delta_n A_n E_n)^2\right) \bigg).
        \end{align*}
        Next, we choose $\delta_n > 1$ such that
        \begin{align*}
            \exp\bigg( \Delta t^n \left( B_1 \delta_n A_n E_n + B_2 (\delta_n A_n E_n)^2\right) \bigg) \leq (1 - \alpha)\delta_n
        \end{align*}
        for some $\alpha \in (0,1)$. Then, with the estimate above, we find 
        \begin{align*}
            \sup_{t \in [t^n,t^*]}\| e(t,\cdot)\|_{L^2}^2 \leq A_n E_n (1 - \alpha)\delta_n <  A_n E_n \delta_n,
        \end{align*}
        and so $t^* < t^{n+1}$ cannot be the maximal value of $\mathcal{I}_n$. Consequently, the maximal value is $t^{n+1}$, which in turn implies that $\mathcal{I}_n = [t^n,t^{n+1}]$.
        Hence, selecting $\delta_n>1$ such that
        \begin{align*}
             \Delta t^n \left( B_1 \delta_n A_n E_n + B_2 \delta_n^2 A_n^2 E_n^2\right) - \log(\delta_n) < 0,
        \end{align*}
        whenever it exists, we obtain the local-in-time estimate
        \begin{align*}
            \sup_{t \in [t^n,t^{n+1}]}\norm{e(t,\cdot)}_{L^2}^2 \leq \delta_n A_n E_n,
        \end{align*}
        for all $n = 0, 1, \ldots, N_t-1$, which can iteratively be extended to a global stability estimate.
        We prove this by induction.
        Let $\tilde{A}_{n+1}$ be as in (\ref{iter_def}).
        % \begin{align}
        %     \label{iter_def}
        %     \tilde{A}_0 &:= A_0, \notag \\
        %     \tilde{A}_{n+1} &:= \Psi_n + 12\int_{t^{n+1}}^{t^{n+2}} \| R_{\bar{\rho}}(s,\cdot) \|_{(H^1(\mathbb{T}^d))^\prime}^2 \ ds, \quad \text{with} \quad \Psi_n := \delta_n \tilde{A}_n E_n.
        % \end{align}
        Setting $n=0$, the arguments above lead to $$\sup_{t \in [t^0,t^1]} \norm{\rho(t,\cdot) - \bar{\rho}(t,\cdot)}_{L^2}^2 \leq \delta_0 A_0 E_0.$$
        Now assume the estimate $$\sup_{t \in [t^0,t^n]} \norm{\rho(t,\cdot) - \bar{\rho}(t,\cdot)}_{L^2}^2 \leq \delta_{n-1} A_{n-1} E_{n-1}$$ holds for some $n \in  \{0, 1, \ldots, N_t-1\}$.
        Then,
        \begin{align*}
            \sup_{t \in [t^0,t^{n+1}]} \norm{\rho(t,\cdot) - \bar{\rho}(t,\cdot)}_{L^2}^2 = \max\bigg\{  \sup_{t \in [t^0,t^{n}]} \norm{\rho(t,\cdot) - \bar{\rho}(t,\cdot)}_{L^2}^2,  \sup_{t \in [t^n,t^{n+1}]} \norm{\rho(t,\cdot) - \bar{\rho}(t,\cdot)}_{L^2}^2 \bigg\}.
        \end{align*}
        We observe
        \begin{align*}
            \sup_{t \in [t^0,t^{n}]} \norm{\rho(t,\cdot) - \bar{\rho}(t,\cdot)}_{L^2}^2 \leq \delta_{n-1} A_{n-1} E_{n-1} \leq \delta_{n} A_{n} E_{n}. 
        \end{align*}
        In order to show
        \begin{align*}
             \sup_{t \in [t^n,t^{n+1}]} \norm{\rho(t,\cdot) - \bar{\rho}(t,\cdot)}_{L^2}^2 \leq \delta_{n} A_{n} E_{n},
        \end{align*}
        we first note that $A_n \leq \tilde{A}_{n}$, as $\|e(t^n,\cdot)\|_{L^2}^2 \leq \sup_{t \in [t^{n-1},t^{n}]}\|e(t,\cdot)\|_{L^2}^2 \leq \delta_{n-1} A_{n-1} E_{n-1}$ and 
        then repeat all arguments of the local estimates above, starting from (\ref{noncomp_esti}), while replacing $A_n$ by $\tilde{A}_n$.
        All in all, we find
        \begin{align*}
            \sup_{t \in [t^0,t^{n+1}]} \norm{\rho(t,\cdot) - \bar{\rho}(t,\cdot)}_{L^2}^2 \leq \delta_{n} A_{n} E_{n},
        \end{align*}
        whenever (\ref{root}) holds for all $n = 0,1,\ldots,N_t-1$.
        % Then due to 
        % \begin{align*}
        %     \norm{e(t^{n+1},\cdot)}_{L^2}^2 \leq \sup_{t \in [t^n,t^{n+1}]}\norm{e(t,\cdot)}_{L^2}^2\leq \Psi_n
        % \end{align*}
        % for $n = 0, 1, \ldots, N_t-1$, we obtain
        %     Let $n^* \in \{0, 1, \ldots, N_t-1\}$ be such that 
        %     \begin{align*}
        %         \sup_{t \in [0,T]} \norm{\rho(t,\cdot) - \bar{\rho}(t,\cdot)}_{L^2} = \sup_{t \in [t^{n^*},t^{n^*+1}]} \norm{\rho(t,\cdot) - \bar{\rho}(t,\cdot)}_{L^2}.
        %     \end{align*}
        %     Then
        %     \begin{align*}
        %         \sup_{t \in [0,T]} \norm{\rho(t,\cdot) - \bar{\rho}(t,\cdot)}_{L^2} \leq A_{n^*}E_{n^*}\delta_{n^*} \leq \tilde{\Psi}_{n^*}.
        %     \end{align*}
        \end{proof}
        
        % Note again that $\norm{\nabla \bar{c}(t,\cdot)}^2_{L^\infty(\Omega)} \leq C_\varepsilon^2 \norm{\bar{c}(t,\cdot)}^2_{W^{2,d+\varepsilon}(\Omega)}\leq C_\varepsilon^2 C_\text{ell}^2 \norm{\bar{\rho}(t,\cdot)}^2_{L^{d+\varepsilon}(\Omega)}$, for any $\varepsilon > 0$,
        % where the constant $C_\varepsilon> 0$ can be estimated explicitly, see \cite{Mizuguchi2017}.
        \begin{Remark}
            \label{rem:loc}
            This stability estimate is also conditional, see (\ref{root}), depending on the smallness of the pre-factors $B_2\Delta t^n A_n^2 E_n^2$ and $B_1\Delta t^n A_n E_n$. 
            Similar to the approach using the Generalized Gronwall Lemma, we expect the factors $A_n$ to be small, whenever the mesh size $h$ is sufficiently small.
            % Nevertheless, the lack of sharp upper bounds of constant $C_{\text{S}}$, especially in the case of $d=3$, see \cite{Mizuguchi2017}, restricts the practicality of both conditional stability frameworks presented in this section.
        \end{Remark}
        \noindent We can again formulate an \emph{a posteriori} verifiable existence statement as a direct consequence of Proposition~\ref{apostverif}. It reads as
        \begin{Corollary}
            Let $0 \leq \rho_0 \in L^2$ and let $T_{\text{max}}>0$ be the maximal existence time of the weak solution $(\rho,c)$ to (\ref{KS1}), (\ref{KS2}), (\ref{KScond}).
            Let $\tilde{\rho} \in L^\infty(0,T;L^2) \cap L^2(0,T;H^1)$ be a reconstruction of a numerical approximation to $\rho$.
            Suppose that for some $T>0$, one has $\| \tilde{\rho}(t,\cdot)\|_{L^\infty(0,T;L^2)} < \infty$. Suppose further that for every $n = 0,1,\ldots,N_t-1$, (\ref{root}) holds and that $A_n$ as well as $E_n$ are finite.
            Then $T < T_{\text{max}}$, i.e. the weak solution exists beyond time $T$.
        \end{Corollary}

        \begin{Remark} 
        Note that both stability frameworks above are independent of the exact choice of the solution $(\bar{\rho},\bar{c})$ to (\ref{res_bdrKS}), i.e. the reconstruction of the numerical solution and the numerical scheme itself. 
        For example in \cite{Kwon2023} a similar stability framework as Proposition \ref{adj_gengronwall} is used to derive a posteriori error estimates for a dG method.
        The only restrictions regarding the reconstruction are regularity constraints, the fact that a residual may only occur on the right-hand side of the first equation (\ref{KS1}), while no residual may occur in the elliptic equation (\ref{KS2}) and that we need to be able to compute all norms of the reconstruction, or upper bounds thereof, explicitly.
        % Additionally, similar stability results for the chemical concentration $c$ can be deduced via Theorem~\ref{stab_thm} and the elliptic coupling of $\rho$ and $c$ in (\ref{KS2}).
    \end{Remark}

%% file: numericalmethod.tex
\section{Numerical method}
\label{chapter-Num-method}

In order to tackle the Keller-Segel system numerically, we solve (\ref{KS1}) for the bacterial density, using a cell-centered FV scheme and (\ref{KS2}) for the chemical concentration using a $P1$-FE scheme. 
The numerical scheme is presented for $d = 3$, as this case is more delicate. Similar ideas are equally applicable for $d = 2$ and used to perform the corresponding numerical simulations, discussed in Section~\ref{num_sim}.

\noindent We choose this specific numerical scheme as it simplifies the presentation of the \emph{a posteriori} residual estimates, as fewer technical difficulties occur when dealing with the elliptic equation (\ref{KS2}), compared to using a FV-FV scheme:
Firstly, \emph{a posteriori} error estimators for diffusion-reaction problems, as (\ref{KS2}), are readily available for FE schemes, see \cite[Section 4.3]{Verfurth2013}. 
Secondly, using a FV scheme for (\ref{KS2}), would lead to a discontinuous velocity field in the convective term of (\ref{KS1}), which is not covered by \cite{Nicaise2006}, which is needed for reconstructing our numerical approximation, see Section~\ref{Morley}. 
Nevertheless, we expect that similar results to the ones presented in Section~\ref{aposterirorierror} can be shown for the FV scheme in \cite{Filbet2006} or the FE scheme in \cite{Saito2007}, but we anticipate them to require more technical work.

\noindent Let $\mathcal{T}_h$ be a conforming and uniformly shape regular tetrahedral mesh of $\mathbb{T}^d$,  with the maximal mesh size $h$.
% We discretize it, by choosing a conforming and uniformly shape regular tetrahedral mesh $\mathcal{T}_h$ with the maximal mesh size $h$, see \cite[Section 3.1.5]{Bartels2015}.
This means among others, that each element $K \in \mathcal{T}_h$ is a tetrahedron, i.e. it has four triangular faces, and satisfies $\sup_{h>0}\sup_{K \in \mathcal{T}_h} \frac{h_K}{r_K} \leq c_\text{usr}$ for some constant $c_\text{usr} > 0$.
Here $h_K$ denotes the diameter of the tetrahedron $K$ and $r_K$ its inradius, i.e. the radius of the largest ball that is completely contained in $K$.
Then the (maximal) mesh size is quantified by $h := \max_{K \in \mathcal{T}_h} \left\{ h_K \right\}$.
For cell-centered FV schemes, one needs to make sure, that for each element $K \in \mathcal{T}_h$ the cell-center, i.e. the circumcenter $x_K$, lies in the interior of $K$. 
The construction of \emph{well-centered} tetrahedral meshes is a delicate task in itself, yet it can be accomplished, see \cite{Hirani2008}.
 
\noindent We further define $F_h$, to be the set of all faces of the tetrahedral mesh $\mathcal{T}_h$.
% Let $F_h^{int} := \left\{ F \in F_h : F \nsubseteq \partial \Omega  \right\}$ denote the set of interior faces and $F_h^{bdr} := \left\{ F \in F_h : F \subset \partial \Omega  \right\}$ the set of boundary faces. 
% Consequently, $F_h = F_h^{int} \stackrel{.}{\cup} F_h^{bdr}$. 
For each $K \in \mathcal{T}_h$, we define $F_K := \left\{ F \in F_h : F \subset \partial K  \right\}$, the set of its (four) faces.

\noindent We denote the $d$-dimensional Lebesgue measure by $|\cdot|_d$. Then, for each tetrahedron $K \in \mathcal{T}_h$ and face $F \in F_K$, $|K|_3$ and $|F|_2$ denote their volume and area, respectively. We omit the subscript whenever it is clear from the context.
The unit outward normal to $K$ along $F$ is denoted by $n_{K,F}$. 
Furthermore, the unit outward normal vector of the boundary $\partial K$ is denoted by $n_K = n_K(x)$.
Lastly, for each face $F$ we fix two elements $K, L \in \mathcal{T}_h$ with $F = K \cap L$ and denote the unit normal vector associated with a face $F$ by $n_F$, pointing from $K$ to $L$. 
This notation is ambiguous as for each face $F=K\cap L$ we may choose $n_F = n_{K,F}$ or $n_F = n_{L,F}$ with $K,L \in \mathcal{T}_h$.
The exact choice does not matter, as long as it remains fixed throughout this whole work.
We will see that this ambiguity of the face normal, does not pose a problem for our purposes. 

\noindent In order to discretize in time, we choose the potentially non-uniform time steps $t^n \in [0,T]$ for $n = 0, 1, \dots, N_t$ with the time step sizes $\Delta t^n := t^{n+1}-t^n \geq 0$ such that $\sum_{n = 0}^{N_t-1} \Delta t^n = T$. 

\subsection{Finite Volume scheme for the bacterial density}
\label{FV-scheme-section}

Assume that some approximation $c_h^n$ of the chemical concentration $c$ at time $t^n$ is given, such that the function $\nabla c_h^n \cdot n_{K,F}$ is well-defined on all faces $F\in F_h$.
Then, we approximate the normal derivative of the chemical concentration by the mean over the face $F$ of its approximation $\nabla c_h^n \cdot n_{K,F}$, that is 
\begin{align*}
    \mathcal{M}_F(\nabla c_h^n \cdot n_{K,F}) := \frac{1}{|F|} \int_F \nabla c_h^n \cdot n_{K,F} \ dS(x) .
\end{align*}

% \begin{align*}
%     \int_K \partial_t \rho \ dx + \sum_{F \in F_K}\int_F  \left( \rho \nabla c - \nabla \rho \right)\cdot n_{K,F} \ dS(x)= 0
% \end{align*}
% is locally approximated in space by 
% \begin{align*}
%     |K| \partial_t \rho_K + \sum_{\substack{F \in F_K, \\ F = K \cap L}} |F| \left( \frac{\rho_K-\rho_L}{\log(\rho_K)-\log(\rho_L)} \mathcal{M}_F(\nabla c_h \cdot n_{K,F}) - \frac{\rho_L - \rho_K}{d_F} \right) = 0.
% \end{align*}

\noindent Then, on each $K \in \mathcal{T}_h$, the FV scheme reads as 
\begin{align}
    \label{FV_loc}
    \frac{\rho_K^{n+1}-\rho_K^{n}}{\Delta t^n} + \frac{1}{|K|} \sum_{F \in  F_K }\left( \mathcal{C}_F(\rho_h^{n}) \mathcal{M}_F(\nabla c^n_h \cdot n_{K,F}) - \mathcal{D}_{K,F}(\rho_h^{n+1}) \right) = 0,
\end{align}
% where we only sum over interior faces, to account for the homogenous Neumann boundary condition.
where for $F = K \cap L \in F_h$, with $K,L \in \mathcal{T}_h$, we define the convective flux
\begin{align*}
    \mathcal{C}_F(\rho_h^{n}) := \begin{cases}
        |F|\frac{\rho_K^{n}-\rho_L^{n}}{\log(\rho_K^{n})-\log(\rho_L^{n})}, \quad &\text{if } \rho_K^{n}, \rho_L^{n} > 0 \text{ and } \rho_K^{n} \neq \rho_L^{n},\\
        |F|\rho_K^{n} , \quad &\text{if } \rho_K^{n}, \rho_L^{n} > 0 \text{ and } \rho_K^{n} = \rho_L^{n}, \\
        0, \quad & \text{else},
    \end{cases}
\end{align*}
and the diffusive flux
\begin{align*}
    \mathcal{D}_{K,F}(\rho_h^{n+1}) = \frac{|F|}{d_F}\left( \rho_L^{n+1} - \rho_K^{n+1} \right),
    % := \begin{cases}
    %     \frac{|F|}{d_F}\left( \rho_L^{n+1} - \rho_K^{n+1} \right)\quad &\text{if } F = K \cap L \in F_h^\text{int} \\
    %     0 & \text{else},
    % \end{cases}
\end{align*}
with $d_F := \norm{x_L - x_K}_2$.
\noindent In words, we use the logarithmic mean whenever it is well-defined, else we continuously extend the flux.
This choice is motivated by the gradient flow structure of the Keller-Segel system and the corresponding entropy balance, see \cite[Remark 3.1]{Kolbe2023}.
The log-mean flux was originally used for appropriately discretizing non-linear (cross-) diffusion, see \cite{Bessemoulin2012,Jungel2023}.

\begin{Remark}
    We could replace the logarithmic mean in the definition of the convective flux by some other mean, e.g.
the arithmetic mean. Since we discretize the convective flux explicitly in our time discretization the choice of the
numerical convective flux only has a very small impact on the computational costs.
    % The logarithmic mean is one possible choice. One could alternatively think of e.g. the arithmetic mean.
    % As our time integrator, we choose a semi-implicit Euler scheme, implicit in the (linear) diffusive part and explicit in the (non-linear) advective part, such that in each time step the FV scheme boils down to solving a linear system of equations.
    % The choice of the numerical flux does hence not significantly influence the computational effort.
%     We choose a semi-implicit Euler scheme, implicit in the (linear) diffusive part and explicit in the (non-linear) advective part, such that in each time step the FV scheme boils down to solving a linear system of equations.
%    It is also feasible to choose the implicit Euler method for the bacterial density, while evaluating the chemical concentration explicitly. 
%    As this requires solving a non-linear system of equations in each time step, it is computationally more expensive and therefore not further considered here.
%    This approach did however not meet the expected computational advantages in numerical experiments compared with the Euler method that is fully-implicit in the bacterial density.
%    Especially the simulation of classical blow-up examples were not convincing.
%    The semi-implicit Euler method is thus not further considered here.
\end{Remark}

\noindent With
\begin{align*}
    \mathcal{D}_h(\rho^{n+1}_h) |_K &:= \frac{1}{|K|} \sum_{F \in F_K} \mathcal{D}_{K,F}(\rho_h^{n+1}) \quad \text{and} \\
    \mathcal{C}_h(\rho^n_h,c_h^n) |_K &:= \frac{1}{|K|} \sum_{F \in F_K} \mathcal{C}_F(\rho_h^{n}) \mathcal{M}_F(\nabla c^n_h \cdot n_{K,F}),
\end{align*}
we obtain the global formulation of (\ref{FV_loc}), which reads as
\begin{align}
    \label{FV}\tag{FV}
    \frac{\rho_h^{n+1}-\rho_h^{n}}{\Delta t^n} + \mathcal{C}_h(\rho^n_h,c_h^n) -  \mathcal{D}_h(\rho^{n+1}_h) = 0.
\end{align}

\subsection{Finite Element scheme for the chemical concentration}
\label{FE-scheme-section}

In order to get an approximation $c_h^n$ to the chemical concentration at time $t^n$, we use a $P1$-FE scheme, numerically solving (\ref{KS2}).

\noindent In (\ref{FV}), we require well-defined quantities $\nabla c_h^n \cdot n_{K,F}$ on all faces $F \in F_h$ and all $n = 0,1,\ldots,N_t$. 
For a $P1$-FE solution $c_h^n$ on the same mesh $\mathcal{T}_h$, the directional derivatives $\nabla c_h^n \cdot n_{K,F}$, are piecewise constant functions, jumping across the faces $F \in F_h$.
Hence, the required traces are not well-defined.
There are multiple ways to overcome this. We do so, by constructing another mesh $\mathcal{S}_{\mathcal{T}_h}$ for the $P1$-FE scheme, such that the quantities $\nabla c_h^n \cdot n_{K,F}$ are well-defined for all $K \in \mathcal{T}_h$ and $F \in F_K$.
We call the mesh $\mathcal{S}_{\mathcal{T}_h}$ the dual mesh with respect to the (primal) mesh $\mathcal{T}_h$. 
The notation introduced above regarding sets of faces, unit normal vectors etc., applies analogously for the dual mesh $\mathcal{S}_{\mathcal{T}_h}$ by replacing $F$ by $S$.

\noindent In order to construct the dual mesh $\mathcal{S}_{\mathcal{T}_h}$, we observe that for $P1$-FE approximations $c_h^n$, tangential gradients $\nabla c_h^n \cdot \tau_F$, i.e. directional derivatives along a face $F$ do not jump and are thus well-defined. 
We therefore construct the mesh $\mathcal{S}_{\mathcal{T}_h}$ such that the primal normal vectors $n_{K,F}$ lie tangential to the dual faces $S \in S_h$ for all $K \in \mathcal{T}_h$ and $F \in F_K$. In other words, the directional derivatives $\nabla c_h^n \cdot n_{K,F}$ need to be tangential gradients with respect to $\mathcal{S}_{\mathcal{T}_h}$.
One way to achieve this, is constructing each element of the dual mesh in the following way: Let a primal face $F = K \cap L \in F_h$ be given. A dual mesh element is then uniquely defined via its four vertices, which are the circumcenters $x_K$ and $x_L$, as well as two adjacent nodes of $F$.
% On boundary faces, we replace one circumcenter of a primal element by the circumcenter of the boundary face. 
This results in exactly three dual elements per primal face.
The construction of one dual element is illustrated in \autoref{meshes:dual3D}. For $d = 2$ a similar construction, connecting the two circumcenters of neighboring primal elements to one vertex of the edge that forms their intersection, leads to the dual mesh in \autoref{meshes:dual2D}.

% \begin{figure}[ht]
%     \centering
%     \includegraphics[width=0.6\linewidth]{figures/dual_mesh.png} 
%     \caption{Illustration of the construction of a single dual element (red) with respect to two neighboring primal elements (blue) and their circumcenters (green) for $d=3$.}
%     \label{meshes:dual3D} 
% \end{figure}

% \begin{figure}[ht]
%     \centering
%     \includegraphics[width=0.7\linewidth]{figures/intersected_mesh.png} 
%     \caption{Illustration of the primal mesh (blue), its circumcenters (green) and the corresponding dual mesh (red) for $d = 2$.}
%     \label{meshes:dual2D} 
% \end{figure}

% \begin{figure}
% \centering
% \begin{subfigure}{.5\textwidth}
%   \centering
%   \includegraphics[width=0.6\linewidth]{figures/dual_mesh.png} 
%     \caption{Illustration of the construction of a single dual element (red) with respect to two neighboring primal elements (blue) and their circumcenters (green) for $d=3$.}
%     \label{meshes:dual3D} 
% \end{subfigure}%
% \begin{subfigure}{.5\textwidth}
%     \centering
%     \includegraphics[width=0.7\linewidth]{figures/intersected_mesh.png} 
%     \caption{Illustration of the primal mesh (blue), its circumcenters (green) and the corresponding dual mesh (red) for $d = 2$.}
%     \label{meshes:dual2D} 
% \end{subfigure}
% % \caption{A figure with two subfigures}
% % \label{fig:test}
% \end{figure}

\begin{figure}
\centering
\begin{subfigure}[t]{.45\textwidth}
  \centering
  \includegraphics[width=\linewidth]{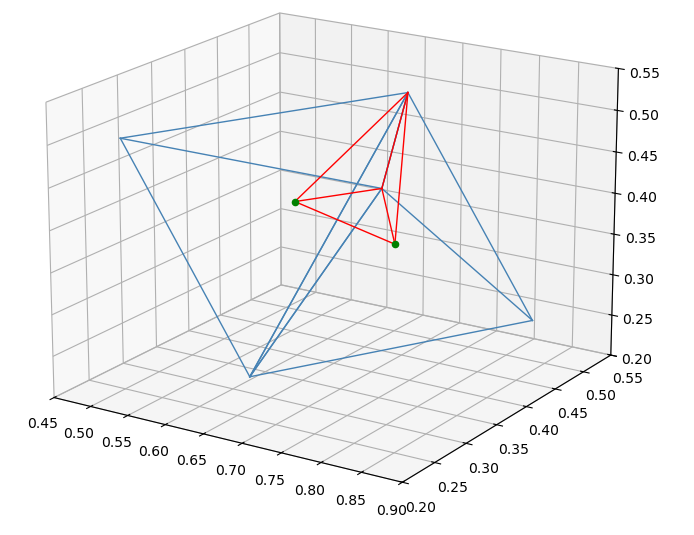}
  \caption{Illustration of the construction of a single dual element (red) with respect to two neighboring primal elements (blue) and their circumcenters (green) for $d=3$.}
  \label{meshes:dual3D}
\end{subfigure}\hfill
\begin{subfigure}[t]{.48\textwidth}
  \centering
  \includegraphics[width=\linewidth]{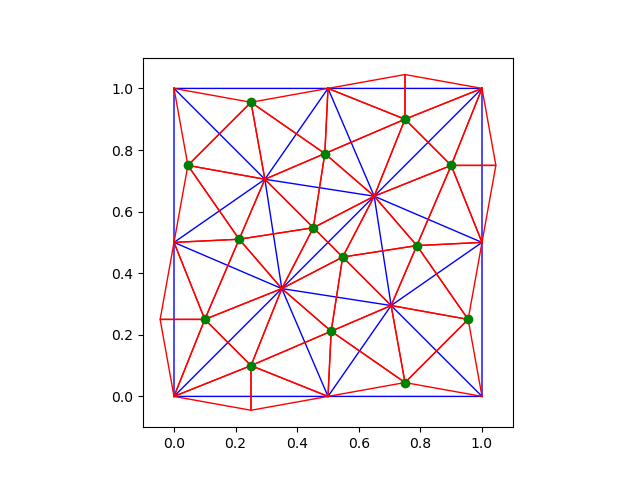}
  \caption{Illustration of the primal mesh (blue), its circumcenters (green) and the corresponding dual mesh (red) for $d = 2$.}
  \label{meshes:dual2D}
\end{subfigure}
\caption{Dual meshes in three and two dimensions.}
\end{figure}

\begin{Remark}
    Another idea to get well-defined traces of normal derivatives is to increase the regularity of our finite element. 
    \begin{enumerate}
        \item[i)] Indeed, $C^1$-finite elements can be considered, see \cite{Guzman2022}, and should fit our a posteriori analysis without big adjustments. 
        However, $C^1$-finite elements require a large number of DoFs, especially in three space dimensions and are absent from all major FE packages. 
        One way to circumvent this issue is suggested in \cite{Ainsworth2024}, but not further pursued here.

        \item[ii)] Alternatively, one can also think of choosing $H(div)$-conforming mixed finite elements, e.g. the Raviart-Thomas elements \cite{Raviart1975}, solving a first-order representation of the elliptic PDE (\ref{KS2}).
        This approach would require some adaptation of the arguments in our a posteriori analysis and is also not further considered in this work.

        \item[iii)] A third possibility is to use finite elements of Morley-type, as considered for reconstructions in Section~\ref{Morley}. However, normal derivatives can not straightforwardly be transformed from a reference element to a physical element, as needed for an efficient implementation of a finite element method. 
        This problem can be tackled using Piola transforms, see \cite{Kirby2018}. Nevertheless, this approach is also not followed here. 
    \end{enumerate}
\end{Remark}

\noindent In order to state our FE method, we introduce the discrete spaces
\begin{align*}
    P_0(\mathcal{T}_h) &:= \left\{ v \in L^2(\mathbb{T}^d,\R) : v |_{K} = \text{constant},\ \forall K \in \mathcal{T}_h \right\}, \\
    P_1(\mathcal{S}_{\mathcal{T}_h}) &:= \left\{ v \in H^1(\mathbb{T}^d,\R) \cap C^0(\mathbb{T}^d,\R) : v |_{K} = \text{affine linear},\  \forall K \in \mathcal{S}_{\mathcal{T}_h} \right\}.
\end{align*} 
We denote by $\Pi_0 : L^2(\mathbb{T}^d) \rightarrow P_0(\mathcal{T}_h)$ and $\Pi_1 : L^2(\mathbb{T}^d) \rightarrow P_1(\mathcal{S}_{\mathcal{T}_h})$ the $L^2$-projections onto the respective discrete spaces.
We can apply the $P1$-FE method to (\ref{KS2}), to compute an approximation $c_h^n \in P_1(\mathcal{S}_{\mathcal{T}_h})$ via
\begin{align}
    \label{FE}\tag{FE}
    \int_{\mathbb{T}^d} c_h^n v_h + \nabla c_h^n \cdot \nabla v_h \ dx = \int_{\mathbb{T}^d} \rho_h^n v_h \ dx, \quad \forall v_h \in P_1(\mathcal{S}_{\mathcal{T}_h}).
\end{align}
As the right-hand side, we use the integrable, piecewise constant FV approximation $\rho_h^n \in P_0(\mathcal{T}_h)$. 

\subsection{Finite Volume -- Finite Element algorithm}

Combining the schemes (\ref{FV}) and (\ref{FE}), we obtain an algorithm that we expect to approximate the (weak) solution of the Keller-Segel system (\ref{KS1}) and (\ref{KS2}), equipped with (\ref{KScond}).
\begin{Algorithm}[FV-FE]
    \label{FV-FE-algo}
    Let a primal mesh $\mathcal{T}_h$, its dual mesh $\mathcal{S}_{\mathcal{T}_h}$ and a time discretization $t^n \in [0,T]$ for $n = 0, ..., N_t$ be given, as above. \\

    \noindent We start by defining $\rho_h^0 := \Pi_0 \rho_0 \in P_0(\mathcal{T}_h)$. Then, using $\rho_h^{0}\in P_0(\mathcal{T}_h)$, we apply (\ref{FE}) to solve for $c_h^0 \in P_1(\mathcal{S}_{\mathcal{T}_h})$.
    \\ \ \\
    \noindent For $n = 0, 1, \dots, N_t -1 $ : 
    \begin{enumerate}
        \item Using $c_h^{n} \in P_1(\mathcal{S}_{\mathcal{T}_h})$, we apply (\ref{FV}) to solve for $\rho_h^{n+1} \in P_0(\mathcal{T}_h)$. 
        \item Using $\rho_h^{n+1} \in P_0(\mathcal{T}_h)$, we apply (\ref{FE}) to solve for $c_h^{n+1} \in P_1(\mathcal{S}_{\mathcal{T}_h})$.
    \end{enumerate}
\end{Algorithm}

%% file: reconstruction.tex
\section{Reconstruction of numerical approximations}
\label{Morley}

In order to establish the desired \emph{a posteriori} error estimates, we introduce a reconstruction of our numerical approximations.
For the FV solution, we use an interpolant of \emph{Morley-type}, which is explained in the following.
Alternatively, one could also use potential and flux reconstructions \cite{Chalhoub2011,Ern2010}, which will however not be discussed in the sequel.
%This is why we choose this exact interpolation.
%We expect similar \emph{a posteriori} error estimates, to the ones that we obtain, also to hold when using potential and flux reconstructions. %\textcolor{red}{das nochmal checken!}

\noindent Various interpolants of \emph{Morley-type} are introduced and used to derive \emph{a posteriori} error estimates for FV approximations of the steady diffusion-convection-reaction equation for $d \in \{2,3\}$ in \cite{Nicaise2006}.
In the following, we adapt the general case, stated in \cite[Section 4.4]{Nicaise2006}, to our purposes. %, of spatial interpolation in each time step $t^n$ for $n = 0,1,\dots,N_t$. \\

\noindent To this end, denote bubble functions, as defined in \cite{Nicaise2006}, by $b_K$ for the element bubble functions and $b_F$ for face bubble functions. 
All bubble functions are normalized, i.e. multiplied by a constant factor, such that they attain a maximal value of $1$. 
%This factor equals $256$ for all $b_K$ and $27$ for all $b_F$, in the case of $d = 3$.
For each $K\in\mathcal{T}_h$ and $F \in F_K$, we observe that bubble functions satisfy
\begin{align*}
    b_K = 0 \text{ on } \partial K \quad \text{and} \quad b_F = 0 \text{ on } \partial K \setminus F.
\end{align*}

\noindent The finite element of \emph{Morley-type} on tetrahedra, fit for our purposes, reads as 
\begin{align*}
    P_K &:= \left\{q_0 + \sum_{F \in F_K} \beta_{K,F} b_F b_K : q_0 \in \mathbb{P}^1(K), \ \beta_{K,F} \in \R\right\}, \\
    \Sigma_K &:= \left\{ p(a_i^K) \right\}_{i = 1,2,3,4} \cup \left\{ \int_F \nabla p \cdot n_{K,F} \ dS(x) \right\}_{F \in F_K},
\end{align*}
for each $K \in \mathcal{T}_h$, where the function $p$ lies in $P_K$ and where the $a_i^K$, denote the four vertices of the tetrahedron $K$. 
Here $P_K$ is the space of local ansatz functions, from which we construct the Morley interpolant, using the degrees of freedom that are collected in the set $\Sigma_K$, as done below, in (\ref{morley:nodes}) and (\ref{morley:diff}).
Along these lines, one can construct a similar finite element of \emph{Morley-type} for $d=2$.
%Note that for the linear interpolation $q_0$, many choices of nodal values are possible.
We immediately get
\begin{Lemma}
    The triple $\left(K,P_K,\Sigma_K\right)$ is a $C^0$-finite element.
\end{Lemma}
\begin{proof}
    The proof follows similar arguments as in \cite[Section 4]{Nicaise2006}. We omit it for brevity.
\end{proof}
% Note that for this proof the exact construction of the tetrahedral mesh $\mathcal{T}_h$ is irrelevant.

% \subsection{The Morley interpolant}
% \label{morley-int}
\noindent We are now prepared to introduce the Morley interpolant, similar to the one defined in \cite{Nicaise2006}. 
For any $K \in \mathcal{T}_h$, we choose the $C^0$-finite element $\left(K,P_K,\Sigma_K\right)$ from above and set
\begin{align*}
    U_h := \left\{p \in H^1(\Omega) \cap C^0(\bar{\Omega}) : p|_K \in P_K, K \in \mathcal{T}_h \right\}.
\end{align*}
Then, for a piecewise constant function $\rho_h^n \in P_0(\mathcal{T}_h)$, we define its Morley interpolant as the unique element $\tilde{\rho}^n \in U_h$, such that
\begin{align}
    \tilde{\rho}^n(a_i^K) &= y_i, \quad \forall K \in \mathcal{T}_h \text{ and } i = 1,2,3,4\ , \label{morley:nodes} \\
    \int_F \nabla \tilde{\rho}^n \cdot n_{K,F} \ dS(x)  &= \mathcal{D}_{K,F}(\rho_h^n), \quad \forall F \in F_K, K \in \mathcal{T}_h. \label{morley:diff} %\\
    %\int_E \left(\nabla c_h^{n-1} \cdot n_{E}\right) \tilde{\rho}^n \ dS(x) &= \mathcal{M}_E\left(\nabla c_h^{n-1} \cdot n_{E}\right) \mathcal{F}_E(\rho_h^n), \quad \forall E \in E_h^{int}.  \label{morley:convection1} 
\end{align}
The nodal values $y_i$ are chosen via a least-squares approach as done in \cite[Section 3.3]{Coudiere1999}.
The preservation of numerical diffusive fluxes (\ref{morley:diff}) will be exploited in residual estimates during the derivation of the \emph{a posteriori} error estimates in Section~\ref{aposterirorierror}. 

\begin{Remark}
    In Section~\ref{num_sim}, we choose the nodal values, i.e. (\ref{morley:nodes}), according to \cite[Section 3.3]{Coudiere1999}, which leads to optimal polynomial approximation properties of $\tilde{\rho}$, i.e. quadratic in $L^2(\mathbb{T}^d)$ and linear in $H^1(\mathbb{T}^d)$.
    This is not the case if one chooses the arithmetic mean over all adjacent elements with respect to $a_i^K$, for $d=2$.
\end{Remark} 

% In this work we choose the latter. 

%\textcolor{red}{need practical perspective here?}
% From a practical point of view, given some $\rho_h^n \in P_0(\mathcal{T}_h)$, the Morley interpolant can be computed on each element $K \in \mathcal{T}_h$, by successively imposing the prescribed values via (\ref{morley:nodes}) and (\ref{morley:diff}). % and (\ref{morley:convection}). 
% Let
% \begin{align*}
%     \tilde{\rho}^n|_K = q_0^n|_K + \sum_{F \in F_K} \beta_F^n b_F b_K,
% \end{align*}
% for the unknowns $q_0^n \in \mathbb{P}^1(\mathcal{T}_h)$ and $\beta_F^n \in \R$ for all $F \in F_K$ and all $K \in \mathcal{T}_h$ and some $n = 0,1,\dots,N_t$.

% We first consider (\ref{morley:nodes}). 
% As the bubble functions vanish in all nodes of our tetrahedral mesh, we can uniquely find the piecewise linear function $q_0^n \in \mathbb{P}^1(\mathcal{T}_h)$ subject to the nodal values, independently of coefficients $\beta_F$.
% Thereafter, in order to compute the coefficients $\beta_F$, one plugs $\tilde{\rho}^n|_K$ from above into (\ref{morley:diff}). 
% Knowing that $b_F^n$ is zero on all faces except $F$ and having already computed $q_0^n \in \mathbb{P}^1(\mathcal{T}_h)$, the product rule and rearranging yields a closed-form expression for the coefficients $\beta_E^n$, i.e.
% \begin{align}
%     \label{closed_beta}
%     \beta_F^n = \frac{\mathcal{D}_{K,F}(\rho_h^{n}) - \int_F \nabla q_0^n \cdot n_{K,F} \ dS(x)}{\int_F b_F \nabla b_K \cdot n_{K,F} \ dS(x)}.
% \end{align}

\noindent Note that this finite element of \emph{Morley type} is not needed in the numerical scheme itself. 
It is only needed to compute our \emph{a posteriori} error estimator, see Section~\ref{aposterirorierror}.

\noindent In order to reconstruct in time, we assume $t \in [t^n,t^{n+1}]$, for any $n = 0,1,\dots, N_t-1$.
    %The case of $t \in [0,t^{1}]$ will be considered later.
    Let $\tilde{\rho}^n$ denote the Morley interpolant of the FV solution $\rho^n_h \in P_0(\mathcal{T}_h)$.
    % For $n=0$, we recall $\rho_h^0|_K = \rho_0(x_K)$, for all $K\in\mathcal{T}_h$, where $x_K$ denotes the circumcenter of the tetrahedron $K$. 
    %Additionally, we set $\tilde{\rho}^0 := \rho_0$.
    We introduce the temporal interpolation
    \begin{align*}
        \tilde{\rho}(t,x) := \ell_0^n(t)\tilde{\rho}^{n+1}(x)+\ell_1^n(t)\tilde{\rho}^{n}(x), \quad \forall t \in \left[ t^n , t^{n+1} \right],
    \end{align*}
    where we use the Lagrange polynomials
    \begin{align*}
       \ell_0^n(t) := \frac{t-t^n}{\Delta t^n} \quad \text{and} \quad  \ell_1^n(t) := \frac{t^{n+1}-t}{\Delta t^n}.
    \end{align*}
    It holds $0 \leq \ell_0^n(t) \leq 1$ and $0 \leq \ell_1^n(t) \leq 1$ as well as
    \begin{align*}
        %\label{temporal_derivative}
        \partial_t \tilde{\rho}(t,\cdot) = \frac{\tilde{\rho}^{n+1}-\tilde{\rho}^{n}}{\Delta t^n},
    \end{align*}
    for all $t \in [t^n,t^{n+1}]$.
    We reconstruct the chemical concentration at every point in time $t \in [0,T]$ as the weak solution to the elliptic equation 
    \begin{align*}
            \tilde{c}(t,\cdot)-\Delta \tilde{c}(t,\cdot) &= \tilde{\rho}(t,\cdot), \quad \text{in } \mathbb{T}^d, \\
            % \nabla \tilde{c}(t,\cdot) \cdot n &= 0, \quad\quad\ \   \text{on } \partial \Omega.
    \end{align*}

    % In other words
    % \begin{align*}
    %     \int_\Omega \tilde{c} \phi \ dx + \int_\Omega \nabla \tilde{c} \cdot \nabla \phi \ dx = \int_\Omega \tilde{\rho} \phi \ dx, \quad \forall \phi \in H^1(\Omega).
    % \end{align*}
    % \begin{align*}
    %     \int_\Omega \tilde{c} \phi \ dx + \int_\Omega \nabla \tilde{c} \cdot \nabla \phi \ dx = \int_\Omega g_h \phi \ dx, \quad \forall \phi \in H^1(\Omega).
    % \end{align*}
    % Here the right-hand side is given by
    % \begin{align*}
    %     g_h := A_h c_h - \Pi_1 \tilde{\rho} + \tilde{\rho},
    % \end{align*}
    % where $\Pi_1$ is the $L^2$-projection onto the finite element space $P_1(\mathcal{T}_h)$, the time-interpolated finite element solution is given by 
    % \begin{align*}
    %     c_h(t,x) := \ell_0^n(t)c_h^{n+1}(x)+\ell_1^n(t)c_h^{n}(x), \quad \forall t \in \left[ t^n , t^{n+1} \right],
    % \end{align*}
    % and the discrete elliptic operator $A_h: P_1(\mathcal{T}_h) \rightarrow P_1(\mathcal{T}_h)$ is defined by
    % \begin{align*}
    %     \int_\Omega A_h v \phi = \int_\Omega v \phi \ dx + \int_\Omega \nabla v \cdot \nabla \phi \ dx, \quad \forall \phi \in P_1(\mathcal{T}_h). 
    % \end{align*}
    % This reconstruction $\tilde{c}$ is of purely theoretical use and is not needed explicitly for the computations in numerical experiments. 
    % However, in order to bound it using fully computable terms, we exploit a posteriori error estimates for second order elliptic problems from the literature.
    % This will be made more precise in what follows.
    \noindent Due to the linearity of the elliptic equation and due to the uniqueness of weak solutions thereof, we observe that
    \begin{align*}
        \tilde{c}(t,x) = \ell_0^n(t)\tilde{c}(t^{n+1},x) + \ell_1^n(t)\tilde{c}(t^{n},x).
    \end{align*}
    We need this specific reconstruction for the chemical concentration, such that our perturbed equation has the form of (\ref{res_bdrKS}), i.e. there is no residual in the second equation.
    Other reconstructions, inducing residuals in the second equation, would require a different stability framework.
    
    % \noindent Consequently, for convex domains  $\Omega$, elliptic regularity of $\tilde{c}$ holds, see \cite{Grisvard2011}. 
    % We thus assume throughout the rest of this work that the domain $\Omega$ is convex. 
    
    \noindent In accordance with the stability frameworks above, we interpret $(\tilde{\rho},\tilde{c})$ as a solution to (\ref{res_bdrKS}). 

    \begin{Remark}
        Note, that the Morley interpolant $\tilde{\rho}$ and especially its Lebesgue norms are computable, which is required for the computation of the a posteriori residual estimates, see Section~\ref{aposterirorierror}.
        On the other hand, the elliptic reconstruction $\tilde{c}$ is not computable and a purely theoretical object.
        Certain norms of $\tilde{c}$ will be estimated by computable terms in the following.
    \end{Remark}
    \begin{Remark}
        \label{grad_c_up}
        The term $\norm{\nabla \tilde{c}(t,\cdot)}_{L^\infty}$ is not computable, for all $t \in [0,T]$. For brevity, we omit the $t$-dependence in the sequel.
        We provide an upper bound for $\norm{\nabla \tilde{c}}_{L^\infty}$ as follows.
        \begin{enumerate}
            \item[i)] Via the Sobolev embedding $W^{2,d + \mu}(\mathbb{T}^d) \hookrightarrow W^{1,\infty}(\mathbb{T}^d)$ for all $\mu > 0$ and elliptic regularity in $L^p$-spaces, see \cite[Theorem 9.11]{Gilbarg1977},
            we have 
            \begin{align*}
                \norm{\nabla \tilde{c}}_{L^\infty} \leq C_{d+\mu} \norm{ \nabla^2 \tilde{c}}_{L^{d+\mu}}\leq C_{d+\mu}
                \tilde{C}_{\text{ell}} \norm{\tilde{\rho}}_{L^{d+\mu}},
            \end{align*}
            for any  $\mu > 0$, where we used that gradient fields on the flat torus are mean free.
            The constants $C_{d+\mu}, \tilde{C}_{\text{ell}} > 0$ can be estimated explicitly, see \cite[Appendix A]{BrunkPrep} and \cite[Theorem 3.4]{Mizuguchi2017}, when working on the torus $\mathbb{T}^d$.
            \item[ii)] Alternatively, one can observe that $\nabla \tilde{c}$ weakly solves the vector-valued, elliptic problem $(I - \Delta)\nabla \tilde{c} = \nabla \tilde{\rho}$ in $\mathbb{T}^d$.
            We can approximate $\nabla \tilde{c}$ by some FE solution $q_h$. 
            We bound 
            \begin{align*}
                \norm{\nabla \tilde{c}}_{L^\infty} \leq \norm{q_h}_{L^\infty} + \norm{\nabla \tilde{c} - q_h}_{L^\infty}
                \leq \norm{q_h}_{L^\infty} + \eta_\infty,
            \end{align*}
            where we use the maximum-norm a posteriori error estimator $\eta_\infty$ as derived in Appendix~\ref{A:Linf}. % hier demlow,kopteva,franz zitieren!
            % \item[iii)] For rougher domains, e.g. bounded polygonal domains with non-periodic boundary conditions, \cite[Theorem 9.11]{Gilbarg1977} still applies, but the constant $\tilde{C}_{\text{ell}} > 0$ is to our knowledge in general not available explicitly.
        \end{enumerate}
        For sufficiently fine meshes, we expect the second approach to provide sharper upper bounds than the first approach.
        % \begin{align*}
        %     \eta_\infty = \max_{K \in \mathcal{T}_h} \{ \alpha_K \norm{\Delta q_h - q_h +  \nabla \bar{\rho}}_{L^\infty(K)} + \beta_K \norm{ \llbracket \nabla q_h \rrbracket }_{L^\infty(\partial K \setminus \partial \Omega)} \}.
        % \end{align*}
        % with $\alpha_K := \min\{ 1, \gamma_h h_K^2 \}$, $\beta_K := \min\{ 1, \gamma_h h_K \}$ and $\gamma_h := \log\big(2+(\min_{K \in \mathcal{T}_h} \{ h_K \})^{-1}\big)$ as in \cite[Section 3]{Demlow2016}.
        % The constant $C_\infty > 0$ can be estimated explicitly by investigating the proof of \cite[Lemma 6]{Demlow2016}, using \cite[Section 4]{Brenner2008}, as well as \cite{Dong2009, Verfurth1999}.
        % This is a computable upper bound, whenever $\bar{\rho}$ is computable.
    \end{Remark}

%% file: residualestimates.tex
\section{A posteriori analysis}
    \label{aposterirorierror}
In this section, we aim to derive a fully computable \emph{a posteriori} error estimator, fit to infer \emph{a posteriori} existence via Proposition \ref{apostverif}.
For full computability, we need to deal with errors due to floating-point arithmetic, errors from iteratively solving the resulting linear system of equations, also called the algebraic error, and the discretization error, which is encoded in the residual $R_{\bar{\rho}}$ from (\ref{res_bdrKS}).

\subsection{Algebraic error}
\label{alg-error}
    First, we investigate how round-off errors can be estimated. Let $*$ denote an abstract operation, that requires $N$ elementary operations. Consider
    \begin{align*}
        \widetilde{a*b} = (a*b)\prod_{i=1}^N(1+\delta_i), \quad |\delta_i|\leq \varepsilon,
    \end{align*}
    where $\widetilde{a*b}$ denotes perturbed versions of $a*b$ due to the round-off error for each elementary operation, and $\varepsilon>0$ is the machine epsilon, e.g. $\varepsilon \approx 2.2204$e-$16$ for double precision.
    Then, we observe,
    \begin{align}
        \label{alg:roff}
        \frac{|\widetilde{a*b} - a*b |}{|a*b|} \leq ((1+\varepsilon)^N - 1) \leq \frac{N \varepsilon}{1 - N \varepsilon}, \quad \text{for } N\varepsilon<1,
    \end{align}
    which we will use in the following. \\
    \noindent Now, we are prepared to follow \cite[Section 4]{Vohralik2010FV}. 
    Computing the finite volume solution, via (\ref{FV}) is, for each $n = 1, 2, \ldots, N_t$, equivalent to solving a linear system of equations in each time step, i.e.
    \begin{align}
        \label{FV_LSE}
    \mathbb{S}_1 P^{n} = H_1 %P^{n} - \Delta t^n \mathcal{C}_h(\rho^n_h,c_h^n).
    \end{align}
    with $H_1 \in \mathbb{R}^{N_{\mathcal{T}_h}}$, $\mathbb{S}_1 \in \mathbb{R}^{N_{\mathcal{T}_h} \times N_{\mathcal{T}_h}}$ and $P^{n} \in \mathbb{R}^{N_{\mathcal{T}_h}}$, where $\rho_K^{n} = P_{\pi(K)}^n$ for all $K \in \mathcal{T}_h$ with a bijection $\pi : \mathcal{T}_h \rightarrow \{1, \ldots, N_{\mathcal{T}_h}\}$, where $N_{\mathcal{T}_h}$ is the number of elements in $\mathcal{T}_h$.
    In practice, the floating-point arithmetic leads to round-off errors in the system matrix and the right-hand side, while iteratively solving the LSE yields an additional algebraic error. We account for these inevitable inaccuracies in the following:
    Let $P_a^{n}$ be an approximate solution to (\ref{FV_LSE}), due to round-off errors and algebraic errors, solving
    $\tilde{\mathbb{S}}_1 P_a^{n} =: \tilde{H}_1^a$
    where $\tilde{\mathbb{S}}_1 \in \mathbb{R}^{N_{\mathcal{T}_h} \times N_{\mathcal{T}_h}}$ is the perturbed system matrix due to round-off errors corresponding to $\mathbb{S}_1$. 
    Similarly, let $\tilde{H}_1 \in \mathbb{R}^{N_{\mathcal{T}_h}}$ be the perturbed right-hand side, corresponding to $H_1$. 
    The vector $\tilde{H}_1^a$ differs from $\tilde{H}_1$ as it additionally covers the algebraic error from iteratively solving the linear system of equations $\tilde{\mathbb{S}}_1 P = \tilde{H}_1$.
    We then have 
    \begin{align}
        \label{FV_LSE_alg}
        \mathbb{S}_1 P_a^{n} = H_1 + R_{\text{FV}}^n, %\underbrace{P_a^{n} - \Delta t^n \mathcal{C}_h(\rho^n_{h,a},c_{h,a}^n)}_{ =: H} 
    \end{align}
    where the algebraic residual is defined as $R_{\text{FV}}^n := ( \tilde{H}_1^a - H_1) - (\tilde{\mathbb{S}}_1 - \mathbb{S}_1) P_a^{n}$. The residual function $r_{\text{FV}}^n \in P_0(\mathcal{T}_h)$ associated with the residual vector $R_{\text{FV}}^n$ is
    \begin{align*}
        r_{\text{FV}}^n |_K := \frac{(R_{\text{FV}}^n)_{\pi(K)}}{\Delta t^{n-1}}, \quad \text{for } n = 1,2,\ldots, N_t.
    \end{align*}
    Then, setting $\rho_{h,a}^n|_K := P_{\pi(K),a}^{n}$ for all $K \in \mathcal{T}_h$, identity (\ref{FV_LSE_alg}) is equivalent to
    \begin{align*}
        \frac{\rho_{h,a}^{n}-\rho_{h,a}^{n-1}}{\Delta t^{n-1}} + \mathcal{C}_h(\rho^{n-1}_{h,a},c_{h,a}^{n-1}) -  \mathcal{D}_h(\rho^{n}_{h,a}) = -r_{\text{FV}}^n.
    \end{align*}
    We are later interested in the $L^2$-norm of $r_{\text{FV}}^n$. 
    We denote the Euclidean norm for vectors by $\| v\|_2^2 := \sum_{i=1}^N v_i^2$ and the associated matrix norm is the spectral norm $\| A \|_2^2 := \lambda_\text{max}(A^\top A)$, where $\lambda_\text{max}(A^\top A)$ denotes the largest eigenvalue of $A^\top A$.
    Let $\text{diag}(|\mathcal{T}_h|)^{1/2}$ be the diagonal matrix with $(\text{diag}(|\mathcal{T}_h|)^{1/2})_{i,i} := |\pi^{-1}(i)|^{1/2}$ for $i \in \{1, \ldots, N_{\mathcal{T}_h}\}$. 
    We observe
    \begin{align}
        \label{FV-alg}
        &\|r_{\text{FV}}^n\|_{L^2(\Omega)} =  \frac{1}{\Delta t^{n-1}} \| \text{diag}(|\mathcal{T}_h|)^{1/2} R_{\text{FV}}^n \|_2 \notag \\ 
        &\quad\leq \frac{1}{\Delta t^{n-1}}\| \text{diag}(|\mathcal{T}_h|)^{1/2} \|_2 \  \| R_{\text{FV}}^n \|_2 \notag \\
        &\quad\leq \frac{1}{\Delta t^{n-1}} \max_{K \in \mathcal{T}_h} |K|^{1/2} \big(  \| \tilde{H}_1^a - H_1 \|_2 + \| \tilde{\mathbb{S}}_1 - \mathbb{S}_1 \|_2 \| P_a^{n} \|_2 \big). \notag  \\
        &\quad\leq \frac{1}{\Delta t^{n-1}} \max_{K \in \mathcal{T}_h} |K|^{1/2} \bigg(\frac{C_1(h) \varepsilon}{1-2C_1(h) \varepsilon}\|\tilde{H}_1\|_2 +  \| \tilde{H}_1^a - \tilde{H}_1 \|_2  + \frac{C_2(h) \varepsilon}{1-2C_2(h) \varepsilon} \|\tilde{\mathbb{S}}_1\|_2 \| P_a^{n} \|_2 \bigg) =: \theta_\text{FV}^n,
    \end{align}
    similar to (\ref{alg:roff}),
    where $C_1(h), C_2(h) > 0$ are the numbers of elementary operations needed to compute $\tilde{H}_1$ and $\tilde{\mathbb{S}}_1$, respectively, which can be obtained from the implementation and are stated in Section~\ref{num_sim}.
    Additionally, note that we also need to account for round-off errors in the identity on the preservation of diffusive fluxes (\ref{morley:diff}).
    This is equivalent to solving an affine equation for each face $F \in F_h$, i.e. find $x \in \R$ such that $a^\top x + b = 0$. The perturbed affine equation reads as
    \begin{align*}
        a^\top(1+e_a) x + b(1+e_b) = 0,
    \end{align*}
    where $|e_a| \leq C_a \varepsilon$ and $|e_b| \leq C_b \varepsilon$ for some $C_a, C_b > 0$ which correspond to the number of elementary operations needed to assemble $a \in \R^N$ and $b \in \R$, respectively.
    Let $\tilde{\rho}_a^n$ be the Morley reconstruction of $\rho_{h,a}^n$. For the preservation of diffusive fluxes this translates to
    \begin{align}
        \int_F \nabla \tilde{\rho}_a^n \cdot n_{K,F} \ dS(x) + \mathcal{D}_{K,F}(\rho_{h,a}^n)e_a^F + |F| (\nabla q_0 \cdot n_{K,F}) e_b^F &= \mathcal{D}_{K,F}(\rho_{h,a}^n), \label{perturbed:diff} 
    \end{align}
    for all $F \in F_K, K \in \mathcal{T}_h$,
    where $e_a^F,e_b^F  \in \R$ satisfy $|e_a^F| \leq C_a^F \varepsilon$ with $C_a^F>0$ denoting the number of elementary operations per face, needed to compute $\mathcal{D}_{K,F}(\rho_{h,a}^n)$ and $|e_b^F| \leq C_b^F \varepsilon$ with $C_b^F>0$ denoting the number of elementary operations per face, needed to compute $|F| (\nabla q_0 \cdot n_{K,F})$.
    We set 
    \begin{align}
        \label{def:c3}
        C_3(h) := C_3^a(h) + C_3^b(h),
    \end{align}
    where $C_3^a(h) := \max_{F \in F_h, F \subset \partial K} |\mathcal{D}_{K,F}(\rho_{h,a}^n)|C_a^F$ and $C_3^b(h) := \max_{F \in F_h, F \subset \partial K} | |F| (\nabla q_0 \cdot n_{K,F}) |C_b^F$, for which explicit values can be obtained from the implementation and are stated in Section~\ref{num_sim}. \\ \ \\
    \noindent For the finite element scheme, following \cite[Section 5]{Vohralik2018FE}, we similarly see that
    %We take a slightly different approach for the finite element approximation of the second equation. 
    the identity (\ref{FE}) is  equivalent to 
    \begin{align}
        \label{FE_LSE}
        \mathbb{S}_2 C^n = H_2
    \end{align}
    with $H_2 \in \mathbb{R}^{N_x}$, $\mathbb{S}_2 \in \mathbb{R}^{N_x \times N_x}$ and $C^{n} \in \mathbb{R}^{N_x}$ where  $c_h^{n} =  \mathbf{\Phi} \cdot C^n $, where $\mathbf{\Phi} = (\Phi_1, \Phi_2, \ldots, \Phi_{N_x})^\top$ is a basis of $P_1(\mathcal{S}_{\mathcal{T}_h})$ and $N_x$ the number of nodes in $\mathcal{S}_{\mathcal{T}_h}$.
    Let $C_a^n$ be an approximate solution to (\ref{FE_LSE}), corresponding to the perturbed $\tilde{H}_2 \in \mathbb{R}^{N_x}$, $\tilde{\mathbb{S}}_2 \in \mathbb{R}^{N_x \times N_x}$ and an iterative solver $\tilde{H}_2^a := \tilde{\mathbb{S}}_2 C_a^n$. 
    Then,
    \begin{align}
        \label{FE_LSE_alg}
        \mathbb{S}_2 C^n_a = H_2 + R_\text{FE}^n,
    \end{align}
    where $R_{\text{FE}}^n := ( \tilde{H}_2^a - H_2) - (\tilde{\mathbb{S}}_2 - \mathbb{S}_2) C_a^{n}$.
    Similar to before, we set $c_{h,a}^{n} := \mathbf{\Phi} \cdot C^n_a$.
    We define the residual function $r_\text{FE}^n$ associated with the residual vector $R_\text{FE}^n$ by 
    \begin{align*}
        r_\text{FE}^n := \mathbf{\Phi}^\top \mathbb{M}^{-1} R_\text{FE}^n,
    \end{align*}
    where $\mathbb{M}$ is the global mass matrix $\mathbb{M}_{i,j} = \int_{\mathbb{T}^d} \Phi_i \Phi_j \ dx$.
    Note that the residual function $r_\text{FE}^n$ is a purely theoretical object which will not be computed explicitly.
    % \begin{align*}
    %     \langle r_h^\text{FE} , \Phi_j \rangle_{L^2(K)} = \frac{R_j^\text{FE}}{n_j} =: (R_K^\text{FE})_j, \quad \text{for supp}\Phi_j  \cap K \neq \emptyset,
    % \end{align*}
    % where $n_j$ is the number of elements forming supp$\Phi_j$. This requires solving a small LSE for each element $K \in \mathcal{S}_{\mathcal{T}_h}$. Note that this is a piece-wise polynomial function that is discontinuous across faces.
    % Additionally, we set
    % \begin{align*}
    %     (\mathbb{S}_K)_{ij} = \langle \Phi_i , \Phi_j \rangle_{L^2(K)}, \quad \text{for supp}\Phi_j  \cap K \neq \emptyset \text{ and supp}\Phi_j  \cap K \neq \emptyset.
    % \end{align*}
    Then, (\ref{FE_LSE_alg}) is equivalent to
    \begin{align*}
        \int_{\mathbb{T}^d} c_{h,a}^n v_h + \nabla c_{h,a}^n \cdot \nabla v_h \ dx = \int_{\mathbb{T}^d} (\rho_{h,a}^n + r_\text{FE}^n) v_h \ dx, \quad \text{ for all } v_h \in P_1(\mathcal{S}_{\mathcal{T}_h}).
    \end{align*}
    The resulting perturbed Galerkin orthogonality, used in the corresponding finite element \emph{a posteriori} error estimator in Lemma~\ref{conv-res}, reads as
    \begin{align}
        \label{pert_orth}
        \langle c - c_{h,a} , v_h \rangle_{H^1} = -  \langle r_\text{FE}^n , v_h \rangle_{L^2}, \quad \text{ for all } v_h \in P_1(\mathcal{S}_{\mathcal{T}_h}),
    \end{align}
    where $\langle \cdot , \cdot \rangle_{H^1}$ and $\langle \cdot , \cdot \rangle_{L^2}$ denote the usual scalar products corresponding to Hilbert spaces $H^1$ and $L^2$.
    We are later interested in the $L^2$-norm of $r_{\text{FE}}^n$.
    As preparation, we represent the perturbation of the mass matrix as 
    \begin{align*}
        \tilde{\mathbb{M}} - \mathbb{M} =: E = \Theta \odot  \mathbb{M}, \quad \text{with } |\Theta_{ij}| \leq \frac{C_4 \varepsilon}{1-C_4 \varepsilon} \quad \forall i,j \in {1,2,\ldots,N_x},
    \end{align*}
    where $\odot $ denotes the entry-wise matrix-product,
    and the constant $C_4>0$ denotes the maximal number of elementary operations per entry needed to compute $\tilde{\mathbb{M}}$.
    We can estimate
    \begin{align*}
        \|E\|_\infty \leq \max_{i,j = 1,\ldots,N_x} |\Theta_{ij}| \| \mathbb{M} \|_\infty \leq \frac{C_4 \varepsilon}{1-C_4\varepsilon} \| \mathbb{M} \|_\infty \leq \frac{C_4 \varepsilon}{1-C_4 \varepsilon} (\| \tilde{\mathbb{M}} \|_\infty + \|E\|_\infty),
    \end{align*}
    hence
    \begin{align*}
        \|E\|_\infty \leq  \frac{C_4 \varepsilon}{1-2C_4\varepsilon} \| \tilde{\mathbb{M}} \|_\infty, \quad \text{for} \quad \frac{C_4 \varepsilon}{1-C_4 \varepsilon} < 1,
    \end{align*}
    where $\|A\|_\infty := \max_{i \in \{1,\ldots,N\}} \sum_{j \in \{1,\ldots,N\}} |a_{ij}|$ denotes the usual matrix infinity norm.
    In order to relate this to the spectral norm, we observe that for any sparse matrix $A$, one has
    \begin{align*}
        \frac{1}{\sqrt{k}} \|A\|_2 \leq \|A\|_\infty \leq \sqrt{k} \|A\|_2,
    \end{align*}
    where $k$ is the maximal number of non-zero entries per row and column, i.e. $k=6$ for the triangular mesh we use, see Figure~\ref{meshes:dual2D} and $k=12$ for our tetrahedral mesh.
    Thus,
    \begin{align}
        \label{A:E}
        \|E\|_2 \leq \sqrt{k}\|E\|_\infty \leq \frac{\sqrt{k} C_4 \varepsilon}{1-2C_4\varepsilon} \| \tilde{\mathbb{M}} \|_\infty \leq   \frac{k C_4 \varepsilon}{1-2C_4\varepsilon} \| \tilde{\mathbb{M}} \|_2.
    \end{align}
    Further, we observe %$\tilde{\mathbb{M}} = \mathbb{M} + \tilde{\mathbb{M}} - \mathbb{M} = \mathbb{M}( I + \mathbb{M}^{-1}E)$,
     \begin{align*}
        \mathbb{M}^{-1} = (I + \mathbb{M}^{-1}E) \tilde{\mathbb{M}}^{-1} = \bigg( \sum_{j=0}^\infty (\tilde{\mathbb{M}}^{-1} E)^j \bigg)\tilde{\mathbb{M}}^{-1},
    \end{align*}
    where we use $I + \mathbb{M}^{-1}E = (I - \tilde{\mathbb{M}}^{-1}E)^{-1}$.
    Therefore, we can estimate
    \begin{align*}
        \|\mathbb{M}^{-1}\|_2 \leq \bigg\|\sum_{j=0}^\infty (\tilde{\mathbb{M}}^{-1} E)^j \bigg\|_2 \|\tilde{\mathbb{M}}^{-1}\|_2 \leq \frac{\|\tilde{\mathbb{M}}^{-1}\|_2}{1-\|\tilde{\mathbb{M}}^{-1}E\|_2} \leq \frac{(1 - 2 C_4\varepsilon)\| \tilde{\mathbb{M}}^{-1} \|_2}{1 - (2 + k\|\tilde{\mathbb{M}}\|_2\|\tilde{\mathbb{M}}^{-1}\|_2 )C_4\varepsilon},
    \end{align*}
    where we use (\ref{A:E}) in the last step.
    This allows us to obtain the $L^2$-norm upper bound
    \begin{align*}
        \| r_\text{FE}^n \|_{L^2(\Omega)}^2 = |(R_\text{FE}^n)^\top \mathbb{M}^{-1} R_\text{FE}^n| \leq \|\mathbb{M}^{-1}\|_2 \| R_\text{FE}^n \|_2^2 \leq \frac{(1 - 2 C_4\varepsilon)\| \tilde{\mathbb{M}}^{-1} \|_2}{1 - (2 + k\|\tilde{\mathbb{M}}\|_2\|\tilde{\mathbb{M}}^{-1}\|_2 )C_4\varepsilon}\| R_\text{FE}^n \|_2^2 .
    \end{align*}
    We get
    \begin{align}
        \label{FE_theta}
        \| r_\text{FE}^n \|_{L^2(\Omega)}^2 &\leq \frac{(1 - 2 C_4\varepsilon)\| \tilde{\mathbb{M}}^{-1} \|_2}{1 - (2 + k\|\tilde{\mathbb{M}}\|_2\|\tilde{\mathbb{M}}^{-1}\|_2 )C_4\varepsilon} \| ( \tilde{H}_2^a - H_2) - (\tilde{\mathbb{S}}_2 - \mathbb{S}_2) C_a^{n} \|_2^2 \notag\\
        &\leq  \frac{(1 - 2 C_4\varepsilon)\| \tilde{\mathbb{M}}^{-1} \|_2}{1 - (2 + k\|\tilde{\mathbb{M}}\|_2\|\tilde{\mathbb{M}}^{-1}\|_2 )C_4\varepsilon} \bigg( \frac{C_5(h) \varepsilon}{1-2C_5(h) \varepsilon} \| \tilde{H}_2\|_2 + \| \tilde{H}_2^a -\tilde{H}_2 \|_2\notag\\
        &\quad\quad\quad\quad\quad\quad\quad\quad\quad\quad\quad\quad\quad\quad\quad + \bigg(\frac{C_6(h)\varepsilon}{1-2C_6(h)\varepsilon}\bigg) \|\tilde{\mathbb{S}}_2\|_2 \|C_a^{n}\|_2 \bigg)^2 =: (\theta_\text{FE}^n)^2,
    \end{align}
    where $C_5(h), C_6(h) > 0$ are the numbers of elementary operations needed to compute $\tilde{H}_2$ and $\tilde{\mathbb{S}}_2$, respectively, which can be obtained from the implementation and are stated in Section~\ref{num_sim}.
    
    We must assume that the machine epsilon $\varepsilon>0$ is sufficiently small such that all denominators occurring in this derivation are positive, in particular $(2 + k\|\tilde{\mathbb{M}}\|_2\|\tilde{\mathbb{M}}^{-1}\|_2 )C_4\varepsilon < 1$. 
    In our application, for $d=3$, using the eigenvalue bounds of \cite{Wathen1987}, a sufficient condition is $(74 + 2220\max_i|\omega_i|/\min_i|\omega_i| )\varepsilon < 1$,
    where $\omega_i$ denotes the patch of node $i$.
    Therefore, on shape-regular, quasi-uniform meshes the requirement on $\varepsilon$ does not tighten under refinement.
    A similar assumption is needed in Appendix~\ref{A:Linf}.

\subsection{Residual estimates}
    \label{residual_esti}
    This subsection is devoted to deriving fully computable upper bounds for the $L^2(0,T; (H^1(\mathbb{T}^d))^\prime)$-norm of the residual $R_{\tilde{\rho}_a}$ that is obtained when the reconstruction $(\tilde{\rho}_a , \tilde{c}_a)$, defined in Section~\ref{Morley}, of the approximation $(\rho_{h,a} , c_{h,a})$ as in the previous subsection, is interpreted as a solution to (\ref{res_bdrKS}).
    We describe the estimates for $d=3$ and unstructured tetrahedral meshes. The case $d=2$ and unstructured triangular meshes is analogous. Our approach is motivated by the one in \cite{Kolbe2023}, where \emph{a posteriori} residual estimates are derived for structured meshes for $d=2$. 
    
    \noindent We need further tools and notations, that will help us to estimate the residual. We denote the jump of a function $g$ across a face $F$ by
    \begin{align*}
        \llbracket g \rrbracket_F(x) := g|_K(x) - g|_L(x), \quad &\text{if } F \in F_h \text{ and } F = K \cap L.
        % \begin{cases}
        %     g|_K(x) - g|_L(x), \quad &\text{if } F \in F_h^{int} \text{ and } F = K \cap L, \\
        %     g|_K(x), \quad &\text{if } F \in F_h^{bdr} \text{ and } F \in F_K.
        % \end{cases}
    \end{align*}
    Note that all faces are interior faces when discretizing the flat torus $\mathbb{T}^d$.
    We omit the subscript $F$, whenever the respective face is clear from the context.
    Note that the jump $\llbracket g \rrbracket_F(x)$ depends on the orientation of the normal $n_F$. However, a jump term of the form $\llbracket \nabla g \cdot n_{F} \rrbracket_F(x)$ does not.

    \noindent Further, we introduce the mean over an element $K \in \mathcal{T}_h$ of an integrable function $g$,
    \begin{align*}
        \mathcal{M}_K(g) := \frac{1}{|K|}\int_K g(x) \ dx.
    \end{align*}
    For $v \in H^1(K)$, we have the Poincar\'e-Wirtinger inequality \cite[Corollary 2.3]{Bartels2015}, i.e.
    \begin{align}
        \norm{v - \mathcal{M}_K (v)}_{L^2(K)} \leq c_{P} h_K \norm{\nabla v}_{L^2(K)},
        \label{BrH}
        \tag{PW}
    \end{align}
    for constant $c_{P}=\frac{1}{\pi}$, independent of $K$ and $v$, see \cite{Payne1960}. For $v \in H^1(K)$,  $K \in \mathcal{T}_h$ and $F \in F_K$ there is the local trace inequality \cite[Lemma 4.2]{Bartels2015}, i.e.
    \begin{align}
        \norm{v}_{L^2(F)} \leq c_{\text{Tr}} \left( h_F^{-1/2} \norm{v}_{L^2(K)} + h_F^{1/2}\norm{\nabla v}_{L^2(K)} \right),
        \label{Tr}
        \tag{Tr}
    \end{align}
    for some constant $c_{\text{Tr}}>0$, independent of $F, K$ and $v$, which can be determined explicitly by investigating the proof of the estimate.
    Here, $h_F$ denotes the diameter of the face $F$. 

    % Furthermore, we need a consequence of the Bramble-Hilbert Lemma, which reads as
    % \begin{Lemma}[Consequence of Bramble-Hilbert Lemma, see Theorem 3.2 in \cite{Bartels2015}]
    %     \label{Bramble-Hilbert}
    %     Assume that $(\hat{K},\hat{P},\hat{\Sigma})$ is a reference finite element with polynomial degree $m=0$ on a triangle $\hat{T} \subset \R^d$.
    %     Let $\left(K,P,\Sigma\right)$ be the finite element that is obtained by an affine transformation $\Phi_K : K \rightarrow \hat{K}$. \\
    %     Then, for every $v \in H^1(K)$, the mean $\mathcal{M}_K (v) \in \mathcal{P}^0(K)$ satisfies
    %     \begin{align*}
    %         \norm{v - \mathcal{M}_K (v)}_{L^2(K)} \leq c_{\mathcal{M}_K} h_K \norm{v}_{H^1(K)},
    %     \end{align*}
    %     for some constant $c_{\mathcal{M}_K}(d,m,\hat{T})>0$. 
    % \end{Lemma}
    % For our purposes it suffices to apply this result for for $d=2$, which yields estimates on elements $K \in \mathcal{T}_h$.
    % In fact, a much more general statement is true, which is stated in \cite{Bartels2015}. \\
    % Furthermore, we consider the local trace inequality, which will be important later. It reads as 
    % \begin{Lemma}[Local trace inequality, see Lemma 4.2 in \cite{Bartels2015}]
    %     \label{local_trace}
    %     Let $K \in \mathcal{T}_h$ and $E \in E_K$.
    %     There exists a constant $c_{\text{Tr}} > 0$, such that for every $v \in H^1(\Omega)$ we have
    %     \begin{align*}
    %         \norm{v}_{L^2(E)} \leq c_{\text{Tr}} \left( |E|^{-1/2} \norm{v}_{L^2(K)} + |E|^{1/2}\norm{\nabla v}_{L^2(K)} \right).
    %     \end{align*}
    % \end{Lemma}
    % Now, we are prepared to tackle the residual estimates.

    \ \\ 
    \noindent Let $(\tilde{\rho}_a,\tilde{c}_a)$ be the reconstruction from Section~\ref{Morley} of the numerical approximation $(\rho_{h,a},c_{h,a})$, as in the previous subsection. 
    The residual, that we want to bound in the following, is given by
    \begin{align*}
        R_{\tilde{\rho}_a} :=  \partial_t \tilde{\rho}_a + \divergence\left( \tilde{\rho}_a \nabla \tilde{c}_a \right) - \Delta \tilde{\rho}_a.
    \end{align*}
    We set $\tilde{c}_a^n := \tilde{c}_a(t^n,\cdot)$.
    By inserting the reconstruction introduced above, as well as adding and subtracting the scheme (\ref{FV}) twice, once for each time step $t^n$ and $t^{n+1}$, we can rewrite the residual as
    \begin{align*}
        R_{\tilde{\rho}_a} &= \partial_t \tilde{\rho}_a + \divergence\left(\left(\ell_0^n(t)\tilde{\rho}_a^{n+1}(x)+\ell_1^n(t)\tilde{\rho}_a^{n}(x)\right)\left(\ell_0^n(t){\nabla \tilde{c}_a}^{n+1}(x)+\ell_1^n(t){\nabla \tilde{c}_a}^{n}(x)\right) \right) - \Delta \tilde{\rho}_a \\
        & \quad \quad \quad \quad \quad  \quad \quad \quad \quad \quad  \quad \quad \quad \quad + \ell_0^n(t) \left(  \mathcal{D}_h(\rho^{n+1}_{h,a}) - \mathcal{C}_h(\rho^{n}_{h,a},c_{h,a}^n) - \frac{\rho_{h,a}^{n+1}-\rho_{h,a}^{n}}{\Delta t^n} - r_\text{FV}^{n+1}\right) \\
        & \quad \quad \quad \quad \quad  \quad \quad \quad \quad \quad  \quad \quad \quad \quad  + \ell_1^n(t) \left(  \mathcal{D}_h(\rho^{n}_{h,a}) - \mathcal{C}_h(\rho^{n-1}_{h,a},c_{h,a}^{n-1}) - \frac{\rho_{h,a}^{n}-\rho_{h,a}^{n-1}}{\Delta t^{n-1}} - r_\text{FV}^n \right),
    \end{align*}
    for $n=1,2,\ldots,N_t-1$. The case $n=0$ will be considered separately later.
    Comparing the temporal derivatives, the diffusive terms and the convective terms of the residual to the ones of the scheme, we arrive at
    \begin{align}
        \label{residual_sum}
        R_{\tilde{\rho}_a} = \tilde{R}^D + \tilde{R}^T + \tilde{R}^C + \tilde{R}^a,
    \end{align}
    where
    \begin{align*}
        \tilde{R}^D &:=  \ell_0^n(t) \mathcal{D}_h(\rho^{n+1}_{h,a}) + \ell_1^n(t)\mathcal{D}_h(\rho^{n}_{h,a}) - \Delta \tilde{\rho}_a\\
        \tilde{R}^T \  &:= \frac{\tilde{\rho}_a^{n+1}-\tilde{\rho}_a^{n}}{\Delta t^n} - \ell_0^n(t)\frac{\rho_{h,a}^{n+1}-\rho_{h,a}^{n}}{\Delta t^n} - \ell_1^n(t) \frac{\rho_{h,a}^{n}-\rho_{h,a}^{n-1}}{\Delta t^{n-1}} \\
        \tilde{R}^C &:= \divergence\left(\left(\ell_0^n(t)\tilde{\rho}_a^{n+1}(x)+\ell_1^n(t)\tilde{\rho}_a^{n}(x)\right)\left(\ell_0^n(t){\nabla \tilde{c}_a}^{n+1}(x)+\ell_1^n(t){\nabla \tilde{c}_a}^{n}(x)\right) \right) \\
        & \quad\quad\quad\quad\quad\quad\quad\quad\quad\quad\quad\quad\quad\quad\quad\quad\quad -\ell_0^n(t)\mathcal{C}_h(\rho^n_{h,a},c_{h,a}^n) - \ell_1^n(t)\mathcal{C}_h(\rho^{n-1}_{h,a},c_{h,a}^{n-1}) \\
        \tilde{R}^a &:=  - \ell_0^n(t) r_\text{FV}^{n+1} -  \ell_1^n(t) r_\text{FV}^{n} .
    \end{align*}    
    Note that all derivatives occurring in these definitions are to be understood in the distributional sense.
    We now aim to find a bound for the term
    \begin{align*}
        \int_0^T \norm{R_{\tilde{\rho}_a}}^2_{(H^1(\mathbb{T}^d))^\prime} \ dt = \int_0^T \sup_{\substack{\phi \in H^1, \\ \norm{\phi}_{H^1} \leq 1}} \left( \int_{\mathbb{T}^d} R_{\tilde{\rho}_a} \phi \ dx  \right)^2 dt.
    \end{align*}
    We will exploit the representation (\ref{residual_sum}) by bounding each summand individually.

    \begin{Lemma} Let $\phi \in H^1(\mathbb{T}^d)$ with $\norm{\phi}_{H^1} \leq 1$. For $t \in [t^n,t^{n+1}]$ and $n = 1,2,\ldots,N_t-1$, there holds
        \label{diff-res}
        \begin{align*}
            \int_{\mathbb{T}^d} \tilde{R}^D \phi \ dx &\leq \ell_0^n(t) \theta_D^{n+1} + \ell_1^n(t) \theta_D^{n} +  \theta_{D,a},
        \end{align*}
        where
        \begin{align*}
            \theta_D^{n} &:= \sum_{K \in \mathcal{T}_h} \int_{K} \Delta \tilde{\rho}_a^{n} \left(\mathcal{M}_K(\phi) - \phi \right) \ dx + c_R \left(\sum_{F \in F_h} h_F \norm{\llbracket \nabla \tilde{\rho}_a^{n}\cdot n_{K,F} \rrbracket}_{L^2(F)}^2 \right)^{1/2},\\
            \theta_{D,a} & :=  12 C_3(h) C_{\text{S}} \varepsilon  \sum_{K \in \mathcal{T}_h} \bigg(\frac{1}{|K|} \bigg)^{1/6}
        \end{align*}
        with constant $c_R := c_{\text{Tr}} \sqrt{(d+1) (c_{\text{P}}^2 c_d^2 c_{usr}^2 + 1)}$, where for $d=2$, $c_2 = \frac12$ and for $d=3$, $c_3 = \frac{\sqrt{3}}{4}$, and $C_3(h) > 0$ as in (\ref{def:c3}).%, as well as $F \subset \partial K$.
    \end{Lemma}
    \begin{proof}
        %Let further $t \in \left[ t^n, t^{n+1} \right]$. 
        We can write
        \begin{align*}
            \int_{\mathbb{T}^d} \tilde{R}^D \phi \ dx &= \left(\int_{\mathbb{T}^d} \ell_0^n(t) \mathcal{D}_h(\rho^{n+1}_{h,a}) + \ell_1^n(t)\mathcal{D}_h(\rho^{n}_{h,a})\right)\phi + \int_{\mathbb{T}^d} \nabla \tilde{\rho}_a \cdot \nabla \phi \ dx \\
            & = \sum_{K \in \mathcal{T}_h} \bigg( \int_K \bigg(\ell_0^n(t) \frac{1}{|K|} \sum_{F \in F_K} D_{K,F}(\rho^{n+1}_{h,a})+\ell_1^n(t) \frac{1}{|K|} \sum_{F \in F_K} D_{K,F}(\rho^{n}_{h,a}) \bigg) \phi \ dx \\
            & \quad \quad \quad \quad \quad \quad \quad \quad \quad \quad \quad \quad \quad \quad \quad \   + \int_{\partial K} \nabla \tilde{\rho}_a \cdot n_{K} \phi \ dS(x) - \int_K \Delta \tilde{\rho}_a \phi \ dx \bigg) ,
        \end{align*}
        where we used integration by parts on each element $K\in \mathcal{T}_h$.
        Recalling $\tilde{\rho}_a = \ell_0^n(t) \tilde{\rho}_a^{n+1} + \ell_1^n(t) \tilde{\rho}_a^{n}$, we can reduce the remaining calculation by looking at the $\ell_0^n(t)$-terms only and applying the following machinery analogously for the $\ell_1^n(t)$-terms.
        In this spirit, we consider
        \begin{align*}
            & \ell_0^n(t) \sum_{K \in \mathcal{T}_h}  \int_K \frac{1}{|K|}\sum_{F \in F_K}D_{K,F}(\rho^{n+1}_{h,a})\phi\ dx + \left(\int_{\partial K} \nabla \tilde{\rho}_a^{n+1} \cdot n_{K} \phi \ dS(x) -\int_K \Delta \tilde{\rho}_a^{n+1} \phi \ dx \right) \\
            = & \ell_0^n(t) \underbrace{\sum_{K \in \mathcal{T}_h} \sum_{F \in F_K}D_{K,F}(\rho^{n+1}_{h,a})\mathcal{M}_K(\phi) - \int_K \Delta \tilde{\rho}_a^{n+1} \phi \ dx}_{=: \ I} \\
            & \quad \quad \quad \quad \quad \quad \quad \quad \quad \quad \quad \quad \quad \quad \quad \quad \quad \quad \quad \quad  + \ell_0^n(t)\underbrace{\sum_{K \in \mathcal{T}_h} \int_{\partial K} \nabla \tilde{\rho}_a^{n+1} \cdot n_{K} \phi \ dS(x)}_{=: \ II}.
        \end{align*}
        Using (\ref{perturbed:diff}) and the divergence theorem, we can rewrite $I$ as
        \begin{align*}
            I %&=  \ell_0^n(t) \sum_{K \in \mathcal{T}_h} \sum_{E \in E_K}\int_E \nabla \tilde{\rho}^{n+1} \cdot n_{K,E} \ dS(x) \mathcal{M}_K(\phi) - \int_K \Delta \tilde{\rho}^{n+1} \phi \ dx \\
            &= \sum_{K \in \mathcal{T}_h} \int_{\partial K} \nabla \tilde{\rho}_a^{n+1} \cdot n_{K} \ dS(x)\mathcal{M}_K(\phi) - \int_K \Delta \tilde{\rho}_a^{n+1} \phi \ dx \\
            & \quad\quad\quad\quad\quad\quad\quad\quad\quad\quad\quad\quad\quad\quad\quad\  + \sum_{F \in F_K} (\mathcal{D}_{K,F}(\rho_{h,a}^n)e_a^F + |F| (\nabla q_0 \cdot n_{K,F}) e_b^F) \mathcal{M}_K(\phi)  \\
            %&= \ell_0^n(t) \sum_{K \in \mathcal{T}_h} \int_{K} \Delta \tilde{\rho}^{n+1} \ dx \mathcal{M}_K(\phi) - \int_K \Delta \tilde{\rho}^{n+1} \phi \ dx \\
            &=  \sum_{K \in \mathcal{T}_h} \int_{K} \Delta \tilde{\rho}_a^{n+1} \left(\mathcal{M}_K(\phi) - \phi \right) dx + \sum_{F \in F_K} (\mathcal{D}_{K,F}(\rho_{h,a}^n)e_a^F + |F| (\nabla q_0 \cdot n_{K,F}) e_b^F) \mathcal{M}_K(\phi).
        \end{align*}
        Using (\ref{def:c3}), we bound 
        \begin{align*}
            & \sum_{K \in \mathcal{T}_h} \sum_{F \in F_K} (\mathcal{D}_{K,F}(\rho_{h,a}^n)e_a^F + |F| (\nabla q_0 \cdot n_{K,F}) e_b^F) \mathcal{M}_K(\phi) \\
            &\quad\quad\quad\quad\quad\quad\quad\quad\leq  4C_3(h) \varepsilon \sum_{K \in \mathcal{T}_h}  \frac{1}{|K|} \int_K \phi \ dx \leq  4C_3(h) \varepsilon \sum_{K \in \mathcal{T}_h}  \frac{1}{|K|} \| 1 \|_{L^{6/5}(K)} \| \phi \|_{L^6(K)} \\
            &\quad\quad\quad\quad\quad\quad\quad\quad\leq  4C_3(h) \varepsilon \sum_{K \in \mathcal{T}_h}  \bigg(\frac{1}{|K|}\bigg)^{1/6}  \| \phi \|_{L^6(K)} \leq 4 C_3(h) C_{\text{S}} \varepsilon  \sum_{K \in \mathcal{T}_h} \bigg(\frac{1}{|K|} \bigg)^{1/6} \| \phi \|_{H^1(\Omega)}.
        \end{align*}
        The other part of $I$ will be the diffusive contribution to the element-wise residual and explicitly enters the desired bound.
        %This term is the diffusive part of an element-wise residual that we will obtain in the following as the combination of terms that will result from the time residual and the convective residual.
        In order to estimate $II$, we observe that (\ref{perturbed:diff}) and the fact that $D_{K,F}(\rho_{h,a}^{n+1}) = - D_{L,F}(\rho_{h,a}^{n+1})$ for $F = K \cap L$ with $K,L \in \mathcal{T}_h$, imply
        \begin{align*}
            \int_F \llbracket \nabla \tilde{\rho}_a^{n+1} \cdot n_F \rrbracket \ dS(x) = 2(\mathcal{D}_{K,F}(\rho_{h,a}^n)e_a^F + |F| (\nabla q_0 \cdot n_{K,F}) e_b^F), \quad \text{for all } F \in F_h.
        \end{align*}
        % \begin{align*}
        %     \sum_{K \in \mathcal{T}_h} \int_{\partial K} \nabla \tilde{\rho}^{n+1} \cdot n_{K} \ dS(x) = \sum_{K \in \mathcal{T}_h} \sum_{F\in F_K} D_{K,F}(\rho^{n+1}) = 0.
        % \end{align*}
        % This is due to the fact that each interior face $F = K \cap L$ occurs twice in the above sum, once involving the trace of element $K$ and once involving the trace of element $L$. 
        % Then,  yields the cancellation of all terms. 
        % For the boundary face $F\in F^{\text{bdr}}_h$, that only occur once in the sum, the diffusive flux vanishes due to the homogeneous Neumann boundary condition.
        % Additionally, it holds
        % \begin{align*}
        %     \sum_{K \in \mathcal{T}_h} \int_{\partial K} \nabla \tilde{\rho}^{n+1} \cdot n_{K} \ dS(x) = \sum_{F \in F_h} \int_F \llbracket \nabla \tilde{\rho}^{n+1} \cdot n_F \rrbracket \ dS(x)
        % \end{align*}
        % and thus 
        % \begin{align*}
        %     \sum_{F \in F_h}\int_F \llbracket \nabla \tilde{\rho}^{n+1} \cdot n_F \rrbracket \ dS(x) = 0, \quad \text{for all } F \in F_h.
        % \end{align*}
        % and so
        % \begin{align*}
        %     \sum_{\substack{F \in F_h, \\ F \subset \partial K}} \int_F \llbracket \nabla \tilde{\rho}^{n+1} \cdot n_F \rrbracket \mathcal{M}_K(\phi) \ dS(x) = 0,
        % \end{align*}
        % for any function $\phi$ that is integrable on $K$.
        \noindent With this at hand, we can rewrite 
        \begin{align*}
            II %&= \ell_0^n(t)\sum_{K \in \mathcal{T}_h} \int_{\partial K} \nabla \tilde{\rho}^{n+1} \cdot n_{K} \phi \ dS(x) \\
            %&= \ell_0^n(t)\sum_{E \in E_h} \int_{E} \llbracket \nabla \tilde{\rho}^{n+1} \cdot n_E \rrbracket \phi \ dS(x) \\
            &= \ell_0^n(t)\underbrace{\sum_{F \in F_h} \int_{F} \llbracket \nabla \tilde{\rho}_a^{n+1} \cdot n_F \rrbracket \left(\phi - \mathcal{M}_K(\phi) \right) \ dS(x)}_{ =: \  III}\\
            & \quad\quad\quad\quad\quad\quad\quad\quad\quad\quad\quad\quad\quad\quad\quad + \ell_0^n(t)2 \sum_{F \in F_h} (\mathcal{D}_{K,F}(\rho_{h,a}^n)e_a^F + |F| (\nabla q_0 \cdot n_{K,F}) e_b^F)\mathcal{M}_K(\phi),
        \end{align*}
        for any $K \in \mathcal{T}_h$ such that $F \subset \partial K$.
        We use Hölder's inequality, the Cauchy-Schwarz inequality for sums and (\ref{Tr}) to get
        \begin{align*}
            III %&\leq \sum_{\substack{E \in E_h, \\ E \subset \partial K}} \norm{\llbracket \nabla \tilde{\rho}^{n+1} \cdot n_E \rrbracket}_{L^2(E)} \norm{\phi - \mathcal{M}_K(\phi)}_{L^2(E)} \\
            %&= \sum_{\substack{E \in E_h, \\ E \subset \partial K}} |E|^{1/2} \norm{\llbracket \nabla \tilde{\rho}^{n+1} \cdot n_E\rrbracket}_{L^2(E)} |E|^{-1/2} \norm{\phi - \mathcal{M}_K(\phi)}_{L^2(E)} \\
            &\leq \left( \sum_{F \in F_h} h_F \norm{\llbracket \nabla \tilde{\rho}_a^{n+1}\cdot n_F \rrbracket}_{L^2(F)}^2 \right)^{1/2} \left( \sum_{F \in F_h} h_F^{-1} \norm{\phi - \mathcal{M}_K(\phi)}_{L^2(F)}^2 \right)^{1/2} \\
            &\leq \left( \sum_{F \in F_h} h_F \norm{\llbracket \nabla \tilde{\rho}_a^{n+1}\cdot n_F \rrbracket}_{L^2(F)}^2 \right)^{1/2} \\
            & \quad\quad\quad\quad\quad\quad\quad\quad\quad \times \left( \sum_{F \in F_h} c_{\text{Tr}}^2 h_F^{-2} \norm{\phi - \mathcal{M}_K(\phi)}_{L^2(K)}^2 + c_{\text{Tr}}^2\norm{\nabla \phi}_{L^2(K)}^2 \right)^{1/2}.
        \end{align*}

        \noindent Using that our mesh sequence is uniformly shape regular, $r_K \leq c_d h_F$ for any face $F \subset \partial K$ and using (\ref{BrH}), we can further estimate
        \begin{align*}
            III &\leq \left( \sum_{F \in F_h} h_F \norm{\llbracket \nabla \tilde{\rho}_a^{n+1}\cdot n_F \rrbracket}_{L^2(F)}^2 \right)^{1/2} \left( c_{\text{Tr}}^2 \sum_{F \in F_h} \left(c_{\text{P}}^2  \frac{h_K^2}{h_F^{2}} +1 \right) |\phi|_{H^1(K)}^2 \right)^{1/2} \\
            &\leq \left( \sum_{F \in F_h} h_F \norm{\llbracket \nabla \tilde{\rho}_a^{n+1} \cdot n_F\rrbracket}_{L^2(F)}^2 \right)^{1/2} \left((d+1) c_{\text{Tr}}^2 \sum_{K \in \mathcal{T}_h} \left(c_{\text{P}}^2 c_d^2 c_{usr}^2 + 1 \right) |\phi|_{H^1(K)}^2\right)^{1/2} \\
            &\leq c_R \left( \sum_{F \in F_h} h_F \norm{\llbracket \nabla \tilde{\rho}_a^{n+1}\cdot n_F \rrbracket}_{L^2(F)}^2 \right)^{1/2} \norm{\phi}_{H^1}.
            %&\leq c_R \left( \sum_{K \in \mathcal{T}_h} \sum_{E \in E_K} |E| \norm{\llbracket \nabla \tilde{\rho}^{n+1}\cdot n_{K,E} \rrbracket}_{L^2(E)}^2 \right)^{1/2},
        \end{align*}
        Here, the factor $d+1$ is due to the fact that tetrahedra have four faces and triangles three edges.
        Thus, we obtain 
        \begin{align*}
            II \leq c_R \left( \sum_{F \in F_h} h_F \norm{\llbracket \nabla \tilde{\rho}_a^{n+1}\cdot n_F \rrbracket}_{L^2(F)}^2 \right)^{1/2} + 8 C_3(h) C_{\text{S}} \varepsilon  \sum_{K \in \mathcal{T}_h} \bigg(\frac{1}{|K|} \bigg)^{1/6},
        \end{align*}
        where the bound for the algebraic error term works analogous to before.
        This implies the desired upper bound.

        % \noindent Using analogous arguments for the $\ell_1^n(t)$-terms, we get in total \textcolor{red}{really need this?}
        % \begin{align*}
        %     \int_\Omega \tilde{R}^D \phi \ dx &\leq \ell_0^n(t) \sum_{K \in \mathcal{T}_h} \int_{K} \Delta \tilde{\rho}^{n+1} \left(\mathcal{M}_K(\phi) - \phi \right) \ dx \\
        %     & \quad + c_R \left( \sum_{K \in \mathcal{T}_h} \sum_{E \in E_K} |E| \norm{\llbracket \nabla \tilde{\rho}^{n+1}\cdot n_{K,E} \rrbracket}_{L^2(E)}^2 \right)^{1/2} \\
        %     & \quad + \ell_1^n(t) \sum_{K \in \mathcal{T}_h} \int_{K} \Delta \tilde{\rho}^{n} \left(\mathcal{M}_K(\phi) - \phi \right) \ dx\\
        %     & \quad + c_R \left( \sum_{K \in \mathcal{T}_h} \sum_{E \in E_K} |E| \norm{\llbracket \nabla \tilde{\rho}^{n}\cdot n_{K,E} \rrbracket}_{L^2(E)}^2 \right)^{1/2},
        % \end{align*}
        % for $t \in \left[ t^n, t^{n+1} \right]$. 
    \end{proof}

    \begin{Lemma}
        \label{time-res}
        Let $\phi \in H^1(\mathbb{T}^d)$ with $\norm{\phi}_{H^1} \leq 1$. For $t \in [t^n,t^{n+1}]$ and $n = 1,2,\ldots,N_t-1$, there holds
        \begin{align*}
            \int_{\mathbb{T}^d} \tilde{R}^T \phi \ dx \leq& \left( \ell_0^n(t) + \ell_1^n(t) \right) \sum_{K \in \mathcal{T}_h}\int_K \frac{\tilde{\rho}_a^{n+1}-\tilde{\rho}_a^{n}}{\Delta t^n} \left(\phi-\mathcal{M}_K(\phi)\right) \ dx \\ 
            & \quad + \bigg(\sum_{K \in \mathcal{T}_h} \|\frac{\tilde{\rho}_a^{n+1}-\tilde{\rho}_a^{n}}{\Delta t^n} - \frac{\rho_{K,a}^{n+1}-\rho_{K,a}^{n}}{\Delta t^n}\|_{L^2(K)}^2\bigg)^{1/2} \\
            & \quad + \ell_1^n(t) \bigg(\sum_{K \in \mathcal{T}_h} |K| \bigg(\frac{\rho_{K,a}^{n+1}-\rho_{K,a}^{n}}{\Delta t^n} - \frac{\rho_{K,a}^{n}-\rho_{K,a}^{n-1}}{\Delta t^{n-1}}\bigg)^2 \bigg)^{1/2}.
        \end{align*}
    \end{Lemma}
    \begin{proof}
        We can rewrite the temporal part of the residual as
        \begin{align*}
            \tilde{R}^T %&= \frac{\tilde{\rho}^{n+1}-\tilde{\rho}^{n}}{\Delta t^n} - \ell_0^n(t)\frac{\rho_h^{n+1}-\rho_h^{n}}{\Delta t^n} - \ell_1^n(t) \frac{\rho_h^{n}-\rho_h^{n-1}}{\Delta t^{n-1}} \\
            %& = \frac{\tilde{\rho}^{n+1}-\tilde{\rho}^{n}}{\Delta t^n} - \ell_0^n(t)\frac{\rho_h^{n+1}-\rho_h^{n}}{\Delta t^n} - \ell_1^n(t)\frac{\rho_h^{n+1}-\rho_h^{n}}{\Delta t^n} + \ell_1^n(t)\frac{\rho_h^{n+1}-\rho_h^{n}}{\Delta t^n} - \ell_1^n(t) \frac{\rho_h^{n}-\rho_h^{n-1}}{\Delta t^{n-1}} \\
            %& = \frac{\tilde{\rho}^{n+1}-\tilde{\rho}^{n}}{\Delta t^n} - \left(\ell_0^n(t) + \ell_1^n(t)\right)\frac{\rho_h^{n+1}-\rho_h^{n}}{\Delta t^n} + \ell_1^n(t) \left(\frac{\rho_h^{n+1}-\rho_h^{n}}{\Delta t^n} - \frac{\rho_h^{n}-\rho_h^{n-1}}{\Delta t^{n-1}}\right), \\
            & = \frac{\tilde{\rho}_a^{n+1}-\tilde{\rho}_a^{n}}{\Delta t^n} - \frac{\rho_{h,a}^{n+1}-\rho_{h,a}^{n}}{\Delta t^n} + \ell_1^n(t) \bigg(\frac{\rho_{h,a}^{n+1}-\rho_{h,a}^{n}}{\Delta t^n} - \frac{\rho_{h,a}^{n}-\rho_{h,a}^{n-1}}{\Delta t^{n-1}}\bigg),
        \end{align*}
        due to $\ell_0^n(t) + \ell_1^n(t) = 1$. %= \frac{t-t^n}{\Delta t^n} + \frac{t^{n+1}-t}{\Delta t^n} 
        We consider 
        \begin{align*}
            \int_{\mathbb{T}^d} \tilde{R}^T \phi \ dx &=  \underbrace{ \sum_{K \in \mathcal{T}_h}\int_K \frac{\tilde{\rho}_a^{n+1}-\tilde{\rho}_a^{n}}{\Delta t^n} \phi - \frac{\rho_{K,a}^{n+1}-\rho_{K,a}^{n}}{\Delta t^n} \phi \ dx}_{=:\ I} \\
            & \quad\quad\quad\quad\quad\quad\quad\quad\quad\quad\quad\quad + \underbrace{ \sum_{K \in \mathcal{T}_h}\int_K \ell_1^n(t) \bigg(\frac{\rho_{K,a}^{n+1}-\rho_{K,a}^{n}}{\Delta t^n} - \frac{\rho_{K,a}^{n}-\rho_{K,a}^{n-1}}{\Delta t^{n-1}}\bigg) \phi \ dx}_{=:\ II} .
        \end{align*}
        Rewriting $I$ gives
        \begin{align*}
            I &=  \sum_{K \in \mathcal{T}_h}\int_K \frac{\tilde{\rho}_a^{n+1}-\tilde{\rho}_a^{n}}{\Delta t^n} \left(\phi-\mathcal{M}_K(\phi)\right) \ dx + \sum_{K \in \mathcal{T}_h}\int_K \frac{\tilde{\rho}_a^{n+1}-\tilde{\rho}_a^{n}}{\Delta t^n}\mathcal{M}_K(\phi) - \frac{\rho_{K,a}^{n+1}-\rho_{K,a}^{n}}{\Delta t^n} \phi \ dx.
        \end{align*}
        %The first term will be the time contribution in the element-wise residual. 
        % As a preparation to further rewrite this, we observe that
        % \begin{align*}
        %     \int_K \frac{\rho_K^{n+1}-\rho_K^{n}}{\Delta t^n} \phi \ dx = |K| \frac{\rho_K^{n+1}-\rho_K^{n}}{\Delta t^n} \mathcal{M}_K(\phi)  = \int_K \frac{\rho_K^{n+1}-\rho_K^{n}}{\Delta t^n} \mathcal{M}_K(\phi) \ dx,
        % \end{align*}
        % as the finite volume solution is constant on each element $K \in \mathcal{T}_h$.
        The first term will be the time contribution to the element-wise residual and explicitly enters the desired bound. Further, using  Hölder's inequality with the constant one function and the Cauchy-Schwarz inequality as before, we find
        \begin{align*}
            & \sum_{K \in \mathcal{T}_h}\int_K \frac{\tilde{\rho}_a^{n+1}-\tilde{\rho}_a^{n}}{\Delta t^n}\mathcal{M}_K(\phi) - \frac{\rho_{K,a}^{n+1}-\rho_{K,a}^{n}}{\Delta t^n} \phi \ dx \\
            =& \sum_{K \in \mathcal{T}_h}\int_K \bigg( \frac{\tilde{\rho}_a^{n+1}-\tilde{\rho}_a^{n}}{\Delta t^n} - \frac{\rho_{K,a}^{n+1}-\rho_{K,a}^{n}}{\Delta t^n}\bigg) \mathcal{M}_K(\phi) \ dx \\
            \leq& \sum_{K \in \mathcal{T}_h}\norm{\frac{\tilde{\rho}_a^{n+1}-\tilde{\rho}_a^{n}}{\Delta t^n} - \frac{\rho_{K,a}^{n+1}-\rho_{K,a}^{n}}{\Delta t^n}}_{L^2(K)}  \norm{\mathcal{M}_K(\phi)}_{L^2(K)} \\
            %\leq& \sum_{K \in \mathcal{T}_h}\int_K \left( \frac{\tilde{\rho}^{n+1}-\tilde{\rho}^{n}}{\Delta t^n} - \frac{\rho_K^{n+1}-\rho_K^{n}}{\Delta t^n}\right) \ dx  \left(\frac{1}{|K|}\right)^{\frac{1}{2}} \norm{\phi}_{L^2(K)} \\
            % \leq& \sum_{K \in \mathcal{T}_h} \norm{\frac{\tilde{\rho}^{n+1}-\tilde{\rho}^{n}}{\Delta t^n} - \frac{\rho_K^{n+1}-\rho_K^{n}}{\Delta t^n}}_{L^2(K)} \norm{1}_{L^2(K)}  \left(\frac{1}{|K|}\right)^{\frac{1}{2}} \norm{\phi}_{L^2(K)} \\
            % \leq& \sum_{K \in \mathcal{T}_h} \norm{\frac{\tilde{\rho}^{n+1}-\tilde{\rho}^{n}}{\Delta t^n} - \frac{\rho_K^{n+1}-\rho_K^{n}}{\Delta t^n}}_{L^2(K)} \norm{\phi}_{L^2(K)}\\
            \leq& \bigg(\sum_{K \in \mathcal{T}_h} \norm{\frac{\tilde{\rho}_a^{n+1}-\tilde{\rho}_a^{n}}{\Delta t^n} - \frac{\rho_{K,a}^{n+1}-\rho_{K,a}^{n}}{\Delta t^n}}_{L^2(K)}^2\bigg)^{1/2} \norm{\phi}_{L^2}.
        \end{align*}
        \noindent In order to deal with $II$, we again apply Hölder's inequality and the Cauchy-Schwarz inequality for sums, which yields
        \begin{align*}
            II %&\leq \sum_{K \in \mathcal{T}_h} \norm{\frac{\rho_K^{n+1}-\rho_K^{n}}{\Delta t^n} - \frac{\rho_K^{n}-\rho_K^{n-1}}{\Delta t^{n-1}}}_{L^2(K)} \norm{\phi}_{L^2(K)} \\
            %&\leq \left(\sum_{K \in \mathcal{T}_h} |K| \left(\frac{\rho_K^{n+1}-\rho_K^{n}}{\Delta t^n} - \frac{\rho_K^{n}-\rho_K^{n-1}}{\Delta t^{n-1}}\right)^2 \right)^{1/2} \left( \sum_{K \in \mathcal{T}_h} \norm{\phi}_{L^2(K)}^2\right)^{1/2}  \\
            &\leq \ell_1^n(t) \bigg(\sum_{K \in \mathcal{T}_h} |K| \bigg(\frac{\rho_{K,a}^{n+1}-\rho_{K,a}^{n}}{\Delta t^n} - \frac{\rho_{K,a}^{n}-\rho_{K,a}^{n-1}}{\Delta t^{n-1}}\bigg)^2 \bigg)^{1/2} \norm{\phi}_{L^2}.
            %&\leq \left(\sum_{K \in \mathcal{T}_h} |K| \left(\frac{\rho_K^{n+1}-\rho_K^{n}}{\Delta t^n} - \frac{\rho_K^{n}-\rho_K^{n-1}}{\Delta t^{n-1}}\right)^2 \right)^{1/2} \norm{\phi}_{H^1(\Omega)} \\
            %&\leq \left(\sum_{K \in \mathcal{T}_h} |K| \left(\frac{\rho_K^{n+1}-\rho_K^{n}}{\Delta t^n} - \frac{\rho_K^{n}-\rho_K^{n-1}}{\Delta t^{n-1}}\right)^2 \right)^{1/2}.
        \end{align*}
        % All estimates on the time residual together yield \textcolor{red}{really need this?}
        % \begin{align*}
        %     \int_\Omega \tilde{R}^t \phi \ dx \leq& \left( \ell_0^n(t) + \ell_1^n(t) \right) \sum_{K \in \mathcal{T}_h}\int_K \frac{\tilde{\rho}^{n+1}-\tilde{\rho}^{n}}{\Delta t^n} \left(\phi-\mathcal{M}_K(\phi)\right) \ dx \\ 
        %     & \quad + \left(\sum_{K \in \mathcal{T}_h} \norm{\frac{\tilde{\rho}^{n+1}-\tilde{\rho}^{n}}{\Delta t^n} - \frac{\rho_K^{n+1}-\rho_K^{n}}{\Delta t^n}}_{L^2(K)}^2\right)^{1/2} \\
        %     & \quad + \left(\sum_{K \in \mathcal{T}_h} |K| \left(\frac{\rho_K^{n+1}-\rho_K^{n}}{\Delta t^n} - \frac{\rho_K^{n}-\rho_K^{n-1}}{\Delta t^{n-1}}\right)^2 \right)^{1/2},
        % \end{align*}
        % for $t \in \left[ t^n, t^{n+1} \right]$.
    \end{proof}
    
    \noindent Before tackling the convective part of the residual, we need some further preparation.
    As the convective term $\divergence(\tilde{\rho}_a^{n+1}\nabla c_{h,a}^n)$ is only well-defined in a point-wise sense on dual elements and as the Laplacian of the Morley interpolant $\Delta \tilde{\rho}_a^{n+1}$ is only well-defined on primal elements, we introduce the 'intersected' mesh of the primal mesh $\mathcal{T}_h$ and the dual mesh $\mathcal{S}_{\mathcal{T}_h}$, i.e.
    \begin{align*}
        \mathcal{C}_{\mathcal{S}_{\mathcal{T}_h}}^{\mathcal{T}_h} := \left\{ T \ : \ T = K\cap L, \ K \in \mathcal{T}_h, \ L \in \mathcal{S}_{\mathcal{T}_h} \text{ and } |T|_3>0  \right\}.
    \end{align*}
    In words, the 'intersected' mesh consists of all intersections of primal and dual elements with positive three-dimensional Lebesgue measure.
    Note that the convective term $\divergence(\tilde{\rho}_a^{n+1} \nabla c_{h,a}^n)$ as well as the Laplacian $\Delta \tilde{\rho}_a^{n+1}$ are well-defined on all 'intersected' elements.
    
    \noindent We denote the set of all faces of the 'intersected' mesh by $C_h$. 
    Given a tetrahedron $T \in \mathcal{C}_{\mathcal{S}_{\mathcal{T}_h}}^{\mathcal{T}_h}$, we denote the set consisting of its four faces by $C_T$.
    Further, we denote the unique element $K\in \mathcal{T}_h$ with $T \subset K$ by $K_T$.
  
    \noindent Furthermore, we want to differentiate between 'intersected' faces $C \in C_h$ which are part of a primal face $F \in F_h$ and those which are part of a dual face $S \in S_h$.
    In this spirit, we define
    \begin{align*}
        C_h^{F_h} &:= \left\{C \in C_h : \exists F \in F_h, \text{ such that } C \subset F\right\}, \\ C_h^{S_h} &:= \left\{C \in C_h : \exists S \in S_h, \text{ such that } C \subset S\right\}.
    \end{align*}
    Note that we have $C_h =  C_h^{F_h} \stackrel{.}{\cup} C_h^{S_h}$.
    %Analogously, for each $T \in \mathcal{C}_{\mathcal{S}_{\mathcal{T}_h}}^{\mathcal{T}_h}$, we define $C_T^{F_h^{int}}$ and $C_T^{S_h}$ by replacing $C_h$ by $C_T$ in the definitions above.
    This will become relevant in the sequel.
    \begin{Lemma}
        \label{conv-res}
        Let $\phi \in H^1(\mathbb{T}^d)$ with $\norm{\phi}_{H^1} \leq 1$. For $t \in [t^n,t^{n+1}]$ and $n = 1,2,\ldots,N_t-1$, there holds
        \begin{align*}
            \int_{\mathbb{T}^d} \tilde{R}^C \phi \ dx \leq& \   \ell_0^n(t)\theta_C^{n+1} + \ell_1^n(t)\theta_C^n + \sum_{K \in \mathcal{T}_h} \theta_{K}^n,
        \end{align*}
        where 
        \begin{align*}
            \theta_{K}^n &:= C_{\text{ell}}\big(\left(\norm{\tilde{\rho}_a^{n+1}}_{L^\infty} + \norm{\tilde{\rho}_a^{n} - \tilde{\rho}_a^{n+1}}_{L^\infty} \right)\norm{\tilde{\rho}_a^{n} - \tilde{\rho}_a^{n+1}}_{L^2(K)} + \norm{\tilde{\rho}^{n}_a}_{L^\infty}  \norm{\tilde{\rho}_a^{n-1} - \tilde{\rho}_a^{n}}_{L^2(K)}\big), \\
            \theta_C^n &:= \sum_{T \in \mathcal{C}_{\mathcal{S}_{\mathcal{T}_h}}^{\mathcal{T}_h}}  \int_T \divergence \left(\tilde{\rho}_a^{n} \nabla c_{h,a}^{n-1} \right) \big(\phi  - \mathcal{M}_{K_T}(\phi)\big) \ dx  \\
            & \quad\quad\quad + \tilde{c}_R \bigg( \sum_{T \in \mathcal{C}_{\mathcal{S}_{\mathcal{T}_h}}^{\mathcal{T}_h} }\sum_{S \in C_T^{S_h}} h_S \norm{\llbracket \nabla c_{h,a}^{n-1} \cdot n_{T,S} \rrbracket \tilde{\rho}_a^{n} }_{L^2(S)}^2 \bigg)^{1/2} + \norm{\tilde{\rho}_a^{n}}_{L^\infty} (\eta_{\Omega}^{n-1} + \tilde{C}^{\text{sz}} \theta_\text{FE}^{n-1})  \\
            & \quad\quad\quad + \tilde{c}_R \bigg( \sum_{T \in \mathcal{C}_{\mathcal{S}_{\mathcal{T}_h}}^{\mathcal{T}_h} }\sum_{F \in C_T^{F_h}} h_F \norm{\big(\nabla c_{h,a}^{n-1} \cdot n_{T,F}\big) \big(\tilde{\rho}_a^{n} - \frac{1}{|F|} \mathcal{C}_F(\rho_{h,a}^{n-1})\big)}_{L^2(F)}^2  \bigg)^{1/2}, \\
            \eta_\Omega^{n} &:= \tilde{C}^{\text{sz}}\left( \sum_{\substack{T \in \mathcal{S}_{\mathcal{T}_h}}} \eta_{T,n}^2 + \sum_{\substack{T \in \mathcal{S}_{\mathcal{T}_h}}} h_{T}^2 \norm{\tilde{\rho}_a^n - \Pi_1\tilde{\rho}_a^n}_{L^2(T)}^2 + \frac43 \sum_{\substack{T \in \mathcal{S}_{\mathcal{T}_h}}} \norm{\rho_{h,a}^n - \tilde{\rho}_a^n}_{L^2(T)}^2\right)^{1/2}, \\
            \eta_{T,n}^2 &:=  h_T^2\norm{c_{h,a}^n - \Pi_1 \tilde{\rho}_a^n}_{L^2(T)}^2 + \frac{1}{2}\sum_{F \in S_T} h_F \norm{\llbracket \nabla c_{h,a}^n \cdot n_F\rrbracket}_{L^2(F)}^2,
        \end{align*}
        with $\Pi_1$ as defined in Subsection~\ref{FE-scheme-section} and constant $\tilde{c}_R := c_{\text{Tr}} \sqrt{d(d+1)(c_{\text{P}}^2 c_d^2 c_{usr}^2 +1)} $, where for $d=2$, $c_2 = \frac12$ and for $d=3$, $c_3 = \frac{\sqrt{3}}{4}$, $\theta_\text{FE}^n$ defined as in (\ref{FE_theta}) and $\tilde{C}^{\text{sz}} > 0$, which is the stability constant of the Scott-Zhang interpolant, which one can compute explicitly by investigating the proof, see \cite[Theorem 4.6]{Bartels2015}.
    \end{Lemma}
    \begin{proof}
        For any $\phi \in H^1$, we observe
        \begin{align*}
            \int_{\mathbb{T}^d} \tilde{R}^C \phi \ dx =& - \int_{\mathbb{T}^d} \left(\ell_0^n(t)\tilde{\rho}_a^{n+1}(x)+\ell_1^n(t)\tilde{\rho}_a^{n}(x)\right)\left(\ell_0^n(t)\nabla {\tilde{c}_a}^{n+1}(x)+\ell_1^n(t)\nabla {\tilde{c}_a}^{n}(x)\right) \cdot \nabla \phi \ dx \\
            & - \int_{\mathbb{T}^d} \left( \ell_0^n(t)\mathcal{C}_h(\rho^{n}_{h,a},c_{h,a}^n) + \ell_1^n(t)\mathcal{C}_h(\rho^{n-1}_{h,a},c_{h,a}^{n-1}) \right) \phi \ dx \\
            =& \ \ I + II,
        \end{align*}
        % Adding a zero and rearranging gives \textcolor{red}{maybe dont need this}
        % \begin{align*}
        %     \int_\Omega \tilde{R}^F \phi \ dx =& - \int_\Omega \left(\ell_0^n(t)\tilde{\rho}^{n+1}(x)+\ell_1^n(t)\tilde{\rho}^{n}(x)\right)\left(\ell_0^n(t){\nabla \tilde{c}}^{n+1}(x)+\ell_1^n(t){\nabla \tilde{c}}^{n}(x)\right) \cdot \nabla \phi \ dx \\
        %     & + \int_\Omega \left( \left( \ell_0^n(t) \tilde{\rho}^{n+1} \nabla \tilde{c}^n + \ell_1^n \tilde{\rho}^{n} \nabla \tilde{c}^{n-1} \right) - \left( \ell_0^n(t) \tilde{\rho}^{n+1} \nabla \tilde{c}^n + \ell_1^n \tilde{\rho}^{n} \nabla \tilde{c}^{n-1} \right) \right) \cdot \nabla \phi \ dx \\
        %     & - \int_\Omega \left( \ell_0^n(t)\mathcal{F}_h(\rho^n_h,\rho^{n+1}_h) + \ell_1^n(t)\mathcal{F}_h(\rho^{n-1}_h,\rho^{n}_h) \right) \phi \ dx \\
        %     =& \ \ I + II
        % \end{align*}
        where 
        \begin{align*}
            I :=& \int_{\mathbb{T}^d} \bigg( \left( \ell_0^n(t) \tilde{\rho}_a^{n+1} \nabla \tilde{c}_a^n + \ell_1^n(t) \tilde{\rho}_a^{n} \nabla \tilde{c}_a^{n-1} \right) \\
            & \quad - \left(\ell_0^n(t)\tilde{\rho}_a^{n+1}+\ell_1^n(t)\tilde{\rho}_a^{n}\right)\left(\ell_0^n(t)\nabla {\tilde{c}_a}^{n+1}+\ell_1^n(t)\nabla {\tilde{c}_a}^{n}\right)\bigg) \cdot \nabla \phi \ dx, \quad \text{and} \\
            II :=&  - \int_{\mathbb{T}^d} \left( \ell_0^n(t) \tilde{\rho}_a^{n+1} \nabla \tilde{c}_a^n + \ell_1^n \tilde{\rho}_a^{n} \nabla \tilde{c}_a^{n-1} \right) \cdot \nabla \phi  + \left( \ell_0^n(t)\mathcal{C}_h(\rho^{n}_{h,a},c_{h,a}^n) + \ell_1^n(t)\mathcal{C}_h(\rho^{n-1}_{h,a},c_{h,a}^{n-1}) \right) \phi \ dx.
        \end{align*}
        As a preparation to bound $I$, we follow ideas from \cite{Kolbe2023} and observe that for $a_1,a_2,b_1,b_2 \in \mathbb{R}$
        \begin{align}
            \label{a1b1l0}
            \left(\ell_0^n a_1 + \ell_1^n a_2 \right)\left(\ell_0^n b_1 + \ell_1^n b_2 \right) - \ell_0^n a_1b_1 - \ell_1^n a_2b_2 = \left(a_1-a_2\right)\left(b_1-b_2\right)\left(\tilde{t}^2-\tilde{t}\right),
        \end{align}
        for $t \in \left[ t^n, t^{n+1} \right]$, where $\tilde{t} := \frac{t-t^n}{\Delta t^n}$.
        %One can see this by factoring out both sides.

        \noindent In order to exploit this observation, we further rewrite 
        \begin{align*}
            I =& \int_{\mathbb{T}^d} \big( \ell_0^n(t) \left(\tilde{\rho}_a^{n+1} \nabla \tilde{c}_a^n - \tilde{\rho}_a^{n+1} \nabla \tilde{c}_a^{n+1} \right) + \ell_1^n(t) \left(\tilde{\rho}_a^{n} \nabla \tilde{c}_a^{n-1} - \tilde{\rho}_a^{n} \nabla \tilde{c}_a^{n} \right) \big) \cdot \nabla \phi \ dx \\
            & + \int_{\mathbb{T}^d} \big( \ell_0^n(t)\tilde{\rho}_a^{n+1} \nabla \tilde{c}_a^{n+1} + \ell_1^n(t)\tilde{\rho}_a^{n} \nabla \tilde{c}_a^{n} \\
            & \quad\quad\quad\quad - \left(\ell_0^n(t)\tilde{\rho}_a^{n+1}+\ell_1^n(t)\tilde{\rho}_a^{n}\right)\left(\ell_0^n(t)\nabla {\tilde{c}_a}^{n+1}+\ell_1^n(t)\nabla {\tilde{c}_a}^{n}\right) \big) \cdot \nabla \phi \ dx.
        \end{align*}
        Applying (\ref{a1b1l0}) to the last two lines, using $\left(\tilde{t}^2-\tilde{t}\right) \leq 1$ and elliptic regularity, we get
        \begin{align*}
            I &\leq \left( \norm{\tilde{\rho}_a^{n+1} \left( \nabla \tilde{c}_a^n - \nabla \tilde{c}_a^{n+1} \right)}_{L^2} + \norm{\tilde{\rho}_a^{n} \left( \nabla \tilde{c}_a^{n-1} - \nabla \tilde{c}_a^{n}\right)}_{L^2} \right) \norm{\nabla \phi}_{L^2} \\
            & \quad \quad \quad \quad \quad \quad \quad  \quad \quad  \quad \quad \quad   + \norm{\left(\tilde{\rho}_a^{n+1}-\tilde{\rho}_a^{n}\right)\left(\nabla \tilde{c}_a^{n+1}-\nabla \tilde{c}_a^{n}\right)\left(\tilde{t}^2-\tilde{t}\right)}_{L^2} \norm{\nabla \phi}_{L^2} \\
            &\leq C_{\text{ell}}\left(\norm{\tilde{\rho}_a^{n+1}}_{L^\infty} + \norm{\tilde{\rho}_a^{n} - \tilde{\rho}_a^{n+1}}_{L^\infty} \right) \norm{\tilde{\rho}_a^{n} - \tilde{\rho}_a^{n+1}}_{L^2} + C_{\text{ell}}\norm{\tilde{\rho}_a^{n}}_{L^\infty} \norm{\tilde{\rho}_a^{n-1} - \tilde{\rho}_a^{n}}_{L^2} \\
            &\leq \sum_{K \in \mathcal{T}_h} \theta_{K}^n.
        \end{align*}
        \noindent In order to bound $II$, we again consider the terms with the factors $\ell_0^n(t)$ and $\ell_1^n(t)$ individually.
        We observe
        \begin{align*}
            &- \int_{\mathbb{T}^d} \ell_0^n(t) \tilde{\rho}_a^{n+1} \nabla \tilde{c}_a^n \cdot \nabla \phi  + \ell_0^n(t)\mathcal{C}_h(\rho^{n}_{h,a},c_{h,a}^n)  \phi \ dx \\
            = &\  \underbrace{- \ell_0^n(t)\sum_{K \in \mathcal{T}_h}  \int_K \tilde{\rho}_a^{n+1} \nabla \tilde{c}_a^n \cdot \nabla \phi \ dx}_{=: \ III} \ \underbrace{- \ell_0^n(t)\sum_{K \in \mathcal{T}_h} \sum_{F \in F_K} \mathcal{C}_F(\rho_{h,a}^{n}) \mathcal{M}_F(\nabla c^n_{h,a} \cdot n_{K,F}) \mathcal{M}_K(\phi)}_{=:\  IV},
        \end{align*}

        %We first bound the terms in $III$ and treat $IV$ thereafter.
        \noindent  We start by estimating $III$ and, in the course of this, see how $IV$ can be integrated into these estimates.
        As the integrand in $III$ is globally well-defined, we can swap the mesh on which we work and rewrite $III$ as
        \begin{align*}
            III %&= - \ell_0^n(t)\sum_{T \in \mathcal{C}_{\mathcal{S}_{\mathcal{T}_h}}^{\mathcal{T}_h}}  \int_T \tilde{\rho}^{n+1} \nabla \tilde{c}^n \cdot \nabla \phi \ dx \\ 
            & = \underbrace{- \ell_0^n(t)\sum_{T \in \mathcal{C}_{\mathcal{S}_{\mathcal{T}_h}}^{\mathcal{T}_h}}  \int_T \tilde{\rho}_a^{n+1} \left( \nabla \tilde{c}_a^n - \nabla c_{h,a}^n \right) \cdot \nabla \phi \ dx}_{=: \ III_1}\  \underbrace{- \ell_0^n(t)\sum_{T \in \mathcal{C}_{\mathcal{S}_{\mathcal{T}_h}}^{\mathcal{T}_h}} \int_T \tilde{\rho}_a^{n+1} \nabla c_{h,a}^n \cdot \nabla \phi \ dx}_{ =: \ III_2}.
        \end{align*}
        We estimate $III_1$ using Hölder's inequality and the Cauchy-Schwarz inequality for sums, which gives
        \begin{align*}
            III_1 %\leq \sum_{T \in \mathcal{C}_{\mathcal{S}_{\mathcal{T}_h}}^{\mathcal{T}_h}}  \norm{\tilde{\rho}^{n+1}}_{L^\infty(T)} \norm{\nabla \tilde{c}^n - \nabla c_h^n}_{L^2(T)} \norm{\nabla \phi}_{L^2(T)} \\
            &\leq \ell_0^n(t)\norm{\tilde{\rho}_a^{n+1}}_{L^\infty} \bigg( \sum_{T \in \mathcal{C}_{\mathcal{S}_{\mathcal{T}_h}}^{\mathcal{T}_h}}  \norm{\tilde{c}_a^n - c_{h,a}^n}_{H^1(T)}^2 \bigg)^{1/2} \bigg( \sum_{T \in \mathcal{C}_{\mathcal{S}_{\mathcal{T}_h}}^{\mathcal{T}_h}} \norm{\nabla \phi}_{L^2(T)}^2 \bigg)^{1/2}\\
            &\leq \ell_0^n(t)\norm{\tilde{\rho}_a^{n+1}}_{L^\infty} \norm{\tilde{c}_a^n - c_{h,a}^n}_{H^1} 
        \end{align*}

        \noindent The term $\norm{\tilde{c}^{n} - c_{h,a}^{n}}_{H^1}$ can be estimated using classical \emph{a posteriori} analysis for FE schemes for the Poisson problem \cite[Chapter 4.3]{Verfurth2013}, but taking the perturbed Galerkin orthogonality (\ref{pert_orth}) into account. This  yields
        \begin{align*}
            \norm{\tilde{c}_a^n - c_{h,a}^{n}}_{H^1} \leq \eta_{\Omega}^n + \tilde{C}^{\text{sz}} \| r_\text{FE}^n \|_{L^2} \leq \eta_{\Omega}^n + \tilde{C}^{\text{sz}} \theta_\text{FE}^n,
        \end{align*}
        using (\ref{FE_theta}).
        Next, we use integration by parts on each element to rewrite
        \begin{align*}
            III_2 &= - \ell_0^n(t)\sum_{T \in \mathcal{C}_{\mathcal{S}_{\mathcal{T}_h}}^{\mathcal{T}_h}} \int_T \tilde{\rho}_a^{n+1} \nabla c_{h,a}^n \cdot \nabla \left(\phi  - \mathcal{M}_{K_T}(\phi)\right) \ dx \\
            & =\ell_0^n(t)\sum_{T \in \mathcal{C}_{\mathcal{S}_{\mathcal{T}_h}}^{\mathcal{T}_h}}  \int_T \divergence \left(\tilde{\rho}_a^{n+1} \nabla c_{h,a}^n \right) \left(\phi  - \mathcal{M}_{K_T}(\phi)\right) \ dx \\
            & \quad\quad\quad\quad\quad\quad\quad\quad\underbrace{- \ell_0^n(t)\sum_{T \in \mathcal{C}_{\mathcal{S}_{\mathcal{T}_h}}^{\mathcal{T}_h}} \int_{\partial T} \tilde{\rho}_a^{n+1} (\nabla c_{h,a}^n \cdot n_{T,S}) \left(\phi  - \mathcal{M}_{K_T}(\phi)\right) \ dS(x)}_{ =: \ III_3}.
        \end{align*}

        \noindent The first term will be the convective contribution to the element-wise residual and explicitly enters the desired estimate. 
        We write $III_3$ in terms of a sum over the 'intersected' faces and combine it with the leftover term $IV$ to get
        % \begin{align*}
        %     III_5 &= \ell_0^n(t)\sum_{T \in \mathcal{C}_{\mathcal{S}_{\mathcal{T}_h}}^{\mathcal{T}_h}}  \int_{\partial T} (\nabla c_h^n \cdot n_{T}) \tilde{\rho}^{n+1} \ dS(x) \mathcal{M}_{K_T}(\phi)  \\
        %     &= \ell_0^n(t)\sum_{T \in \mathcal{C}_{\mathcal{S}_{\mathcal{T}_h}}^{\mathcal{T}_h}}  \sum_{C \in C_T} \int_{C} (\nabla c_h^n \cdot n_{T,C}) \tilde{\rho}^{n+1}  \ dS(x) \mathcal{M}_{K_T}(\phi).
        % \end{align*}
        % Then, using first (\ref{obs}) and then (\ref{obs1}), we can observe that
        \begin{align*}
            III_3  + IV 
            %&  = \ell_0^n(t)\sum_{T \in \mathcal{C}_{\mathcal{S}_{\mathcal{T}_h}}^{\mathcal{T}_h}}  \sum_{C \in C_T} \int_C (\nabla c_h^n \cdot n_{T,C}) \bigg(\tilde{\rho}^{n+1} \left(\mathcal{M}_{K_T}(\phi)-\phi\right) - \frac{1}{|C|} \mathcal{F}_C(\rho_h^{n+1})\mathcal{M}_{K_T}(\phi) \bigg) \ dS(x)  \\
            %& = \ell_0^n(t)\sum_{T \in\mathcal{C}_{\mathcal{S}_{\mathcal{T}_h}}^{\mathcal{T}_h}}  \sum_{S \in C_T} \int_S (\nabla c_h^n \cdot n_{T,S}) \big(\tilde{\rho}^{n+1} \left(\mathcal{M}_{K_T}(\phi)-\phi\right) - \frac{1}{|S|} \mathcal{F}_S(\rho_h^{n+1})\left(\mathcal{M}_{K_T}(\phi)-\phi\right) \big) \\
            = & \  \ell_0^n(t)\sum_{T \in\mathcal{C}_{\mathcal{S}_{\mathcal{T}_h}}^{\mathcal{T}_h}}  \sum_{C \in C_T} \int_C (\nabla c_{h,a}^n \cdot n_{T,C}) \tilde{\rho}_a^{n+1} \left(\mathcal{M}_{K_T}(\phi)-\phi\right) dS(x) \\
             &\  - \ell_0^n(t)\sum_{K \in \mathcal{T}_h} \sum_{F \in F_K} \int_F (\nabla c^n_{h,a} \cdot n_{K,F}) \frac{1}{|F|}\mathcal{C}_F(\rho_{h,a}^{n}) \mathcal{M}_K(\phi).
        \end{align*}
        We want to rewrite these double sums in terms of a sum over all 'intersected' faces.
        To this end, we consider two cases: First, we observe that on an interior dual face, i.e. $C \subset S \in S_h$, we see that $K_T = K$ for the unique $K \in \mathcal{T}_h$ with $C \subset K$. Additionally, the gradient of the FE solution $\nabla c_h^n$ jumps across $C$.
        Secondly, on an interior primal face, i.e. $C \subset F \in F_h$, we see that the gradient of the FE solution  $\nabla c_h^n$ does not jump and is well-defined on $C$.
        Additionally, the two sides of $C$ are multiplied by factors $(\mathcal{M}_{K}(\phi) - \phi)$ and $(\mathcal{M}_{L}(\phi) - \phi)$, for the two different primal elements $K,L \in \mathcal{T}_h$ with $F = K \cap L$.
        %Thirdly, if $C$ is part of the boundary, we have $\mathcal{C}_C(\rho_h^{n}) = 0$, no jump of $\nabla c_h^n$ and there is only one factor $(\mathcal{M}_{K}(\phi) - \phi)$, for the $K \in \mathcal{T}_h$ with $C \subset \partial K$.
        % In this spirit, we define
        % \begin{align*}
        %     C_h^{E_h^{int}} &:= \left\{C \in C_h : \exists E \in E_h^{int}, \text{ such that } C \subset E\right\} \quad \text{and} \quad \\ C_h^{S_h} &:= \left\{C \in C_h : \exists S \in S_h, \text{ such that } C \subset S\right\}.
        % \end{align*}
        % Analogously, for each $T \in \mathcal{C}_{\mathcal{S}_{\mathcal{T}_h}}^{\mathcal{T}_h}$, we define $C_T^{E_h^{int}}$ and $C_T^{S_h}$ by replacing $C_h$ by $C_T$ in the definitions above.
        % Note that we have $C_h =  C_h^{E_h^{int}} \stackrel{.}{\cup} C_h^{S_h}$, as the boundary edges are only taken into account once, in $C_h^{S_h}$.
        We thus get
        \begin{align*}
            &III_3 + IV = \  \ell_0^n(t)\underbrace{\sum_{S \in C_h^{S_h}} \int_S \llbracket \nabla c_{h,a}^n \cdot n_{S} \rrbracket \tilde{\rho}_a^{n+1} (\mathcal{M}_T(\phi)-\phi) \ dS(x)}_{ =: \ V} \\ %\label{dual_edges}
            & \ + \ell_0^n(t)\underbrace{\sum_{F \in C_h^{F_h}} \int_F (\nabla c_{h,a}^n \cdot n_{F}) \big(\tilde{\rho}_a^{n+1} - \frac{1}{|F|} \mathcal{C}_F(\rho_{h,a}^{n}) \big)(\mathcal{M}_K(\phi)-\mathcal{M}_L(\phi)) \ dS(x)}_{ =: \ VI}, %\label{primal_edges}
        \end{align*}
        where $T \in \mathcal{T}_h$ is such that $S \subset T$ and $K,L \in \mathcal{T}_h$ are such that $F \subset K \cap L$.
        % Note that the case $C \subset \partial \Omega$ is covered in term $V$, due to the definition of jumps on the boundary.
        Applying the same arguments as for the jump terms in the proof of Lemma~\ref{diff-res}, we find
        \begin{align*}
            V &\leq  \tilde{c}_R \bigg(\sum_{S \in C_h^{S_h}} h_S \norm{\llbracket \nabla c_{h,a}^n \cdot n_{S} \rrbracket \tilde{\rho}_a^{n+1}}_{L^2(S)}^2 \bigg)^{1/2}.
            %&\leq \tilde{c}_R \bigg( \sum_{T \in \mathcal{C}_{\mathcal{S}_{\mathcal{T}_h}}^{\mathcal{T}_h} }\sum_{S \in C_T^{S_h}} \norm{\llbracket \nabla c_h^n \cdot n_{T,S} \rrbracket \tilde{\rho}^{n+1}}_{L^2(S)}^2 |S| \bigg)^{1/2}.
        \end{align*}

        \noindent Further, using (\ref{BrH}) and (\ref{Tr}) we observe
        \begin{align*}
            &h_F^{-1} \|\mathcal{M}_K(\phi)-\mathcal{M}_L(\phi)\|_{L^2(F)}^2 \leq  h_F^{-1}(\|\mathcal{M}_K(\phi) - \phi\|_{L^2(F)}^2 + \| \phi - \mathcal{M}_L(\phi)\|_{L^2(F)}^2) \\
            &\quad\quad\quad \leq c_{\text{Tr}}^2 \big(h_F^{-2} (\|\mathcal{M}_K(\phi)- \phi \|_{L^2(K)}^2 + \|\phi - \mathcal{M}_L(\phi)\|_{L^2(L)}^2) + \|\nabla \phi \|_{L^2(K)}^2 +  \|\nabla \phi \|_{L^2(L)}^2 \big)\\
            &\quad\quad\quad  \leq c_{\text{Tr}}^2 \big( c_{\text{P}}^2\frac{h_K^2}{h_F^2} \| \nabla \phi \|_{L^2(K)}^2 + c_{\text{P}}^2\frac{h_L^2}{h_F^2}\| \nabla \phi \|_{L^2(L)}^2 + \|\nabla \phi \|_{L^2(K)}^2 +  \|\nabla \phi \|_{L^2(L)}^2 \big).
        \end{align*}
        We thus find, again analogous to bounding the jump terms in the proof of Lemma~\ref{diff-res}, that
        \begin{align*}
            VI &\leq  \tilde{c}_R \bigg(\sum_{F \in C_h^{F_h}} h_F \norm{(\nabla c_{h,a}^n \cdot n_{F}) \big(\tilde{\rho}_a^{n+1} - \frac{1}{|F|} \mathcal{C}_F(\rho_{h,a}^{n})\big)}_{L^2(F)}^2  \bigg)^{1/2}.
            %&\leq \tilde{c}_R \bigg( \sum_{T \in \mathcal{C}_{\mathcal{S}_{\mathcal{T}_h}}^{\mathcal{T}_h} }\sum_{E \in C_T^{E_h^{int}}} \norm{(\nabla c_h^n \cdot n_{T,E}) \big(\tilde{\rho}^{n+1} - \frac{1}{|E|} \mathcal{C}_E(\rho_h^{n+1})\big)}_{L^2(E)}^2  |E| \bigg)^{1/2}.
        \end{align*}

        \noindent Using the same arguments to bound the terms with factor $\ell_1^n(t)$, we get the full estimate for $II$.

    \end{proof}

    \noindent We observe that the upper bounds in Lemmas~\ref{diff-res}, \ref{time-res} and \ref{conv-res} are not yet fully independent of $\phi$. 
    By summing, we can combine the terms that still depend on $\phi$ and thus obtain the element-wise residual (\ref{locres}), once with the factor $\ell_0^n(t)$ and once with $\ell_1^n(t)$.
    \noindent Using Hölder's inequality, (\ref{BrH}) and the Cauchy-Schwarz inequality for sums, we observe
    \begin{align}
        \label{locres}
        \begin{split}
        &\ell_0^n(t) \sum_{T \in \mathcal{C}_{\mathcal{S}_{\mathcal{T}_h}}^{\mathcal{T}_h}} \int_{T} \left(\frac{\tilde{\rho}_a^{n+1}-\tilde{\rho}_a^{n}}{\Delta t^n} + \divergence\left( \tilde{\rho}_a^{n+1} \nabla c_{h,a}^n \right) - \Delta \tilde{\rho}_a^{n+1} \right) \left( \phi - \mathcal{M}_{K_T}(\phi) \right) \ dx   \\
        & \quad\quad \leq \ell_0^n(t)c_{\text{P}}\bigg(\sum_{T \in \mathcal{C}_{\mathcal{S}_{\mathcal{T}_h}}^{\mathcal{T}_h}} \norm{\frac{\tilde{\rho}_a^{n+1}-\tilde{\rho}_a^{n}}{\Delta t^n} + \divergence\left( \tilde{\rho}_a^{n+1} \nabla c_{h,a}^n \right) - \Delta \tilde{\rho}_a^{n+1}}_{L^2(T)}^2  h_{K_T}^2 \bigg)^{1/2} \norm{\phi}_{H^1}.  
        \end{split}
    \end{align}
    This argument works analogously for the $\ell_1^n(t)$-terms.\\    
    
    \noindent Next, we investigate the residual in the first time step $t \in [t^0,t^1]$, i.e.
    \begin{align*}
        R_{\tilde{\rho}} &= \partial_t \tilde{\rho}_a + \divergence\left(\left(\ell_0^0(t)\tilde{\rho}_a^{1}(x)+\ell_1^0(t)\tilde{\rho}_a^{0}(x)\right)\left(\ell_0^0(t)\nabla {\tilde{c}_a}^{1}(x)+\ell_1^0(t)\nabla {\tilde{c}_a}^{0}(x)\right) \right) - \Delta \tilde{\rho}_a \\
        & \quad \quad \quad \quad \quad  \quad \quad \quad \quad \quad  \quad \quad \quad + (\ell_0^0(t)+\ell_1^0(t)) \left(  \mathcal{D}_h(\rho^{1}_{h,a}) - \mathcal{C}_h(\rho^{0}_{h,a},c_{h,a}^0) - \frac{\rho_{h,a}^{1}-\rho_{h,a}^{0}}{\Delta t^0} - r_{h}^1\right) \\
        & = \tilde{R}^D + \tilde{R}^T + \tilde{R}^C +\tilde{R}^a,
    \end{align*}
    where
    \begin{align*}
        \tilde{R}^D &:= \mathcal{D}_h(\rho^{1}_{h,a}) - \Delta \tilde{\rho}_a,\\
        \tilde{R}^T \  &:= \frac{\tilde{\rho}_a^{1}-\tilde{\rho}_a^{0}}{\Delta t^0} - \frac{\rho_{h,a}^{1}-\rho_{h,a}^{0}}{\Delta t^0}, \\
        \tilde{R}^C &:= \divergence\left(\left(\ell_0^0(t)\tilde{\rho}_a^{1}(x)+\ell_1^0(t)\tilde{\rho}_a^{0}(x)\right)\left(\ell_0^0(t)\nabla {\tilde{c}_a}^{1}(x)+\ell_1^0(t)\nabla {\tilde{c}_a}^{0}(x)\right) \right) - \mathcal{C}_h(\rho^{0}_{h,a},c_{h,a}^0), \\
        \tilde{R}^a &:=  r_h^{1}.
    \end{align*}
    \begin{Lemma}
        \label{diff-res0}
        Let $\phi \in H^1(\mathbb{T}^d)$ with $\norm{\phi}_{H^1} \leq 1$. For $t \in [t^0,t^1]$, there holds
        \begin{align*}
            \int_{\mathbb{T}^d} \tilde{R}^D \phi \ dx &\leq  \theta_D^{1} + \theta_{D,a} - \ell_1^0(t) \sum_{K \in \mathcal{T}_h} \int_K \nabla (\tilde{\rho}_a^{1} - \tilde{\rho}_a^{0}) \nabla \phi \ dx.
        \end{align*}
    \end{Lemma}
    \begin{proof}
        The $\ell_0^0(t)$-terms work as before, see Lemma~\ref{diff-res}. For the $\ell_1^0(t)$-terms, we first add and subtract the gradient of the Morley reconstruction $\nabla \tilde{\rho}_a^1$ and then proceed as before.
    \end{proof}
    \begin{Lemma}
        \label{time-res0}
        Let $\phi \in H^1(\mathbb{T}^d)$ with $\norm{\phi}_{H^1} \leq 1$. For $t \in [t^0,t^1]$, there holds
        \begin{align*}
            \int_{\mathbb{T}^d} \tilde{R}^T \phi \ dx \leq& \sum_{K \in \mathcal{T}_h}\int_K \frac{\tilde{\rho}_a^{1}-\tilde{\rho}_a^{0}}{\Delta t^0} \left(\phi-\mathcal{M}_K(\phi)\right) \ dx + \bigg(\sum_{K \in \mathcal{T}_h} \norm{\frac{\tilde{\rho}_a^{1}-\tilde{\rho}_a^{0}}{\Delta t^0} - \frac{\rho_{K,a}^{1}-\rho_{K,a}^{0}}{\Delta t^0}}_{L^2(K)}^2\bigg)^{1/2}.
        \end{align*}
    \end{Lemma}
    \begin{proof}
        The proof is analogous to before, see Lemma~\ref{time-res}.
    \end{proof}

    \begin{Lemma}
        \label{conv-res0}
        Let $\phi \in H^1(\mathbb{T}^d)$ with $\norm{\phi}_{H^1} \leq 1$. For $t \in [t^0,t^1]$, there holds
        \begin{align*}
            \int_{\mathbb{T}^d} \tilde{R}^C \phi \ dx \leq& \ \theta_C^{1} + \sum_{K \in \mathcal{T}_h} \theta_{K}^0 + \ell_1^0(t) \sum_{K \in \mathcal{T}_h} \int_{K} \left( \tilde{\rho}_a^1 - \tilde{\rho}_a^0 \right)\nabla c_{h,a}^0 \cdot \nabla  \phi \ dx +\ell_1^0(t) \norm{\tilde{\rho}_a^1 - \tilde{\rho}_a^0}_{L^\infty} \eta_{\Omega}^0,
        \end{align*}
        where
        \begin{align*}
            \theta_{K}^0 &:= C_{\text{ell}}\left(\norm{\tilde{\rho}_a^{1}}_{L^\infty} + \norm{\tilde{\rho}_a^{0} - \tilde{\rho}_a^{1}}_{L^\infty} \right)\norm{\tilde{\rho}_a^{0} - \tilde{\rho}_a^{1}}_{L^2(K)}.
        \end{align*}
    \end{Lemma}
    \begin{proof}
        The $\ell_0^0(t)$-terms are bounded as above, see Lemma~\ref{conv-res}. For the $\ell_1^0(t)$-terms, we first add and subtract the Morley reconstruction $\tilde{\rho}_a^1$ and then proceed as before.
    \end{proof}

    \noindent We again collect all the terms from the upper bounds from Lemmas~\ref{diff-res0}, \ref{time-res0} and \ref{conv-res0}, that still depend on $\phi$. The element-wise residual can be estimated as in (\ref{locres}). 
    For the extra terms at early times, using Hölder's inequality and the Cauchy-Schwarz inequality for sums, we find
    \begin{align}
        &\ell_1^0(t) \sum_{K \in \mathcal{T}_h} \int_{K} \big( ( \tilde{\rho}_a^1 - \tilde{\rho}_a^0 )\nabla c_{h,a}^0 - \nabla (\tilde{\rho}_a^{1} - \tilde{\rho}_a^{0}) \big) \cdot \nabla  \phi \ dx \notag \\
        & \quad\quad\quad\quad \leq \ell_1^0(t)\bigg(\sum_{K \in \mathcal{T}_h} \norm{ (\tilde{\rho}_a^1 - \tilde{\rho}_a^0 )\nabla c_{h,a}^0 - \nabla (\tilde{\rho}_a^{1} - \tilde{\rho}_a^{0})}_{L^2(K)}^2 \bigg)^{1/2} \norm{\phi}_{H^1}. \label{earlyextra}
    \end{align}

    \noindent Lastly, we observe
    \begin{Lemma}
        \label{alg-res}
        Let $\phi \in H^1(\mathbb{T}^d)$ with $\norm{\phi}_{H^1} \leq 1$. For $t \in [t^n,t^{n+1}]$, there holds
        \begin{align*}
            \int_{\mathbb{T}^d} \tilde{R}^a \phi \ dx \leq \begin{cases}
            \theta_\text{FV}^1, \quad &\text{for } n = 0, \\
            \ell_0^n(t) \theta_\text{FV}^{n+1} + \ell_1^n(t) \theta_\text{FV}^n, \quad &\text{for }  n = 1,\ldots,N_t-1,
            \end{cases}
        \end{align*}
        where $\theta_\text{FV}^m$ is as in (\ref{FV-alg}) for $m=1, \ldots,N_t$.
    \end{Lemma}
    \begin{proof}
        Hölder's inequality and (\ref{FV-alg}) lead to the desired upper bound.
        % \[
        % \norm{\widetilde R^a(t)}_{(H^1(\mathbb T^d))'}
        % \le
        % \begin{cases}
        %     \theta_{\mathrm{FV}}^1,
        %     & t\in[t^0,t^1],\\[2mm]
        %     \ell_0^n(t)\theta_{\mathrm{FV}}^{n+1}
        %     +\ell_1^n(t)\theta_{\mathrm{FV}}^n,
        %     & t\in[t^n,t^{n+1}],\quad n\ge1.
        % \end{cases}
        % \]
    \end{proof}

    \subsection{A posteriori error estimators}
    \label{aposterirorierroresti}

    \noindent Finally, we obtain the following, central result
    \begin{Theorem}[A posteriori error estimators]
        \label{main_thm}
        Let $d \in \left\{2,3\right\}$ and $T>0$. 
        For some $0 \leq \rho_0 \in L^2$, let $(\rho,c)$ be a weak solution of the Keller-Segel system (\ref{KS1}), (\ref{KS2}), (\ref{KScond}) on $[0,T_\text{max}]$ and let $(\tilde{\rho}_a,\tilde{c}_a)$ be the reconstruction, defined in Section~\ref{Morley}, of our numerical approximation $(\rho_{h,a},c_{h,a})$, defined in Subsection~\ref{alg-error}.
        Then
        \begin{align}
            \label{estimator}
            \int_0^T \norm{R_{\tilde{\rho}_a}}^2_{(H^1(\mathbb{T}^d))^\prime} \ dt \leq \sum_{n=0}^{N_t-1}  \frac{\Delta t^n}{2} \big(\Theta_n(\tau_-^n)^2 + \Theta_n(\tau_+^n)^2\big)\bigg(1+\frac{C_7(h)\varepsilon}{1-C_7(h)\varepsilon}\bigg),
        \end{align}
        where $\tau_{\pm}^n := \frac{t^n + t^{n+1}}{2} \pm \frac{\Delta t^n}{2 \sqrt{3}}$ and $C_7(h)$ is the number of elementary operations needed to compute $\Theta_n(\tau_-^n)^2 + \Theta_n(\tau_+^n)^2$.
        For $m = 1, \ldots, N_t -1$, $\Theta_{m}(t)$ is defined by summing the bounds of Lemmas~\ref{diff-res}, \ref{time-res}, \ref{conv-res} and \ref{alg-res} using (\ref{locres}). 
        For $m=0$, $\Theta_{0}(t)$ is defined by summing the bounds of Lemmas~\ref{diff-res0}, \ref{time-res0}, \ref{conv-res0} and \ref{alg-res} using (\ref{locres}) as well as (\ref{earlyextra}).
        % \begin{align}
        %     \label{estimator}
        %     \int_0^T \norm{R_{\tilde{\rho}_a}}_{(H^1(\mathbb{T}^d))^\prime}^2 \ dt \leq  \sum_{n=0}^{N_t-1}  \Delta t^n \bigg( \Theta_{a,n} + \sum_{K \in \mathcal{T}_h} \Theta_{K,n}^2\bigg)\bigg(1+\frac{C_7(h)\varepsilon}{1-C_7(h)\varepsilon}\bigg) ,
        % \end{align}
        % where $C_7(h)$ is the number of elementary operations needed to compute $ \Theta_{a,n} + \sum_{K \in \mathcal{T}_h} \Theta_{K,n}^2$.
        % For \textcolor{purple}{$m=1, \ldots, N_t-1$, $\Theta_{a,m} + \sum_{K \in \mathcal{T}_h} \Theta_{K,m}^2$} is defined by summing the bounds of Lemmas~\ref{diff-res}, \ref{time-res}, \ref{conv-res} and \ref{alg-res} using (\ref{locres})
        % and for \textcolor{purple}{$m=0$}, $\Theta_{a,0} + \sum_{K \in \mathcal{T}_h} \Theta_{K,0}^2$ is defined by summing the bounds of Lemmas~\ref{diff-res0}, \ref{time-res0}, \ref{conv-res0} and \ref{alg-res} using (\ref{locres}) as well as (\ref{earlyextra}). %, once for the elementwise residual and once the terms at early times.
        If additionally,
        \begin{enumerate}
            \item[i)] for some $\delta > 1$ the condition $ 2 C_{\text{S}}^3 C_\text{ell}^2 \delta A E + 4C_{\text{S}}^6 C_{\text{ell}}^4 \delta^2 A^2 E^2 < \frac{\delta -1}{\delta T E}$ is satisfied, where 
            \begin{align*}
                A &:= \norm{\rho(0,\cdot) - \tilde{\rho}_a(0,\cdot)}_{L^2}^2 + 6\sum_{n=0}^{N_t-1}  \Delta t^n \big(\Theta_n(\tau_-^n)^2 + \Theta_n(\tau_+^n)^2\big)\bigg(1+\frac{C_7(h)\varepsilon}{1-C_7(h)\varepsilon}\bigg)  \\
                a(t) &:= 4 C_{\text{S}}^2 C_{\text{ell}}^2 \norm{\tilde{\rho}_a(t,\cdot)}_{L^3}^2 + 4 \big(\norm{q_{h,a}(t,\cdot)}_{L^\infty} + \eta_\infty(t)\big)^2+\frac{1}{8} \\
                E &:= \exp\left(\int_0^T a(t) \ dt \bigg( 1+ \frac{C_8(h,N_t) \varepsilon}{1 - C_8(h,N_t) \varepsilon} \bigg) \right),
            \end{align*}
            with $q_{h,a}(t,\cdot)$ and $\eta_\infty(t)$ as in Remark~\ref{grad_c_up} with the right-hand side $\tilde{\rho}_a(t,\cdot)$, for every $t \in [0,T]$, and with $C_8(h,N_t) > 0$ being the number of elementary operations needed to compute $\int_0^T a(t) \ dt$, we obtain the a posteriori error estimator
            \begin{align*}
                 &\sup_{t \in [0,\min\{T,T_\text{max}\}]}\norm{\rho(t,\cdot) - \tilde{\rho}_a(t,\cdot)}_{L^2}^2  \leq \delta A E.
            \end{align*}
            In particular, $(\rho, c)$ is a weak solution up to time $T < T_\text{max}$.
            \item[ii)] for each $n = 0, 1, \ldots, N_t-1$, there is a $\delta_n > 1$ such that the condition \\$\Delta t^n \left( 2 C_{\text{S}}^3 C_\text{ell}^2 \delta_n A_n E_n + 4C_{\text{S}}^6 C_{\text{ell}}^4 \delta_n^2 A_n^2 E_n^2\right) - \log(\delta_n) < 0$ is satisfied, where $A_n$ and $E_n$ are iteratively defined as 
            % $B_1 := \frac45 C_{\text{S}}^3 C_\text{ell}^2$ and $B_2 := \frac{864}{125}C_{\text{S}}^6 C_{\text{ell}}^4$
            \begin{align*}
                A_0 &:= \norm{\rho(0,\cdot) - \tilde{\rho}_a(0,\cdot)}_{L^2}^2 + 6 \Delta t^0 \big(\Theta_0(\tau_-^0)^2 + \Theta_0(\tau_+^0)^2\big)\bigg(1+\frac{C_7(h)\varepsilon}{1-C_7(h)\varepsilon}\bigg) , \\
                A_{n+1} &:= \Psi_n + 6 \Delta t^{n+1} \big(\Theta_{n+1}(\tau_-^{n+1})^2 + \Theta_{n+1}(\tau_+^{n+1})^2\big)\bigg(1+\frac{C_7(h)\varepsilon}{1-C_7(h)\varepsilon}\bigg) , \quad \text{for } n=0, \ldots, N_t-2, \\
                E_n &:= \exp\bigg(\int_{t^n}^{t^{n+1}} a(t) ds \bigg( 1+ \frac{C_9(h) \varepsilon}{1 - C_9(h) \varepsilon} \bigg)\bigg), \\[10pt]
                \Psi_n &:= \delta_n A_n E_n,
            \end{align*}
            with $a(t)$ as in $i)$,
            with $q_{h,a}(t,\cdot)$ and $\eta_\infty(t)$ are as in Remark~\ref{grad_c_up} with the right-hand side $\tilde{\rho}_a(t,\cdot)$, for every $t \in [0,T]$, and with $C_9(h) > 0$ being the number of elementary operations needed to compute the integral in $E_n$, we obtain the a posteriori error estimator
            \begin{align*}
                \sup_{t \in [0,\min\{T,T_\text{max}\}]}\norm{\rho(t,\cdot) - \tilde{\rho}_a(t,\cdot)}_{L^2}^2  \leq \Psi_{N_t-1}.
            \end{align*}
            In particular, $(\rho, c)$ is a weak solution up to time $T < T_\text{max}$.
        \end{enumerate}
    \end{Theorem}
    \begin{proof}
        We observe
        \begin{align*}
            \int_0^T \norm{R_{\tilde{\rho}_a}}^2_{(H^1(\mathbb{T}^d))^\prime} \ dt \leq \sum_{n=0}^{N_t-1}  \int_{t^n}^{t^{n+1}} \Theta_n(t)^2 dt \leq \sum_{n=0}^{N_t-1}  \frac{\Delta t^n}{2} \big(\Theta_n(\tau_-^n)^2 + \Theta_n(\tau_+^n)^2\big)\bigg(1+\frac{C_7(h)\varepsilon}{1-C_7(h)\varepsilon}\bigg),
        \end{align*}
        with quadrature points $\tau_{\pm}^m := \frac{t^m + t^{m+1}}{2} \pm \frac{\Delta t^m}{2 \sqrt{3}}$ that stem from the two-point Gaussian quadrature, which exactly integrates $\Theta_m(t)$ on $[t^m,t^{m+1}]$, as it is a polynomial of degree 2.
        We combine this with the stability frameworks in Theorem~\ref{stab_thm} and  Theorem~\ref{stab_cont} to obtain $i)$ and $ii)$, respectively. 
        Additionally, we employ the computable upper bound of $\norm{\nabla \tilde{c}_a(t,\cdot)}_{L^\infty}$ from Remark~\ref{grad_c_up}.
        Lastly, we invoke Proposition \ref{apostverif}.
    \end{proof}

    \begin{Remark}
        A closed form representation of the estimator $\sum_{n=0}^{N_t-1}  \frac{\Delta t^n}{2} (\Theta_n(\tau_-^n)^2 + \Theta_n(\tau_+^n)^2)$, can be found in Appendix~\ref{closedform}.
        Note that this representation does not decompose into a sum of contributions of individual cells as required for e.g. adaptive mesh refinements.
        However, for our purposes, i.e. using this estimator to infer existence of a weak solution, see Proposition~\ref{apostverif}, a global estimator is sufficient.
        If one is interested in a decomposition into local contributions, this can be achieved by associating intersected and dual elements to suitable corresponding primal elements and rewriting the sums accordingly. 
        This may require some additional upper bounds, and is not further pursued here.
    \end{Remark}
    
    \noindent Theorem~\ref{main_thm} states two fully computable conditional \emph{a posteriori} error estimators, which can be used in the context of \emph{a posteriori} verifiable existence.
    They play the role of $\theta$ in Proposition~\ref{apostverif}.

    \noindent The values of $C_7(h), C_8(h,N_t)$ and $C_9(h)$ can be obtained from the implementation and are stated in Section~\ref{num_sim}.

%% file: numericalsimulations.tex
\section{Numerical simulations}
\label{num_sim}

In our numerical simulations we compare the availability of the two different \emph{a posteriori} error estimators presented in Theorem~\ref{main_thm} and investigate the scaling behavior of the \emph{a posteriori} residual estimates from Subsection~\ref{residual_esti} using manufactured solutions.
Algorithm~\ref{FV-FE-algo} was implemented using Python 3.8.10.
The code which was used to compute the results below is available on git, see \cite{code_zenodo}. 

\noindent Our implementation leads to the following constants needed to evaluate the parts of the error estimator due to the algebraic error, which read as
\begin{align*}
    C_1(h) &= 1 + (d+4) N_{\mathcal{T}_h} + (5 + 2 N_{\mathcal{S}_{\mathcal{T}_h}})(d+1) \\
    C_2(h) &= 9 (d+1) N_{\mathcal{T}_h}, \\
    C_3^a(h) &= 3 (d+1) N_{\mathcal{T}_h} \max_{F \in F_h}(\mathcal{D}_{K,F}(\rho^n_{h,a})), \\
    C_3^b(h) &= 2 (d+1) N_{\mathcal{T}_h} \max_{F \in F_h}(|F| (\nabla q_0 \cdot n_{K,F})), \\
    C_4 &= 2(d+1)^2 + d + 2, \\
    C_5(h) &= (6d+7) N_{\mathcal{S}_{\mathcal{T}_h}},  \\
    C_6(h) &= (d+1)(d+2) + N_{\mathcal{S}_{\mathcal{T}_h}}\Big(\frac{7}{3} (d+1)^3 + 4(d+1)^2\Big), \\
    C_7(h) &= N_{\mathcal{T}_h}((d+1)d (q_d(9+d) + 5) + 10 q_d + d + 7) \\ 
    & \qquad\qquad\qquad\qquad\qquad + N_{F_h}(2 f_d +3) +   N_{C_h}(f_d(d+3)+2) + 2N_{S_h} + 20, \\
    C_8(h,N_t) &= N_t (B_h + 1) + 4, \\
    C_9(h) &= B_h + 5, \\
    C_{10} &= 2(d+1)^2 + d + 2, \\
    C_{11} &= (6d+7) d N_{\mathcal{T}_h}, \\
    C_{12} &= (d+1)(d+2) + N_{\mathcal{T}_h}\Big(\frac{7}{3} (d+1)^3 + 4(d+1)^2\Big),
\end{align*}
where $q_2 = 28$, $q_3 = 214$, $f_2 = 5$ and $f_3 = 37$ are the number of quadrature points needed for integration over elements and faces for $d \in \{2,3\}$, respectively, and where $N_{F_h}$, $N_{S_h}$ and $N_{C_h}$ denote the numbers of primal, dual and intersected faces, respectively.
Additionally, $N_p$ denotes the number of nodes in $\mathcal{T}_h$ and we set $B_h := N_{\mathcal{T}_h}(3q_d + d( 6d + 26 + 3(d+1)d + 2 N_p + 3 ) + \frac{7}{3} (d+1)^3 + 4(d+1)^2) + (d+1)(d+2)$, for brevity.

\subsection{Availability of the error estimates}
\label{subsec:availability}

In this numerical test, we examine the time horizons in which the respective error estimates from Theorem~\ref{main_thm} are available, i.e. in which the conditions in $i)$ and $ii)$ of Theorem~\ref{main_thm} hold true, for $d \in\{2,3\}$.

% \noindent As already mentioned, for $d=3$, we cannot investigate the availability of the error estimator in a meaningful manner, as there are only very rough upper bounds on especially \textcolor{red}{the constant $C_{\text{S}} \approx 20.6585$ in the literature, see \cite[Table 6]{Mizuguchi2017}}, which in combination with taking it to the sixth power and the exponential function yield that the conditions of the stability frameworks (\ref{stab_cond}) and (\ref{root}) do only hold for extremely short times and extremely fine meshes that would have exceeded the computing times available to us.
% We hope that sharper estimates for these constant will become available in the future.

\noindent To this end, we derived feasible bounds for the Sobolev embedding constant $C_{\text{S}}$ of $H^1(\mathbb{T}^d) \hookrightarrow L^6(\mathbb{T}^d)$, which are, to our knowledge, tighter than any other estimates available in the literature.
We use 
\begin{align*}
  C_{\text{S}} = \begin{cases}
  1.1401792579142, \quad &\text{for } d=2,\\
  \sqrt{\frac{7}{3}}, \quad &\text{for } d=3,
  \end{cases} 
\end{align*}
see Appendix~\ref{subsec:sobolev}.
As the initial datum we choose
\begin{alignat}{3}
  \label{stab:rho02D}
  \rho_0(x,y) &=  \cos(2\pi x)\cos(2\pi y)+1, && \quad \text{for }  d=2, \\
  \label{stab:rho03D}
  \rho_0(x,y,z) &=  \cos(2\pi x)\cos(2\pi y)\cos(2\pi z)+1, && \quad \text{for } d=3.
\end{alignat}

\noindent For $d=2$, we choose a well-centered triangular mesh with mesh size $h = 2^{-(i+1)}$, for $i = 5,6,7$. 
A sufficient condition for well-centeredness for $d=2$ is that all angles of all elements are acute.
\noindent We again discretize in time using equidistant time steps $\Delta t$, specified in Table~\ref{stab_comp2D}.

\begin{figure}[h]
  \centering
  \begin{tabulary}{18cm}{ | C | C || C | C | C | C || C | C | C | C |} 
    \hline
\multicolumn{2}{|c||}{} &
\multicolumn{4}{c||}{Theorem~\ref{main_thm} $i)$} &
\multicolumn{4}{c|}{Theorem~\ref{main_thm} $ii)$} \\
\hline
      \hline     
      $h$       & $\Delta t$ & $T$                           & $A$  & $E$  & $\delta$ & $T$                          & $A_{N_t-1}$ & $E_{N_t-1}$ & $\delta_{N_t-1}$ \\ 
      \hline \hline
      $2^{-6}$  & 3.3e-04    & 2.33e-03  & 2.02 & 1.17 &  1.46       & 4.66e-03 & 6.13     & 1.03      &  1.19\\ 
      \hline
      $2^{-7}$  & 2.5e-04    & 6.50e-03  & 1.04 & 1.25 &  1.38      & 1.30e-02 & 8.54      & 1.02      &  1.40\\ 
      \hline
      $2^{-8}$  & 2.0e-04    & 1.98e-02     & 0.42    & 1.53    &  1.50   & 2.48e-02    & 9.74           & 1.02           &  1.45\\ 
      \hline
  \end{tabulary}
  \captionof{table}{Comparison of the maximal time of availability of the \emph{a posteriori} error estimators presented in Theorem \ref{main_thm}, in the case of the initial value (\ref{stab:rho02D}) for $d = 2$.}
  \label{stab_comp2D}
\end{figure}

\noindent For $d=3$, we use the well-centered tetrahedral mesh of the infinite slab, proposed in \cite{Hirani2008}, with parameters $a = b = \frac34$, which is readily available. 

\noindent As the faces of the infinite slab match up with each other, mesh refinements are possible by decomposing the unit cube into smaller copies of the unit cube and discretizing these, appropriately scaled. 
We use decompositions with $3^3$, $4^3$, $5^3$, $6^3$, $7^3$ and $8^3$ copies.  
Note that this refinement strategy does not halve the mesh size $h$ per refinement but reduces it by some constant in the interval $(0.5,1)$.
We discretize in time using equidistant time steps $\Delta t$, specified in Table~\ref{stab_comp3D}.

\begin{figure}[h]
  \centering
  \begin{tabulary}{18cm}{ | C | C || C | C | C | C || C | C | C | C |} 
    \hline
\multicolumn{2}{|c||}{} &
\multicolumn{4}{c||}{Theorem~\ref{main_thm} $i)$} &
\multicolumn{4}{c|}{Theorem~\ref{main_thm} $ii)$} \\
\hline
      \hline     
      $h$       & $\Delta t$ & $T$                     & $A$  & $E$  & $\delta$ & $T$                          & $A_{N_t-1}$ & $E_{N_t-1}$ & $\delta_{N_t-1}$ \\ 
      \hline \hline
      8.84e-02  & 1.0e-06    &3.00e-06   & 8.27 & 2.35 &  1.29       & 6.00e-06 & 43.35       & 1.33        &  1.38\\ 
      \hline
      7.07e-02  & 8.0e-07    & 4.80e-06  & 10.27 & 1.61 &  1.18       & 1.04e-05 & 48.78       & 1.08        &  1.17\\ 
      \hline
      5.89e-02  & 7.0e-07    & 7.00e-06  & 11.19 & 1.34 &  1.18       & 1.61e-05 & 56.94       & 1.03        &  1.19\\ 
      \hline
      5.05e-02  & 6.7e-07    & 1.00e-05  & 12.20    & 1.24    &  1.37 & 2.20e-05 & 63.69     & 1.02           & 1.25\\ 
      \hline
      4.42e-02  & 6.0e-07    &1.26e-05   & 12.05    & 1.16    &  1.40 & 2.82e-05 & 76.83      & 1.01        &  1.57\\ 
      \hline
  \end{tabulary}
  \captionof{table}{Comparison of the maximal time of availability of the \emph{a posteriori} error estimators presented in Theorem \ref{main_thm}, in the case of the initial value (\ref{stab:rho03D}) for $d = 3$.}
  \label{stab_comp3D}
\end{figure}

\noindent We observe in Table~\ref{stab_comp2D} and Table~\ref{stab_comp3D} that the \emph{a posteriori} error estimator based on the local-in-time continuation argument performs better than the one based on the (adjusted) Generalized Gronwall Lemma, see Proposition \ref{adj_gengronwall}, and its underlying global-in-time continuation.
\emph{A posteriori} error estimators based on the Generalized Gronwall Lemma \cite[Proposition 6.2]{Bartels2015nonlinear} perform even worse.
Additionally, refining in space and time leads to larger time horizons of availability, as expected in Remark~\ref{rem:gen} and Remark~\ref{rem:loc}. Note that $N_t$ implicitly depends on the size of $h$ and $\Delta t$.

\noindent Due to a potential accumulation of an exceeding number of $\delta_n>1$ multiplications in the local-in-time continuation approach, it is however not favorable to do too many time steps, as experimentally shown in Table~\ref{stab_comp2Dtn}. This does not apply to the Generalized Gronwall approach.

\begin{figure}[h]
  \centering
  \begin{tabulary}{18cm}{ | C | C || C | C | C | C || C | C | C | C |} 
    \hline
\multicolumn{2}{|c||}{} &
\multicolumn{4}{c||}{Theorem~\ref{main_thm} $i)$} &
\multicolumn{4}{c|}{Theorem~\ref{main_thm} $ii)$} \\
\hline
      \hline     
      $h$       & $\Delta t$ & $T$                           & $A$  & $E$  & $\delta$ & $T$           & $A_{N_t-1}$ & $E_{N_t-1}$ & $\delta_{N_t-1}$ \\ 
      \hline \hline
      $2^{-7}$  & 1.00e-03    & 4.00e-03  & 1.30 & 1.18 &  1.20       & 9.00e-03 & 4.06      & 1.06       &  1.40\\ 
      \hline
      $2^{-7}$  & 2.50e-04    & 6.50e-03  & 1.04 & 1.25 &  1.38       & 1.30e-02 & 8.54       & 1.02        &  1.40\\ 
      \hline
      $2^{-7}$  & 6.25e-05    & 7.69e-03    & 0.93    & 1.28    &  1.44 & 1.22e-02     & 16.47           & 1.01          &  1.29\\ 
      \hline
      $2^{-7}$  & 1.56e-05    & 8.08e-03     & 0.90    & 1.29    & 1.51  & 9.88e-03   & 31.92           & 1.00    &  1.25\\ 
      \hline
  \end{tabulary}
  \captionof{table}{Investigation of the maximal time of availability of the \emph{a posteriori} error estimators presented in Theorem \ref{main_thm}, in the case of the initial value (\ref{stab:rho02D}) for $d = 2$ for fixed mesh size $h = 2^{-7}$ and very small time steps.}
  \label{stab_comp2Dtn}
\end{figure}

%Numerically testing the CFL condition is not further pursued in this work. 

% One can observe, that its total mass surpasses a threshold which implies finite in time blow-up, as stated in \cite{Nagai1995}.
% The blow-up phenomenon can be observed in the numerical simulation illustrated in \autoref{blowup}.
% We observe that the blow-up is less rapid than in \cite{Kwon2023} which is due to the fact that our mesh is too coarse to resolve the blow-up phenomenon appropriately. 
% This coarse mesh was chosen because the computing time for finer meshes was too high.

%\textcolor{red}{estimation of elliptic regularity constan, see grisvard \cite{Grisvard2011} and bebendorf \cite{Bebendorf2003} and for Sobolev (and with that Gagliardo Nirenberg) constants see Mizuguchi et al \cite{Mizuguchi2017}}

\subsection{Scaling behavior of the residual estimate}
\label{scaling_res}

In this numerical test, we examine the order of convergence of the residual estimator (\ref{estimator}).
We work on $\mathbb{T}^d \cong  [0,1)^d$ for $d\in \{2,3\}$
and use the same space and time discretizations as in the previous Subsection.

% \label{scaling}
% We are interested in the scaling behavior, i.e. the order of convergence, of our a posteriori error estimator and wish to compare it to the scaling behavior of the actual error. 

% An important tool that allows us to approximate the order of convergence of computed quantities, in our case the different parts of the error estimator, is the estimated order of convergence (EOC).
% \begin{Definition}[EOC, see Definition 8.1 in \cite{Kwon2023}]
% Given a sequence $(a_i)_{i\in \N}$ and $(h_i)_{i\in \N} \searrow 0$ for $i \rightarrow \infty$, we define the EOC via the local slope of $a_i$ versus $h_i$, in logarithmic scales, i.e.
% \begin{align*}
%     EOC_i(a,h) := \frac{\log\left( a_{i+1}/a_i \right)}{\log\left( h_{i+1}/h_i \right)}.
% \end{align*}
% \end{Definition}
% We will study the EOC of our a posteriori error estimator from above, see \autoref{apostierroresti}.

\noindent For $d=2$, we consider the manufactured solution
\begin{align}
  \rho(t,x,y) &:= \frac{1}{1 + t} \cos(2\pi x) \cos(2\pi y) + 1, \label{manuf_rho2d}\\ 
  c(t,x,y) &:= \cos(2\pi x) \cos(2\pi y) + 1, \label{manuf_c2d}
\end{align}
for $d=3$,
\begin{align}
  \rho(t,x,y,z) &:= \frac{1}{1 + t} \cos(2\pi x) \cos(2\pi y) \cos(2\pi z) + 1, \label{manuf_rho3d}\\ 
  c(t,x,y,z) &:= \cos(2\pi x) \cos(2\pi y) \cos(2\pi z) + 1, \label{manuf_c3d}
\end{align}
that solve the system
\begin{align*}
  %\label{manuf_system}
  \begin{cases}
      \ \ \ \partial_t \rho + \divergence\left( \rho \nabla c \right) - \Delta \rho &= f \quad\quad\quad\quad \ \ \text{ in } (0,T) \times \mathbb{T}^d, \\
      \quad \quad \quad \quad \quad \quad \quad c - \Delta c &= \rho + g \quad\quad\quad \text{ in } (0,T) \times \mathbb{T}^d, \\
      \quad \quad \quad \quad \quad \quad \quad \ \rho(0,\cdot) &= \rho_0 \quad\quad\quad\quad \ \text{ in } \mathbb{T}^d,
  \end{cases}
\end{align*}
with appropriately chosen $f : (0,T) \times \mathbb{T}^d \rightarrow \mathbb{R}$ and $g : (0,T) \times \mathbb{T}^d \rightarrow \mathbb{R}$.

\noindent We compute the 'exact' error of the interpolation of Morley-type of the numerical approximations, computed using Algorithm~\ref{FV-FE-algo}, to the manufactured solution above measured in the $L^\infty(0,T;L^2)$-norm and the $L^2(0,T;H^1)$-norm for different refinements of our mesh, as explained above.
Additionally, we investigate the scaling behavior of the (global) residual estimate (\ref{estimator}), i.e.
\begin{align*}
  \Theta := \left(\sum_{n=0}^{N_t-1}  \frac{\Delta t^n}{2} (\Theta_n(\tau_-^n)^2 + \Theta_n(\tau_+^n)^2) \right)^{1/2},
\end{align*}
where $\tau_{\pm}^n := \frac{t^n + t^{n+1}}{2} \pm \frac{\Delta t^n}{2 \sqrt{3}}$,

on the same meshes.
We use the estimated order of convergence (EOC) to approximate the orders of convergence of our computed quantities.
% \begin{Definition}[EOC, see Definition 8.1 in \cite{Kwon2023}]
% Given a sequence $(a_i)_{i\in \N}$ and $(h_i)_{i\in \N} \searrow 0$ for $i \rightarrow \infty$, we define the EOC via the local slope of $a_i$ versus $h_i$, in logarithmic scales, i.e.
% \begin{align*}
%     EOC_i(a,h) := \frac{\log\left( a_{i+1}/a_i \right)}{\log\left( h_{i+1}/h_i \right)}.
% \end{align*}
% \end{Definition}
The resulting values, EOCs, mesh sizes and time step sizes are displayed in Table~\ref{scaling2D} and Table~\ref{scaling3D}.

\begin{figure}[h]
  \centering
  \begin{tabulary}{15.5cm}{ | C | C || C  C ||  C C || C C |} 
      \hline     
      $h$ & $\Delta t$ & $L^\infty(0,T;L^2)$ & EOC  & $L^2(0,T;H^1)$ & EOC & $\Theta$ & EOC \\ 
      \hline \hline
      % $ 0.35355 $ & $T/8$   & 2.40e-01 & 1.21 & 2.36e-00 & 0.89 & 8.08e+02 & 0.45 \\ 
      % \hline
      2.50e-01 &  $T/16$  & 1.93e-01 & 1.72 & 1.90e+00 & 1.01 & 1.18e+02 & 0.86\\ 
      \hline
      1.25e-01&  $T/32$ & 5.89e-02   & 1.91 & 9.44e-01 & 1.01  & 6.54e+01 & 0.92 \\ 
      \hline
      6.25e-02 & $T/64$  & 1.57e-02 & 1.94 & 4.68e-01 & 1.00 & 3.47e+01 & 0.99 \\ 
      \hline
      3.125e-02 & $T/128$  & 4.07e-03 & 1.98 & 2.34e-01 & 1.00 & 1.75e+01 & 1.00 \\ 
      \hline
      1.56e-02 & $T/512$  & 1.04e-03 & - & 1.17e-01 & - & 8.76e+00 & - \\ 
      \hline  
      % 7.81e-03 & $T/3072$  & 2.59e-04 & -    & 4.72e-02 & -  & xxx & - \\ 
      % \hline
      % 3.91e-03 & $T/4048$  & xxx & -     & xxx & -  & xxx & - \\ 
      % \hline
      % 1.95e-03 & $T/256$  & xxx & -     & xxx & -  & xxx & - \\ 
      % \hline
 
  \end{tabulary}
  \captionof{table}{Scaling behavior of the 'exact' error in the $L^\infty(0,T;L^2)$- and $L^2(0,T;H^1)$-norm in the case of the manufactured solution (\ref{manuf_rho2d}), (\ref{manuf_c2d}) as well as the corresponding \emph{a posteriori} residual estimator $\Theta$, for $d = 2$.}
  \label{scaling2D}
\end{figure}

\begin{figure}[h]
  \centering
  \begin{tabulary}{15.5cm}{ | C | C || C  C ||  C C || C C |} 
      \hline     
      $h$ & $\Delta t$ & $L^\infty(0,T;L^2)$ & EOC  & $L^2(0,T;H^1)$ & EOC & $\Theta $ & EOC \\ 
      \hline \hline
      % $ 0.35355 $ & $T/8$   & 2.40e-01 & 1.21 & 2.36e-00 & 0.89 & 7.91e+02 & 0.42 \\ 
      % \hline
      % 1.76e-01 &  $T/16$  & 1.10e-01 & 1.64 & 1.28e-00 & 1.14 & 5.91e+02 & 0.78\\ 
      % \hline
      1.18e-01&  $T/32$ & 5.68e-02   & 1.78 & 8.10e-01 & 1.12  & 4.64e+02 & 0.82 \\ 
      \hline
      8.84e-02 & $T/64$  & 3.40e-02 & 1.86 & 5.86e-01 & 1.09 & 3.66e+02 & 0.89 \\ 
      \hline
      7.07e-02 & $T/128$  & 2.25e-02 & 1.90 & 4.60e-01 & 1.07 & 3.002e+02 & 0.93 \\ 
      \hline
      5.89e-02 & $T/256$  & 1.59e-02 & 1.93 & 3.78e-01 & 1.05 & 2.54e+02 & 0.95 \\ 
      \hline  
      5.05e-02 & $T/512$  & 1.12e-02 & -     & 3.22e-01 & -  & 2.19e+02 & - \\ 
      \hline
 
  \end{tabulary}
  \captionof{table}{Scaling behavior of the 'exact' error in the $L^\infty(0,T;L^2)$- and $L^2(0,T;H^1)$-norm in the case of the manufactured solution (\ref{manuf_rho3d}), (\ref{manuf_c3d}) as well as the corresponding \emph{a posteriori} residual estimator $\Theta$, for $d = 3$.}
  \label{scaling3D}
\end{figure}

\noindent As one would expect, the $L^\infty(0,T;L^2)$-error scales with approximately order 2 and the $L^2(0,T;H^1)$-error with approximately order 1. 
The \emph{a posteriori} residual estimate $\Theta$ also scales with approximately order 1, which matches the order of the error, measured in the $L^\infty(0,T;L^2) \cap L^2(0,T;H^1)$-norm, that the stability techniques we use are capable of bounding, see \cite{Kolbe2023,Kwon2023}.
As we only provide a bound for the $L^\infty(0,T;L^2)$-error one could wish for an error estimator that scales with order 2 in space. 
For linear parabolic equations, there are techniques achieving this, which involve defining an "elliptic reconstruction", see \cite{Lakkis2015}.
Applying these techniques to the Keller-Segel system poses many challenges, such as choosing a suitable reconstruction for $\rho$ and performing residual estimates as in Subsection~\ref{residual_esti} for such quite possibly non-computable reconstructions.
This is beyond the scope of this paper.

% The Table~\ref{eoc-LinfL2} depicts the $L^\infty(L^2)$-error of the Morley interpolation computed from our numerical scheme. 
% The EOC in space is in average 3.68 and also appears to stabilize in this region. 
% The EOC subject to the scaling in time appears to be around 1.4 to 1.5, before decreasing and tending to zero. 
% This is a typical behavior when the spatial error dominates but the scheme is only refined in time.

%% file: appendix.tex
\appendix
\section{Appendix}
\subsection{Pointwise \emph{a posteriori} error estimates on the flat torus}
\label{A:Linf}
Let $d \in \{2,3\}$. We consider the elliptic and vector-valued problem 
\begin{align}
    \label{Linf:pb}
    q - \Delta q  = f, \quad \text{in } \mathbb{T}^d,
\end{align}
where $f \in L^2(\mathbb{T}^d,\mathbb{R}^d)$ such that $q \in L^\infty(\mathbb{T}^d,\mathbb{R}^d)$, since $H^2(\mathbb{T}^d,\mathbb{R}^d) \hookrightarrow L^\infty(\mathbb{T}^d,\mathbb{R}^d)$, for $d \in \{2,3\}$.
We are interested in pointwise error estimates for this problem.
As we work on the flat torus, the components of our vector-valued solution decouple.
Thus, we may first consider pointwise error estimates for a scalar-valued elliptic problem and in a second step relate this to the vector-valued problem.
Consider (\ref{Linf:pb}), where $f:\mathbb{T}^d \to \mathbb{R}$ is sufficiently regular.
\noindent In order to derive a fully explicit $L^\infty$-error estimator, we follow the strategy of \cite{Demlow2023}.
Let $G_d = G_d(x,\cdot)$ be the scalar-valued Green's function that, for each fixed $x \in \mathbb{T}^d$, satisfies
\begin{align}
    \label{Linf:Green}
    G_d - \Delta G_d = \delta_x(\cdot), \quad \text{in } \mathbb{T}^d,
\end{align}
where $\delta_x(\cdot)$ is the $d$-dimensional Dirac $\delta$-distribution centered in $x$. 
\noindent We need several explicit upper bounds for $G_d$.
As a preparation, we find the following auxiliary result, on tail bounds for certain series.
\begin{Lemma}
    \label{Lemma:series_esti}
    Let $d\in \{2,3\}$, $N \in \mathbb{N}$ and $f:[N-1,\infty) \to [0,\infty)$ be integrable with finite moments up to $d-1$, non-increasing and non-negative.
    Then
    \begin{align*}
        T_N := \sum_{m \in \mathbb{Z}^d : |m| > N} f(|m|).
    \end{align*}
    satisfies
    \begin{align*}
        T_N \leq \begin{cases}
        \int_{N-1}^\infty (6\pi x+ 3\pi)f(x) dx, \quad &d=2, \\
        \int_{N-1}^\infty 12\pi (x^2+ x + 1)f(x) dx, \quad &d=3.
        \end{cases}
    \end{align*}
\end{Lemma}
\begin{proof}
    As $f$ is non-increasing, we have 
    \begin{align*}
        T_N = \sum_{n=N}^\infty \sum_{n < |m| \leq n+1} f(|m|) \leq \sum_{n=N}^\infty |A_n| f(n),
    \end{align*}
    where $A_n := \{m \in \mathbb{Z}^d : n < |m| \leq n+1\}$ and $|A_n|$ is its cardinality, i.e. the number of $m \in \mathbb{Z}^d$ with $ n < |m| \leq n+1$.
    In order to find an upper bound for this value, we observe that identifying $m$ with $Q_m := m+[-\frac12,\frac12]^d$, we have
    \begin{align*}
        |A_n| =  \left| \bigcup_{m \in A_n} Q_m \right|_d,
    \end{align*}
    as the $Q_m$ are pairwise disjoint and satisfy $|Q_m|_d = 1$.
    Then, due to 
    \begin{align*}
        \bigcup_{m \in A_n} Q_m \subset B_{n+2}(0) \setminus B_{n-1}(0),
    \end{align*}
    we find
    \begin{align*}
        |A_n| \leq \left|B_{n+2}(0) \setminus B_{n-1}(0)\right|_d =  \left|B_{n+2}(0)\right|_d -  \left|B_{n-1}(0)\right|_d  = \begin{cases}
            6\pi n+ 3\pi, \quad &d=2, \\
            12\pi (n^2+ n + 1) \quad &d=3.
        \end{cases}
    \end{align*}
    Furthermore, again as $f$ is non-increasing, we find
    \begin{align*}
        f(n) \leq \int_{n-1}^{n}f(x)dx \quad \Longrightarrow \quad \sum_{n=N}^\infty f(n) \leq \int_{N-1}^{\infty}f(x)dx.
    \end{align*}
\end{proof}
\noindent These estimates serve as error estimators for evaluating series of the type 
\begin{align}
    \label{series_esti}
    \sum_{m \in \mathbb{Z}^d \setminus \{0\}} f(|m|) = \underbrace{\sum_{m \in \mathbb{Z}^d\setminus \{0\} : |m| \leq N} f(|m|)}_{\text{numerical approximation}} + \underbrace{\sum_{m \in \mathbb{Z}^d : |m| > N} f(|m|)}_{\text{Lemma~\ref{Lemma:series_esti}}},
\end{align}
numerically, via partial sums.
The desired Green's function bounds read as
\begin{Lemma}
    \label{Linf:Gbounds}
    The Green's function $G_d$, as defined in (\ref{Linf:Green}), satisfies
    \begin{enumerate}[i)]
        \item $\| G_d(x,\cdot)\|_{L^1} \leq C_1 := 1 $
        \item $\| G_d(x,\cdot)\|_{L^2} \leq \tilde{C}_1 := \begin{cases}
            1.002, &\text{if } d=2, \\
            1.006, &\text{if } d=3.
        \end{cases}$
        \item $\| \nabla G_d(x,\cdot)\|_{L^1} \leq C_2 := \begin{cases}
            \frac{\pi}{2}, &\text{if } d=2, \\
            2, &\text{if } d=3.
        \end{cases}$
    \end{enumerate}
    For $0 < r < \frac12$, the Green's function $G_d$ additionally satisfies
    \begin{enumerate}[i)]
        \setcounter{enumi}{3}
        \item $\| \nabla G_d(x,\cdot)  \|_{L^1(B_r(x))} \leq C_3(r) := \begin{cases}
            r + 16 r^2, &\text{if } d=2, \\
            2-(r+2)e^{-r} + 56 r^3, &\text{if } d=3.
        \end{cases}
        $
        \item $ \| D^2 G_d(x,\cdot) \|_{L^1( \mathbb{T}^d \setminus B_r(x))} \leq C_4(r) := \begin{cases}
            \big(\log\big(\frac{1}{\sqrt{2}r}\big) + r K_1(r) - \frac{\sqrt{2}}{2} K_1\big(\frac{\sqrt{2}}{2}\big)\big) + 11, &\text{if } d=2,\\
            2 \log\big(\frac{\sqrt{3}}{2r}\big) + \big(\frac{\sqrt{3}}{2} - r\big)\big( \frac{\sqrt{3}/2+r+4}{2} \big)e^{-r} + 415, &\text{if } d=3.
        \end{cases}$
    \end{enumerate}
    \begin{proof}
        As preparation, we call $g_d$ the Green's function that solves $(1 - \Delta) g_d = \delta_0 \text{ in } \mathbb{R}^d$.
        We use the representation via the modified Bessel functions of second kind $K_0, K_1$ and $K_2$, see \cite[Section 10]{NIST2020}, i.e.
        \begin{align*}
            g_2(x) = \frac{1}{2 \pi} K_0(|x|) \quad \text{and} \quad g_3(x) =  \frac{e^{-|x|}}{4 \pi|x|}.
        \end{align*}
        The functions $G_d$ and $g_d$ relate via periodization, i.e. $G_d(x,z) = \sum_{m \in \mathbb{Z}^d} g_d(x-z+m)$, which allows us to estimate
        \begin{align*}
            \| G_d(x,\cdot)\|_{L^1(\mathbb{T}^d)}  \leq \sum_{m \in \mathbb{Z}^d} \int_{\mathbb{T}^d} |g_d(x-z+m)| dz = \sum_{m \in \mathbb{Z}^d} \int_{x-m-\mathbb{T}^d} |g_d(y) | dy \leq \int_{\mathbb{R}^d} |g_d(y) | dy,
        \end{align*}
        which is the main structure that enables us to show the claimed upper bounds.
        Note that this works for any periodized function, e.g. derivatives of $G_d$.
        \begin{enumerate}[i)]
            \item We observe, via $K_0(s) = \int_0^\infty e^{-s \cosh(t)} dt$, that
            \begin{align*}
                \int_{\mathbb{R}^2} \frac{1}{2 \pi} K_0(|y|) dy = \int_0^\infty s K_0(s) ds &= \int_0^\infty\int_0^\infty s e^{-s \cosh(t)} ds dt 
                = \int_0^\infty \frac{1}{\cosh^2(t)} dt = 1
            \end{align*}
            and
            \begin{align*}
                \int_{\mathbb{R}^3}\frac{e^{-|y|}}{4 \pi|y|} dy = \int_0^\infty s e^{-s} ds = 1.
            \end{align*}
            \item 
                We consider the Fourier series representation of the Green's function
                \begin{align*}
                    G_d(x,y) = \sum_{m \in \mathbb{Z}^d} \frac{e^{2 \pi i m\cdot (x-y)}}{1+4 \pi^2 |m|^2}.
                \end{align*}
                This allows us to represent their $L^2$-norms in series form, i.e.
                \begin{align*}
                    \|G_d(x,\cdot)\|_{L^2(\mathbb{T}^d)}^2 = \sum_{m \in \mathbb{Z}^d} \frac{1}{(1+4 \pi^2 |m|^2)^2}.
                \end{align*}
                We use Lemma~\ref{Lemma:series_esti} with $f(x) = \frac{1}{(1+4 \pi^2 x^2)^2}$, which requires solving the integrals appearing in the tail bounds and numerically computing the partial sums.
                This results in the claimed bounds.
            
            % We obeserve
            % \begin{align*}
            %     \int_{\mathbb{R}^2} \left(\frac{1}{2 \pi} K_0(|y|)\right)^2 dy = \frac{1}{2 \pi} \int_0^\infty s K_0(s)^2 ds &= \frac{1}{2 \pi}\int_0^\infty  \int_0^\infty  \int_0^\infty  s e^{-s (\cosh(t) + \cosh(r))} ds dt dr \\
            %     &= \frac{1}{2 \pi} \int_0^\infty \int_0^\infty \frac{1}{(\cosh(t) + \cosh(r))^2} dt dr =\frac{1}{4 \pi}
            % \end{align*}
            % and 
            % \begin{align*}
            %     \int_{\mathbb{R}^3}\frac{e^{-2|y|}}{16 \pi^2|y|^2} dy = \frac{1}{4 \pi}\int_0^\infty s e^{-2s} ds =  \frac{1}{8 \pi}.
            % \end{align*}
            \item We observe, via $K_1(s) = \int_0^\infty \cosh(t) e^{-s \cosh(t)} dt$, that
            \begin{align*}
                \int_{\mathbb{R}^2} | \nabla g_2(y) | dy = \int_0^\infty s K_1(s) ds = \int_0^\infty\cosh(t)\int_0^\infty s e^{-s \cosh(t)} ds dt =  \int_0^\infty \frac{1}{\cosh(t)} dt = \frac{\pi}{2}
            \end{align*}
            and
            \begin{align*}
                \int_{\mathbb{R}^3} | \nabla g_3(y) | dy = \int_0^\infty (1+s) e^{-s} ds = 2.
            \end{align*}
            \item Next, we need local bounds, i.e. for integrals over (small) balls. We again relate
            \begin{align*}
                \| \nabla G_d(x,\cdot)\|_{L^1(B_r(x))}  = \sum_{m \in \mathbb{Z}^d} \int_{B_r(x)} | \nabla g_d(x-z+m) | dz = \sum_{m \in \mathbb{Z}^d} \underbrace{\int_{B_r(m)} | \nabla g_d(y) | dy}_{ =:\ I_m}.
            \end{align*}
            We split the sum into the principal part $I_0$ and the tail $\sum_{m \in \mathbb{Z}^d\setminus \{0\} }I_m $, where $0 \in \mathbb{R}^d$ denotes the zero vector.
            We find
            \begin{align*}
                I_0 = \int_{B_r(0)} \frac{1}{2 \pi} K_1(|y|) dy \leq \int_0^r s K_1(s) ds \leq r,
            \end{align*}
            as $0 < s K_1(s) < 1$ by $\lim_{s \searrow 0} s K_1(s) = 1$ and monotonicity via $\frac{d}{ds} s K_1(s) = - s K_0(s) < 0$.
            Further, for $|m| \geq 1$ and $r < \frac12$, we find $|m| > 2r$ and thus $|z+m| > \frac{|m|}{2}$ for $z \in B_r(0)$.
            We use monotonicity of $K_1$ due to $\frac{d}{ds} K_1(s) = - \frac12(K_0(s) + K_2(s)) < 0$ and obtain
            \begin{align*}
                I_m = \int_{B_r(0)} \frac{1}{2 \pi} K_1(|y+m|) dy \leq \int_{B_r(0)} \frac{1}{2 \pi} K_1\left(\frac{|m|}{2}\right) dy = \frac12 K_1\left(\frac{|m|}{2}\right)r^2.
            \end{align*}
            Thus, dealing with $\sum_{m \in \mathbb{Z}^2\setminus \{0\} }\frac12 K_1\left(\frac{|m|}{2}\right)$ as in (\ref{series_esti}), gives
            \begin{align*}
                \| \nabla G_2(x,\cdot)\|_{L^1(B_r(x))} \leq r + \sum_{m \in \mathbb{Z}^2\setminus \{0\} } \frac12 K_1\left(\frac{|m|}{2}\right)r^2 \leq r + 16 r^2.
            \end{align*}
            The argument for $d = 3$ works similarly. We find
            \begin{align*}
                I_0 = \int_{B_r(0)} \frac{e^{-|y|}}{4 \pi} \left(  \frac{1}{|y|^2} +   \frac{1}{|y|} \right) dy \leq \int_0^r (1+s) e^{-s} ds = 2 - (r + 2)e^{-r}. 
            \end{align*}
            For the tail, we find for $|m| \geq 1$ again using $|z+m| > \frac{|m|}{2}$ for $z \in B_r(0)$, that
            \begin{align*}
                I_m = \int_{B_r(0)} \frac{e^{-|z+m|}}{4 \pi} \left(  \frac{1}{|z+m|^2} +   \frac{1}{|z+m|} \right) dz &\leq \int_{B_r(0)} \frac{e^{-\frac{|m|}{2}}}{4 \pi} \left( \frac{4}{|m|^2} +   \frac{2}{|m|} \right) dz \\
                 &\leq \frac{r^3}{3}\left(  \frac{4}{|m|^2} +   \frac{2}{|m|} \right)e^{-\frac{|m|}{2}}.
            \end{align*}
            Thus, using (\ref{series_esti}),
            \begin{align*}
                \| \nabla G_3(x,\cdot)\|_{L^1(B_r(x))} &\leq 2 - (r + 2)e^{-r} + \sum_{m \in \mathbb{Z}^3\setminus \{0\} } \frac{r^3}{3}\left(  \frac{4}{|m|^2} +   \frac{2}{|m|} \right)e^{-\frac{|m|}{2}} \\
                &\leq 2 - (r + 2)e^{-r} + 56 r^3.
            \end{align*}
            \item Lastly, we are interested in far-field bounds. Again, similar to before, we observe
            \begin{align*}
                \| D^2 G_d(x,\cdot)\|_{L^1(\mathbb{T}^d \setminus B_r(x))} &= \sum_{m \in \mathbb{Z}^d} \int_{\mathbb{T}^d \setminus B_r(x)} | D^2 g_d(x-z+m) | dz \\
                &= \sum_{m \in \mathbb{Z}^d} \underbrace{\int_{m + \mathbb{T}^d \setminus B_r(m+1)} | D^2 g_d(y)| dy}_{ =:\ I_m},
            \end{align*}
            where by the $1$ in $m+1 \in \mathbb{Z}^d$, we mean the vector consisting only of ones.
            We first compute the spectral norm of the Hessian of $K_0(|x|)$.
            For any radial function $g(|x|)$, we compute
            \begin{align*}
                D^2g(|x|) = g^{\prime \prime}(|x|) \frac{x x^\top}{|x|^2} + g^\prime(|x|)\bigg(\frac{1}{|x|} I -  \frac{x x^\top}{|x|^3}\bigg),
            \end{align*}
            where $I \in \R^{d \times d}$ is the identity matrix.
            For $d=2$, we consider the case $g(|x|) = K_0(|x|)$.
            Using the elementary identities
            \begin{align} 
                K_0^\prime(|x|) &= -K_1(|x|),\label{A:Kv1}\\
                K_0^{\prime \prime}(|x|) &= -K_1^\prime(|x|) = \frac12 (K_0(|x|) + K_2(|x|)) = K_0(|x|) + \frac{1}{|x|} K_1(|x|),\label{A:Kv2}% \\
                % \text{and} \quad K_2(|x|) &= K_0(|x|) + \frac{2}{|x|} K_1(|x|),
            \end{align}
            we find that
            \begin{align*}
                D^2 K_0(|x|) = K_0(|x|) \frac{x x^\top}{|x|^2} + 2K_1(|x|) \frac{x x^\top}{|x|^3} - K_1(|x|)\frac{I}{|x|}.
            \end{align*}
            Setting $u := \frac{x}{|x|}$ and choosing any $v$ such that $u^\top v = 0$, we find 
            \begin{align*}
                D^2K_0(|x|) u = \bigg(K_0(|x|) + \frac{K_1(|x|)}{|x|}\bigg)u \quad \text{and} \quad  D^2K_0(|x|) v = - \frac{K_1(|x|)}{|x|} v.
            \end{align*}
            As $u$ and $v$ span the whole $\mathbb{R}^2$, these are all eigenvalues. Using symmetry of $D^2f(x)$, we find for the spectral norm
            \begin{align*}
                | D^2 K_0(|x|) | = \max\bigg\{\bigg|K_0(|x|) + \frac{K_1(|x|)}{|x|}\bigg|,\left|\frac{K_1(|x|)}{|x|}\right|\bigg\} = K_0(|x|) + \frac{K_1(|x|)}{|x|}.
            \end{align*}
            With this at hand and using periodicity, i.e. $\mathbb{T}^d \setminus B_{r}(1) = \mathbb{T}^d \setminus B_{r}(0)$, as well as $\mathbb{T}^d \cong [-\frac12,\frac12)^d$, we compute
            \begin{align*}
                I_0 \leq \frac{1}{2 \pi}\int_{B_{\frac{\sqrt{2}}{2}}(0) \setminus B_{r}(0)} K_0(|y|) + \frac{K_1(|y|)}{|y|} dy &=  \int_r^\frac{\sqrt{2}}{2} sK_0(s) + K_1(s) ds \\
                &\leq  \bigg(r K_1(r) - \frac{\sqrt{2}}{2} K_1\bigg(\frac{\sqrt{2}}{2}\bigg)\bigg) +  \int_r^\frac{\sqrt{2}}{2} K_1(s) ds.
            \end{align*}
            Using the elementary bound $K_1(r) \leq \frac{1}{r}$ for $0<r<1$, we further estimate
            \begin{align*}
                \int_r^\frac{\sqrt{2}}{2} K_1(s) ds \leq \int_r^\frac{\sqrt{2}}{2} \frac{1}{s} ds = \log\bigg(\frac{\sqrt{2}}{2}\bigg) - \log(r).
            \end{align*}
            All in all,
            \begin{align*}
                I_0 \leq \bigg(r K_1(r) - \frac{\sqrt{2}}{2} K_1\bigg(\frac{\sqrt{2}}{2}\bigg) + \log\bigg(\frac{\sqrt{2}}{2}\bigg) + \log\bigg(\frac{1}{r}\bigg)\bigg).
            \end{align*}
            Next, for $|m| \geq 1$ and $y \in \mathbb{T}^2 \cong [-\frac12,\frac12)^2$, we find $|y+m| > \big||m| - \frac{\sqrt{2}}{2} \big|$.
            Note that $|m| \geq 1$ means, we are away from the singularity of $g_2$, hence we may integrate over the full flat torus, i.e.
            \begin{align*}
                I_m \leq \int_{m + \mathbb{T}^2} |D^2 g_2(y)| dy &\leq  \frac{1}{2 \pi}\int_{\mathbb{T}^2} K_0(|y+m|) + \frac{K_1(|y+m|)}{|y+m|} dy \\
                &\leq \frac{1}{2 \pi} \bigg( K_0\bigg(\bigg||m| - \frac{\sqrt{2}}{2}\bigg|\bigg) + \frac{K_1\big(\big||m| - \frac{\sqrt{2}}{2}\big|\big)}{\big||m| - \frac{\sqrt{2}}{2}\big|} \bigg)=: \tilde{I}_m,
            \end{align*}
            where we used monotonicity of $f(|x|) := K_0(|x|) + \frac{K_1(|x|)}{|x|}$, i.e.
            \begin{align*}
                \frac{d}{d|x|} f(|x|) = - \bigg(K_1(|x|) + \frac{K_0(|x|)}{|x|} + \frac{2K_1(|x|)}{|x|^2} \bigg) < 0,
            \end{align*}
            which is due to (\ref{A:Kv1}) and (\ref{A:Kv2}).
            Therefore, dealing with $\sum_{m \in \mathbb{Z}^2 \setminus \{0\} }\tilde{I}_m$ as in (\ref{series_esti}), leads to
            \begin{align*}
                \| D^2 G_2(x,\cdot)\|_{L^1(\mathbb{T}^2 \setminus B_r(x))} %\leq 3 \ln\left(\frac{\sqrt{2}}{2r}\right) + (r + 4)e^{-r} - \left(\frac{\sqrt{2}}{2} + 4\right)e^{-\frac{\sqrt{2}}{2}} + \sum_{m \in \mathbb{Z}^d\setminus \{0\} } \frac{1}{2 \pi} \left( \frac{12}{|m|^2} + \frac{6}{|m|} + 1 \right) e^{-\frac{|m|}{2}} \
                \leq 2\pi \bigg(r K_1(r) - \frac{\sqrt{2}}{2} K_1\bigg(\frac{\sqrt{2}}{2}\bigg) + \log\bigg(\frac{\sqrt{2}}{2}\bigg) + \log\bigg(\frac{1}{r}\bigg)\bigg) + 11.
            \end{align*}
            Using the same arguments, for $d=3$, we have
            \begin{align*}
                I_0 \leq \int_{B_{\frac{\sqrt{3}}{2}}(0) \setminus B_{r}(0)} \frac{e^{-|y|}}{4 \pi}\bigg( \frac{1}{|y|} + \frac{2}{|y|^2} + \frac{2}{|y|^3} \bigg) dy &\leq \int_r^{\frac{\sqrt{3}}{2}} e^{-s} \bigg(s+2+\frac{2}{s}\bigg) ds \\
                    &\leq 2 \log\bigg(\frac{\sqrt{3}}{2r}\bigg) + \bigg(\frac{\sqrt{3}}{2} - r\bigg)\bigg( \frac{\sqrt{3}/2+r+4}{2} \bigg)e^{-r}.
            \end{align*}
            For $|m| \geq 1$, using $|y+m| \geq \big||m| - \frac{\sqrt{3}}{2} \big|$, we find
            \begin{align*}
                I_m %&\leq \int_{\mathbb{T}^3} \frac{e^{-|y+m|}}{4 \pi}\left( \frac{1}{|y+m|} + \frac{4}{|y+m|^2} + \frac{4}{|y+m|^3} \right)\\
                \leq \frac{e^{-\big||m| - \frac{\sqrt{3}}{2} \big|}}{4 \pi}\bigg( \frac{1}{\big||m| - \frac{\sqrt{3}}{2} \big|} + \frac{2}{\big||m| - \frac{\sqrt{3}}{2} \big|^2} + \frac{2}{\big||m| - \frac{\sqrt{3}}{2} \big|^3} \bigg).
            \end{align*}
            All in all, we have
            \begin{align*}
                 \| D^2 G_3(x,\cdot)\|_{L^1(\mathbb{T}^3 \setminus B_r(x))} %&\leq 4 \ln\left(\frac{\sqrt{3}}{2r}\right) + (r+5)e^{-r} - (\frac{\sqrt{3}}{2} + 5)e^{-\frac{\sqrt{3}}{2}} + \sum_{m \in \mathbb{Z}^d\setminus \{0\} } \frac{e^{-\frac{|m|}{2}}}{4 \pi}\left( \frac{2}{|m|} + \frac{16}{|m|^2} + \frac{32}{|m|^3} \right) \\
                 \leq  2 \log\bigg(\frac{\sqrt{3}}{2r}\bigg) + \bigg(\frac{\sqrt{3}}{2} - r\bigg)\bigg( \frac{\sqrt{3}/2+r+4}{2} \bigg)e^{-r} + 415\ .
            \end{align*}
        \end{enumerate}
    \end{proof}
\end{Lemma}
\noindent Note that these bounds asymptotically, i.e. for $r \searrow 0$, correspond to \cite[(2.5)]{Demlow2023}.
\noindent In order to prepare the scalar-valued a posteriori error estimate, we are interested in pointwise representations of the solution to (\ref{Linf:pb}) as well as its numerical approximation.
Let $q_h \in P_1(\mathcal{T}_h) \subset H^1 \cap C^0$ be a finite element approximation to the solution $q$ of (\ref{Linf:pb}), i.e.
\begin{align*}
    \langle \nabla q_h, \nabla v_h \rangle_{L^2} + \langle q_h, v_h\rangle_{L^2} = \langle f, v_h \rangle_{L^2}, \quad \text{for all } v_h \in P_1(\mathcal{T}_h).
\end{align*}
Computing this approximation is again susceptible to round-off errors and algebraic errors. We introduce the perturbed system $\tilde{\mathbb{S}}_3 Q_a =: \tilde{H}_3^a$ and the corresponding approximation $q_{h,a}$ and the algebraic residual function $r_\text{FE}$, similar to Subsection~\ref{alg-error}.
Then,
\begin{align}
    \label{Linf:FEM}
    \langle \nabla q_{h,a}, \nabla v_h \rangle_{L^2} + \langle q_{h,a}, v_h\rangle_{L^2} = \langle f + r_\text{FE}, v_h \rangle_{L^2} , \quad \text{for all } v_h \in P_1(\mathcal{T}_h).
\end{align}
Using the Green's function $G$, solving (\ref{Linf:Green}), we can represent the solution of (\ref{Linf:pb}) by 
\begin{align}
    \label{A:obs1}
    q(x) = \langle G(x,\cdot), f(\cdot) \rangle_{L^2},
\end{align}
and, any $v \in H^1 \cap C^0 $ by
\begin{align}
    \label{A:obs2}
    v(x) = \langle \nabla v, \nabla G(x,\cdot)\rangle_{L^2} + \langle v, G(x,\cdot) \rangle_{L^2}.
\end{align}
Using (\ref{A:obs1}) and (\ref{A:obs2}) with $v = q_{h,a}$, the error has the pointwise representation
\begin{align}
    \label{Linf:error_eq}
    (q_{h,a} - q)(x) = \langle \nabla q_{h,a}, \nabla G(x,\cdot)\rangle_{L^2} + \langle q_{h,a} - f, G(x,\cdot) \rangle_{L^2}.
\end{align}
Furthermore, we will need a stability estimate for the Scott-Zhang interpolant.
Let $G_h \in P_1(\mathcal{T}_h)$ be the Scott-Zhang interpolant of $G$. It satisfies the local stability and approximation property
\begin{align}
    \label{Linf:sz}
    | G - G_h |_{W^{k,1}(T)} \leq C_{k,j}^\text{sz} h_T^{j-k} | G |_{W^{j,1}(\omega_T)},
\end{align}
for $0\leq k \leq 1 \leq j \leq2$,
and 
\begin{align*}
    \| G_h \|_{L^2(T)} \leq \tilde{C}^\text{sz} \| G \|_{L^2(\omega_T)},
\end{align*}
where $\omega_T$ is a suitable Scott-Zhang patch of elements in $\mathcal{T}_h$.
We additionally need the local trace inequality
\begin{align*}
    \| g \|_{L^1(F)} \leq \tilde{c}_\text{Tr} \left(\| \nabla g \|_{L^1(T)} + h_T^{-1} \| g \|_{L^1(T)}\right),
\end{align*}
for all faces $F \in F_h$ with $F \subset \partial K$ for $K \in \mathcal{T}_h$, and functions $g \in H^1(\mathbb{T}^d)$.
These constants can all be determined explicitly by investigating the corresponding proofs of the estimates.
\begin{Theorem}[Scalar-valued error estimates]
    Let $d\in \{2,3\}$. Let $q \in L^\infty(\mathbb{T}^d)$ be the solution to (\ref{Linf:pb}) and $q_{h,a} \in P_1(\mathcal{T}_h)$ its perturbed P1 finite element approximation.
    Then
    \begin{align*}
        \| q_{h,a} - q\|_{L^\infty(\mathbb{T}^d)} \leq \max_{T \in \mathcal{T}_h} \left\{  \alpha_h \|q_{h,a} - \Delta q_{h,a} -f \|_{L^\infty(T)} + \beta_h \| \llbracket \nabla q_{h,a} \rrbracket \|_{L^\infty(\partial T)} \right\} + \eta_{a} =: \eta_\infty^\text{sc}(q_{h,a}),
    \end{align*}
    where 
    \begin{align*}
        \alpha_h &:=  C_\text{ol} (C_{0,2}^\text{sz} C_4(\underline{h}) h^2  + C_{0,1}^\text{sz} C_3(3h) h), \\
        \beta_h &:= \tilde{c}_\text{Tr} C_\text{ol} \big((C_{1,2}^\text{sz} + C_{0,2}^\text{sz}) C_4(\underline{h}) h  + (C_{1,1}^\text{sz} + C_{0,1}^\text{sz})C_3(3h)\big) , \\
        \eta_a &:= C_\text{ol}\tilde{C}^\text{sz} \tilde{C}_1 \bigg(\frac{(1 - 2 C_{10}\varepsilon)\| \tilde{\mathbb{M}}^{-1} \|_2}{1 - (2 + k\|\tilde{\mathbb{M}}\|_2\|\tilde{\mathbb{M}}^{-1}\|_2 )C_{10}\varepsilon}\bigg)^{1/2} \bigg(\frac{C_{11}(h) \varepsilon}{1-2C_{11}(h) \varepsilon} \| \tilde{H}_2\|_2 + \| \tilde{H}_2^a -\tilde{H}_2 \|_2 \\
        &\quad\quad\quad\quad\quad\quad\quad\quad\quad\quad\quad\quad\quad\quad\quad\quad\quad\quad\quad\quad\quad\quad\quad + \bigg(\frac{C_{12}(h)\varepsilon}{1-C_{12}(h)\varepsilon}\bigg) \|\tilde{\mathbb{S}}_2\|_2 \|C_a^{n}\|_2 \bigg),
    \end{align*}
    with $C_\text{ol}$ is the maximal number of overlapping elements of two different patches, which can be computed explicitly once the mesh sequence is fixed, and $\underline{h} := \min_{T \in \mathcal{T}_h} \left\{ h_T \right\}$ the minimal mesh size,
    as well as $C_{10}(h), C_{11}(h), C_{12}(h) > 0$ are the numbers of elementary operations needed to compute the corresponding perturbed mass matrix $\tilde{\mathbb{M}}_3$, $\tilde{H}_3$ and $\tilde{\mathbb{S}}_3$, respectively.
\end{Theorem}
\begin{proof}
    Let us fix an arbitrary $x \in \mathbb{T}^d$ and the corresponding element $T_0 \in \mathcal{T}_h$ with $x \in T_0$. 
    We set $g(\cdot) := g(x,\cdot) := G(x,\cdot) - G_h(x,\cdot)$.
    Subtracting (\ref{Linf:FEM}), tested with $G_h \in P_1(\mathcal{T}_h)$, from (\ref{Linf:error_eq}), leads to
    \begin{align*}
        (q_{h,a} - q)(x) &= \langle \nabla q_{h,a}, \nabla g\rangle_{L^2} + \langle q_{h,a} - f , g \rangle_{L^2} + \langle r_\text{FE},G_h \rangle_{L^2} \\
        &= \sum_{K \in \mathcal{T}_h}\langle \nabla q_{h,a}, \nabla g\rangle_{L^2(K)} + \sum_{K \in \mathcal{T}_h}\langle q_{h,a} - f , g \rangle_{L^2(K)} + \langle r_\text{FE},G_h \rangle_{L^2}\\
        &= \sum_{K \in \mathcal{T}_h}\langle q_{h,a} - \Delta q_{h,a} - f, g \rangle_{L^2(K)} + \sum_{K \in \mathcal{T}_h} \int_{\partial K} (\nabla q_{h,a} \cdot n_T) g d S(x) + \langle r_\text{FE},G_h \rangle_{L^2}\\
        % &=\sum_{T \in \mathcal{T}_h}\| q_h - \Delta q_h - f \|_{L^\infty(T)} \| g \|_{L^1(T)} + \sum_{F \in S_h} \|  \llbracket \nabla q_h \rrbracket \|_{L^\infty(F)}\| g \|_{L^1(F)} \\
        & \leq \max_{K\in\mathcal{T}_h}\{ \| q_{h,a} - \Delta q_{h,a} - f \|_{L^\infty(K)} \} \sum_{K \in \mathcal{T}_h}\| g \|_{L^1(K)} + \max_{K\in\mathcal{T}_h} \{\|  \llbracket \nabla q_{h,a} \rrbracket \|_{L^\infty(\partial K)} \sum_{F \in S_h} \| g \|_{L^1(F)}\} \\
        &\quad\quad\quad + C_\text{ol}\tilde{C}^\text{sz} \tilde{C}_1 \| r_\text{FE}\|_{L^2}
    \end{align*}
    Furthermore, using (\ref{Linf:sz}), we find
    \begin{align*}
        \| g \|_{L^1(K)} \leq C_{0,2}^\text{sz} h_K^2 \| D^2 G\|_{L^1(\omega_K)},
    \end{align*}
    for all $K \in \mathcal{T}_h$ with $\omega_K \cap T_0 = \emptyset$, as $D^2G(x,\cdot) \notin L^1_\text{loc}$ due to its singularity in $x \in T_0$.
    In order to obtain a finite upper bound, we additionally use
    \begin{align*}
        \| g \|_{L^1(T)} \leq C_{0,1}^\text{sz} h_{T} \| \nabla G\|_{L^1(\omega_T)},
    \end{align*}
    for all $T \in \mathcal{S}_{0}:= \{K \in \mathcal{T}_h : \omega_K \cap T_0 \neq \emptyset \}$.
    Setting $\underline{h} := \min_{K \in \mathcal{T}_h} \left\{ h_K \right\}$ as above, 
    we observe
    \begin{align*}
         \sum_{K \in \mathcal{T}_h} \| g \|_{L^1(K)} &\leq  \sum_{K \in {\mathcal{T}_h\setminus \mathcal{S}_0}} C_{0,2}^\text{sz} h_K^2 \| D^2 G\|_{L^1(\omega_K)} +  \sum_{T \in \mathcal{S}_0} C_{0,1}^\text{sz} h_T \| \nabla G\|_{L^1(\omega_T)} \\
         &\leq C_\text{ol} C_{0,2}^\text{sz} h^2 \| D^2 G\|_{L^1(\mathbb{T}^d \setminus B_{\underline{h}}(x))} +  C_{0,1}^\text{sz} h \| \nabla G\|_{L^1(B_{3h}(x))} \leq \alpha_{h},
    \end{align*}
    where in the last step, we used Lemma \ref{Linf:Gbounds}.
    Similarly, together with the local trace inequality, we find
    \begin{align*}
        \sum_{F \in S_h} \| g \|_{L^1(F)} &\leq  \tilde{c}_\text{Tr}\sum_{K \in \mathcal{T}_h}\left( \| \nabla g\|_{L^1(K)} + h_K^{-1} \| g \|_{L^1(K)} \right) \\
        &\leq \tilde{c}_\text{Tr} C_\text{ol}(C_{1,2}^\text{sz} + C_{0,2}^\text{sz}) h \| D^2 G\|_{L^1(\mathbb{T}^d \setminus B_{\underline{h}}(x))}  + (C_{1,1}^\text{sz} + C_{0,1}^\text{sz})\| \nabla G\|_{L^1(B_{3h}(x))} \\
        &\leq \beta_h.
    \end{align*}
    This implies the desired estimator.
\end{proof}
\noindent We finally have
\begin{Corollary}[Vector-valued error estimates]
    Let $d\in \{2,3\}$. Let $q \in L^\infty(\mathbb{T}^d,\mathbb{R}^d)$ be the solution to (\ref{Linf:pb}) and $q_h = (q_{h}^1, \ldots, q_{h}^d) \in [P_1(\mathcal{T}_h)]^d$ its P1 finite element approximation.
    Then
    \begin{align*}
        \| q_{h,a} - q\|_{L^\infty(\mathbb{T}^d)} \leq \bigg(\sum_{i=1}^d \eta_\infty^\text{sc}(q_{h,a}^i)^2 \bigg)^{1/2} =: \eta_\infty,
    \end{align*}
    where $\eta_\infty^\text{sc}(\cdot)$ is defined as in the Theorem above.
\end{Corollary}

\subsection{Closed form of the residual estimator in Theorem \ref{main_thm}}
\label{closedform}
We have shown in Subsection~\ref{residual_esti} that
\begin{align*}
    \int_0^T \norm{R_{\tilde{\rho}_a}}^2_{(H^1(\mathbb{T}^d))^\prime} \ dt \leq \sum_{n=0}^{N_t-1}  \int_{t^n}^{t^{n+1}} \Theta_n(t)^2 dt = \sum_{n=0}^{N_t-1}  \frac{\Delta t^n}{2} (\Theta_n(\tau_-^n)^2 + \Theta_n(\tau_+^n)^2),
\end{align*}
where for $m = 1, \ldots, N_t -1$, $\Theta_{m}(t)$ is defined by summing the bounds of Lemmas~\ref{diff-res}, \ref{time-res}, \ref{conv-res} and \ref{alg-res} using (\ref{locres}). 
For $m=0$, $\Theta_{0}(t)$ is defined by summing the bounds of Lemmas~\ref{diff-res0}, \ref{time-res0}, \ref{conv-res0} and \ref{alg-res} using (\ref{locres}) as well as (\ref{earlyextra}).
The quadrature points $\tau_{\pm}^m := \frac{t^m + t^{m+1}}{2} \pm \frac{\Delta t^m}{2 \sqrt{3}}$ stem from the two-point Gaussian quadrature, which exactly integrates $\Theta_m(t)$ on $[t^m,t^{m+1}]$, as it is a polynomial of degree 2.
\noindent Explicitly, for $n = 1, \ldots, N_t -1$, we have

\begin{align*}
     \Theta_{n}(t) &:= \ell_0^n(t)c_{\text{P}} \bigg( \sum_{T \in \mathcal{C}_{\mathcal{S}_{\mathcal{T}_h}}^{\mathcal{T}_h}}\|\frac{\tilde{\rho}_a^{n+1}-\tilde{\rho}_a^{n}}{\Delta t^{n}} + \divergence\left( \tilde{\rho}_a^{n+1} \nabla c_{h,a}^{n} \right) - \Delta \tilde{\rho}_a^{n+1}\|_{L^2(T)}^2  h_{K_T}^2 \bigg)^{1/2}\\
      &+ \ell_1^n(t)c_{\text{P}} \bigg( \sum_{T \in \mathcal{C}_{\mathcal{S}_{\mathcal{T}_h}}^{\mathcal{T}_h}} \|\frac{\tilde{\rho}_a^{n+1}-\tilde{\rho}_a^{n}}{\Delta t^{n}} + \divergence\left( \tilde{\rho}_a^{n} \nabla c_{h,a}^{n-1} \right) - \Delta \tilde{\rho}_a^{n}\|_{L^2(T)}^2  h_{K_T}^2 \bigg)^{1/2} \\
     & + \ell_0^n(t)\Theta^{n+1}_{\text{dj}} + \ell_1^n(t)\Theta^{n}_{\text{dj}}    
     + \Theta^{n}_{\text{t}}(t) 
     + \ell_0^n(t)\Theta^{n+1}_{\text{cj}} + \ell_1^n(t)\Theta^{n}_{\text{cj}}  \\
     &+ 12 C_3(h) C_{\text{S}} \varepsilon  \sum_{K \in \mathcal{T}_h} \bigg(\frac{1}{|K|} \bigg)^{1/6} \\
    & + C_{\text{ell}}\sum_{K \in \mathcal{T}_h}\big(\big(\norm{\tilde{\rho}_a^{n+1}}_{L^\infty} + \norm{\tilde{\rho}_a^{n} - \tilde{\rho}_a^{n+1}}_{L^\infty} \big)\norm{\tilde{\rho}_a^{n} - \tilde{\rho}_a^{n+1}}_{L^2(K)} + \norm{\tilde{\rho}^{n}_a}_{L^\infty}  \norm{\tilde{\rho}_a^{n-1} - \tilde{\rho}_a^{n}}_{L^2(K)}\big) \\
    &+  \ell_0^n(t) \norm{\tilde{\rho}_a^{n+1}}_{L^\infty} (\Theta^{n}_{\text{FE}} +\Theta^{n}_{\text{aFE}}) +  \ell_1^n(t) \norm{\tilde{\rho}_a^{n}}_{L^\infty} (\Theta^{n-1}_{\text{FE}} + \Theta^{n-1}_{\text{aFE}})
    + \ell_0^n(t)\Theta^{n+1}_{\text{aFV}}
    + \ell_1^n(t)\Theta^{n}_{\text{aFV}},
\end{align*}
where
\begin{align*}
    %\Theta^{n}_{\text{res}} &:= c_{\text{P}} \bigg( \sum_{K \in \mathcal{T}_h}\|\frac{\tilde{\rho}_a^{n+1}-\tilde{\rho}_a^{n}}{\Delta t^{n}} + \divergence\left( \tilde{\rho}_a^{n} \nabla c_{h,a}^{n-1} \right) - \Delta \tilde{\rho}_a^{n}\|_{L^2(T)}^2  h_{K_T}^2 \bigg)^{1/2}\\
    \Theta^{n}_{\text{dj}} &:= c_R \bigg( \sum_{F \in F_h} h_F \norm{\llbracket \nabla \tilde{\rho}_a^{n}\cdot n_{K,F} \rrbracket}_{L^2(F)}^2 \bigg)^{1/2} \\
    \Theta^{n}_{\text{t}}(t) &:= \bigg(\sum_{K \in \mathcal{T}_h} \|\frac{\tilde{\rho}_a^{n+1}-\tilde{\rho}_a^{n}}{\Delta t^n} - \frac{\rho_{K,a}^{n+1}-\rho_{K,a}^{n}}{\Delta t^n}\|_{L^2(K)}^2\bigg)^{1/2}, \\
     & \quad\quad\quad\quad\quad\quad\quad\quad\quad\quad\quad\quad\quad\quad\quad + \ell_1^n(t)\bigg(\sum_{K \in \mathcal{T}_h} |K| \bigg(\frac{\rho_{K,a}^{n+1}-\rho_{K,a}^{n}}{\Delta t^n} - \frac{\rho_{K,a}^{n}-\rho_{K,a}^{n-1}}{\Delta t^{n-1}}\bigg)^2 \bigg)^{1/2}, \\
    \Theta^{n}_{\text{cj}} &:= {\tilde{c}_R} \bigg(\sum_{S \in C^{S_h}_h} h_S \norm{\llbracket \nabla c_{h,a}^{n-1} \cdot n_{T,S} \rrbracket \tilde{\rho}_a^{n} }_{L^2(S)}^2\bigg)^{1/2} \\
    & \quad\quad\quad\quad\quad\quad\quad\quad\quad\quad\quad\quad + {\tilde{c}_R} \bigg(\sum_{F \in C^{F_h}_h} h_F \norm{\big(\nabla c_{h,a}^{n-1} \cdot n_{T,F}\big) \big(\tilde{\rho}_a^{n} - \frac{1}{|F|} \mathcal{C}_F(\rho_{h,a}^{n-1})\big)}_{L^2(F)}^2 \bigg)^{1/2},  \\
    \Theta^{n}_{\text{FE}} &:= \tilde{C}^{\text{sz}} \bigg(  \sum_{T \in \mathcal{S}_{\mathcal{T}_h}}  \big(h_T^2\norm{c_{h,a}^{n} - \Pi_1 \tilde{\rho}_a^{n}}_{L^2(T)}^2 + \frac{1}{2}\sum_{F \in S_T} h_F \norm{\llbracket \nabla c_{h,a}^{n} \cdot n_F\rrbracket}_{L^2(F)}^2 \\
    &  \quad\quad\quad\quad\quad\quad\quad\quad\quad\quad\quad\quad\quad\quad\quad\quad\quad + h_{T}^2 \norm{\tilde{\rho}_a^{n} - \Pi_1\tilde{\rho}_a^{n}}_{L^2(T)}^2 + \frac43 \norm{\rho_{h,a}^{n} - \tilde{\rho}_a^{n}}_{L^2(T)}^2\big) \bigg)^{1/2} , \\
    \Theta^{n}_{\text{aFE}} &:= \tilde{C}^{\text{sz}}\bigg(\frac{(1 - 2 C_4\varepsilon)\| \tilde{\mathbb{M}}^{-1} \|_2}{1 - (2 + k \|\tilde{\mathbb{M}}\|_2\|\tilde{\mathbb{M}}^{-1}\|_2 )C_4\varepsilon} \bigg)^{1/2} \bigg( \frac{C_5(h) \varepsilon}{1-2C_5(h) \varepsilon} \| \tilde{H_2}\|_2 + \| \tilde{H}_2^a -\tilde{H}_2 \|_2\\
    & \quad\quad\quad\quad\quad\quad\quad\quad\quad\quad\quad\quad\quad\quad\quad\quad\quad\quad\quad\quad\quad\quad\quad + \bigg(\frac{C_6(h)\varepsilon}{1-C_6(h)\varepsilon}\bigg) \|\tilde{\mathbb{S}}_2\|_2 \|C_a^{n}\|_2 \bigg), \\
    \Theta^{n}_{\text{aFV}} &:= \frac{\max_{K \in \mathcal{T}_h} |K|^{1/2}}{\Delta t^{n-1}} \bigg( \frac{C_1(h) \varepsilon}{1-2C_1(h) \varepsilon} \|\tilde{H}_1\|_2 + \| \tilde{H}_1^a -\tilde{H}_1 \|_2 + \frac{C_2(h) \varepsilon}{1-2C_2(h) \varepsilon} \|\tilde{\mathbb{S}}_1\|_2 \| P_a^{n} \|_2 \bigg).
\end{align*}
\noindent Additionally, for $m=0$, this means
\begin{align*}
     \Theta_{0}(t) &:=c_{\text{P}} \bigg( \sum_{K \in \mathcal{T}_h}\|\frac{\tilde{\rho}_a^{1}-\tilde{\rho}_a^{0}}{\Delta t^{0}} + \divergence\left( \tilde{\rho}_a^{1} \nabla c_{h,a}^{0} \right) - \Delta \tilde{\rho}_a^{1}\|_{L^2(T)}^2  h_{K_T}^2 \bigg)^{1/2} \\
     & + \Theta^{1}_{\text{dj}} + \Theta^{1}_{\text{cj}} + (\norm{\tilde{\rho}_a^1}_{L^\infty} + \ell_1^0(t)\tilde{C}^{\text{sz}}\norm{\tilde{\rho}_a^{1} - \tilde{\rho}_a^{0}}_{L^\infty}) (\Theta^{0}_{\text{FE}} + \Theta^{0}_{\text{aFE}}) + \Theta^{1}_{\text{aFV}} \\
    &+ 12 C_3(h) C_{\text{S}} \varepsilon  \sum_{K \in \mathcal{T}_h} \bigg(\frac{1}{|K|} \bigg)^{1/6}
    + \bigg(\sum_{K \in \mathcal{T}_h}\|\frac{\tilde{\rho}_a^{1}-\tilde{\rho}_a^{0}}{\Delta t^0} - \frac{\rho_{K,a}^{1}-\rho_{K,a}^{0}}{\Delta t^0}\|_{L^2(K)}^2\bigg)^{1/2} \\
    & + C_{\text{ell}}\sum_{K \in \mathcal{T}_h}\big(\norm{\tilde{\rho}_a^{1}}_{L^\infty} + \norm{\tilde{\rho}_a^{0} - \tilde{\rho}_a^{1}}_{L^\infty} \big)\norm{\tilde{\rho}_a^{0} - \tilde{\rho}_a^{1}}_{L^2(K)} \\
    &+ \ell_1^0(t)\bigg(\sum_{K \in \mathcal{T}_h}\norm{ (\tilde{\rho}_a^1 - \tilde{\rho}_a^0 )\nabla c_{h,a}^0 - \nabla (\tilde{\rho}_a^{1} - \tilde{\rho}_a^{0})}_{L^2(K)}^2 \bigg)^{1/2}.
\end{align*}

\subsection{Explicit constants for certain Sobolev embeddings}
\label{subsec:sobolev}
We derive explicit upper bounds for different Sobolev embedding constants for $d \in \{2,3\}$, especially for the embedding $H^1 \hookrightarrow L^6$.
We work on the flat torus \(\T^d=\mathbb R^d/\mathbb Z^d=[0,1)^d\) and denote $\T := \T^1$. As an auxiliary result, we find
\begin{Lemma}[Oscillation lemma]
    \label{lem:osc}
Let $h:\mathbb{T}\to[0,\infty)$ be absolutely continuous. Then
\[
h(x) \le \int_0^1 h(t)\,dt + \frac12 \int_0^1 |h'(t)|\,dt.
\]
\end{Lemma}
\begin{proof}
Let $t_{\min}, t_{\max}\in[0,1)$ be points where $h$ attains its minimum and its maximum on $\mathbb{T}$. Let $P_1,P_2$ be the two arcs joining $t_{\min}$ to $t_{\max}$ (in the two directions around the flat torus). The fundamental theorem of calculus along each arc gives
\[
h(t_{\max}) - h(t_{\min}) = \int_{P_j} h'(t)\,dt \le \int_{P_j} |h'(t)|\,dt, \qquad j=1,2.
\]
Adding the two inequalities, and using that $P_1\cup P_2$ is the whole flat torus up to the two endpoints, gives
\[
2\big(h(t_{\max}) - h(t_{\min})\big) \le \int_0^1 |h'(t)|\,dt.
\]
Since $t_{\min}$ minimizes $h$, we have $h(t_{\min}) \le \int_0^1 h(t)\,dt$. Hence for every $x\in\mathbb{T}$,
\[
h(x) \le h(t_{\max}) \le h(t_{\min}) + \frac12\int_0^1 |h'| \le \int_0^1 h(t)\,dt + \frac12\int_0^1 |h'(t)|\,dt. \qedhere
\]
\end{proof}
\subsubsection{The embedding $H^1 \hookrightarrow L^p$ on $\T^2$}
We first show a general statement for $d=2$.
\begin{Proposition}
    \label{prop:CS2D}
    Let $p \geq 2$ be an integer and let $u \in H^1(\T^2)$.
    Then
    \[
    \norm u_{L^p(\T^2)}
    \le C_{p}^{\rm it}\norm u_{H^1(\T^2)},
    \]
    where, for $\nu:=\left\lfloor\frac p2\right\rfloor$ and $C_{p}(\theta):=1-\theta+\frac{p}{2\sqrt2}\sqrt{\theta(1-\theta)}+\frac{p^2}{32}\theta$, the constant can be obtained via a one-dimensional optimization, i.e.
    \begin{align*}
        C_{p}^{\rm it} := \begin{cases}
        \sup_{0\le\theta\le1}\left((1-\theta)\prod_{j=0}^{\nu-2} C_{p-2j}(\theta)\right)^{1/p}, \quad &\text{if $p$ is even}, \\
        \sup_{0\le\theta\le1}\left((1-\theta)C_4(\theta)^{1/2}\prod_{j=0}^{\nu-2}C_{p-2j}(\theta)\right)^{1/p}, \quad &\text{else.}
        \end{cases}
    \end{align*}
\end{Proposition}
\begin{proof}
    For brevity, we set 
    \[
    Y_r:=\int_{\mathbb{T}^2}\abs u^r\,d x, 
    \qquad
    g_i:=\int_{\mathbb{T}^2}\abs{\partial_i u}^2\,d x \ \  (i = 1,2),
    \qquad
    G:=g_1+g_2. 
    \]
    \textit{Step 1: } Let \(p\ge4\) be finite. Applying Lemma~\ref{lem:osc} in the \(x_2\)-direction gives
\begin{equation}\label{eq:x2}
\abs{u(x_1,x_2)}^{p/2}
\le
\int_0^1\abs{u(x_1,t)}^{p/2}\,d t
+\frac p4\int_0^1
\abs{u(x_1,t)}^{p/2-1}\abs{\partial_2u(x_1,t)}\,d t.
\end{equation}
Likewise, in the \(x_1\)-direction,
\begin{equation}\label{eq:x1}
\abs{u(x_1,x_2)}^{p/2}
\le
\int_0^1\abs{u(t,x_2)}^{p/2}\,d t
+\frac p4\int_0^1
\abs{u(t,x_2)}^{p/2-1}\abs{\partial_1u(t,x_2)}\,d t.
\end{equation}
\noindent Multiplying \eqref{eq:x2} and \eqref{eq:x1}, integrating over $\T^2$, and using Fubini gives
\begin{align}
Y_p
\le Y_{p/2}^2 &+\frac p4 Y_{p/2}
\left(
\int_{\T^2}\abs{\partial_1u}\abs u^{p/2-1}\,d x
+\int_{\T^2}\abs{\partial_2u}\abs u^{p/2-1}\,d x
\right)
\nonumber\\
&+\frac{p^2}{16}
\left(\int_{\T^2}\abs{\partial_1u}\abs u^{p/2-1}\,d x\right)
\left(\int_{\T^2}\abs{\partial_2u}\abs u^{p/2-1}\,d x\right).
\label{eq:2d-expanded}
\end{align}
\noindent We additionally observe
\begin{equation}\label{eq:interpolation-p-half}
Y_{p/2}
=\norm u_{L^{p/2}}^{p/2}
\le \norm u_{L^2}\,\norm u_{L^{p-2}}^{(p-2)/2}
= Y_2^{1/2}Y_{p-2}^{1/2}.
\end{equation}
Moreover, by Cauchy--Schwarz,
\begin{equation}\label{eq:cs-directional}
\int_{\T^2}\abs{\partial_i u}\abs u^{p/2-1}\,d x
\le g_i^{1/2}Y_{p-2}^{1/2}.
\end{equation}
Substituting \eqref{eq:interpolation-p-half} and \eqref{eq:cs-directional} into \eqref{eq:2d-expanded}, gives
\begin{align*}
Y_p
\le &
\left(
Y_2
+\frac p4(\sqrt{g_1}+\sqrt{g_2})\sqrt{Y_2}
+\frac{p^2}{16}\sqrt{g_1g_2}
\right)Y_{p-2}.
\end{align*}
Using
\[
\sqrt{g_1}+\sqrt{g_2}\le\sqrt{2G},
\qquad
\sqrt{g_1g_2}\le\frac G2,
\]
we arrive at
\begin{equation}\label{eq:2d-onestep}
Y_p\le
\left(
Y_2+\frac{p}{2\sqrt2}\sqrt{G Y_2}+\frac{p^2}{32}G
\right)Y_{p-2}.
\end{equation}
\noindent Next, we normalize this inequality by the full \(H^1\)-norm.
To this end, we set
\[
N:=\norm u_{H^1(\T^2)}^2=Y_2+G, % originally E
\quad
\theta:=\frac{G}{N}\in[0,1], \quad \text{and} \quad A_r:=\frac{Y_r}{N^{r/2}}
=\left(\frac{\norm u_{L^r(\T^2)}}{\norm u_{H^1(\T^2)}}\right)^r.
\]
It holds  $A_2=1-\theta$.
Dividing \eqref{eq:2d-onestep} by \(N^{p/2} = N \cdot N^{(p-2)/2}\) gives the scalar recurrence
\begin{equation}\label{eq:2d-recursion}
    A_p\le C_p(\theta)A_{p-2}.
\end{equation}
\textit{Step 2: } Using $\nu$ as defined above, we set $q:=p-2(\nu-1)\in\{2,3\}$.
Iterating \eqref{eq:2d-recursion} yields
\begin{align*}
    A_p\le\left(\prod_{j=0}^{\nu-2} C_{p-2j}(\theta)\right)A_q.
\end{align*}
\noindent For even $p$, we have \(q=2\) and therefore
\begin{align*}
    A_p\le A_2\prod_{j=0}^{\nu-2} C_{p-2j}(\theta) = (1-\theta)\prod_{j=0}^{\nu-2} C_{p-2j}(\theta).
\end{align*}
\noindent For odd \(p\), we have \(q=3\). Then, interpolation gives $\norm u_{L^3}\le \norm u_{L^2}^{1/3}\norm u_{L^4}^{2/3}$, so
\begin{align*}
    A_p\le A_3\prod_{j=0}^{\nu-2} C_{p-2j}(\theta) \le (1-\theta)\,C_4(\theta)^{1/2}\prod_{j=0}^{\nu-2} C_{p-2j}(\theta).
\end{align*}
\end{proof}
\begin{Remark}
    Explicit constants for non-integer $p\geq 2$ can be obtained by interpolation.
\end{Remark}
\noindent For the embedding $H^1 \hookrightarrow L^6$, we find
\begin{Corollary}
    There holds
    \begin{align*}
        \|u\|_{L^6(\mathbb T^2)}\leq C_{\text{S}}\,\|u\|_{H^1(\mathbb T^2)},
    \end{align*}
    with $C_{\text{S}} := 1.1401792579142$.
\end{Corollary}
\begin{proof}
    We use Proposition~\ref{prop:CS2D} with $p=6$, i.e., 
    \begin{align*}
        C_6^{\rm it} = \big(\sup_{0\leq\theta\leq1}(1-\theta)C_4(\theta)C_6(\theta)\big)^{1/6},
    \end{align*}
    with $ C_4(\theta)=1-\frac{\theta}{2}+\sqrt2\,\sqrt{\theta(1-\theta)}$ and $C_6(\theta)=1+\frac{\theta}{8}+\frac{3}{\sqrt2}\sqrt{\theta(1-\theta)}$.
    Setting
    \[
    s:=\sqrt{\frac{\theta}{1-\theta}}
    \quad \iff \quad
    \theta=\frac{s^2}{1+s^2},
    \]
    we find
    \begin{equation}
        \sup_{0\leq\theta\leq1}(1-\theta)C_4(\theta)C_6(\theta) = \sup_{s \in [0,\infty)} F(s), 
    \end{equation}
    with $F(s)=\frac{(s^2+2\sqrt2\,s+2)(9s^2+12\sqrt2\,s+8)}{16(1+s^2)^3}$.
    We now maximize this rational function.
    Differentiating and simplifying gives
    \begin{align*}
    F^\prime(s) = -\frac{ P(s) }{8(1+s^2)^4}
    \end{align*}
    with $P(s)=9s^5 +45\sqrt2\,s^4 +130s^3 +55\sqrt2\,s^2 -26s -20\sqrt2.$
    Since the coefficients of $P$ exhibit exactly one sign change,
    Descartes' rule of signs implies that $P$ has at most one positive
    zero. Moreover,
    \[
    P(0)=-20\sqrt2<0,
    \qquad
    \lim_{s\to\infty}P(s)=\infty,
    \]
    so $P$ has exactly one positive zero, say $s_*$. Hence,
    $F(\cdot)$ is increasing on $[0,s_*]$ and decreasing on $[s_*,\infty)$.
    Consequently, $s_*$ is the unique global maximizer.
    \noindent Numerically, $s_*\approx0.507126874583414$.
    To obtain a rigorous bound, we enclose the algebraic root of $P$ by rational numbers.  Direct evaluation gives
    \[
    P(0.50712687458342)<0<P(0.50712687458341),
    \]
    and therefore
    \[
    0.50712687458341<s_*<0.50712687458342.
    \]
    Setting $a:=0.50712687458341$ and $b:=0.50712687458342$, we find
    \begin{align*}
        F(s_*) \leq \sup_{s \in [a,b]} F(s) \leq \frac{(b^2+2\sqrt2\,b+2)(9b^2+12\sqrt2\,b+8)}{16(1+a^2)^3} <2.197044312999.
    \end{align*}
    Consequently, $C_{6}^{\rm it} \leq (2.197044312999)^{1/6} \leq 1.1401792579142 = C_{\text{S}}$.
\end{proof}
\subsubsection{The embedding $W^{1,r} \hookrightarrow L^p$ on $\mathbb{T}^3$}
Next, we are interested in $d=3$. We show a general statement for the entire family of subcritical exponents $1<r<3$ and recover the special case $r=2$, needed for our application.
Fix $1<r<3$ and set
\[
p := \frac{3r}{3-r}, \qquad r' := \frac{r}{r-1}, \qquad s := \frac{2p}{3} = \frac{2r}{3-r}.
\]
Here $p$ is the usual Sobolev conjugate of $r$ for $d=3$. %Note, as $r\in(1,3)$ we have $p\in(3/2,\infty)$. 
For $u\in W^{1,r}(\mathbb{T}^3)$ we write
\[
X:=\|u\|_{L^r}, \qquad g_i := \|\partial_i u\|_{L^r} \ (i=1,2,3), \qquad Y:=\|u\|_{L^p},
\]
and $N:=\|u\|_{W^{1,r}}$, defined by $N^r := X^r + \sum_i g_i^r$. 
The central identity that makes the method work for every such $r$ is
\begin{equation}\label{eq:identity}
(s-1)r' = p.
\end{equation}
% which can be verified via $s-1 = \frac{2q}{3}-1 = \frac{3(r-1)}{3-r}$, so $(s-1)r' = \frac{3(r-1)}{3-r}\cdot\frac{r}{r-1} = \frac{3r}{3-r}=q$.
\begin{Theorem}\label{thm:general}
Let $1<r<3$ and $p=\dfrac{3r}{3-r}$. For $u\in W^{1,r}(\mathbb{T}^3)$, there holds
\[
\|u\|_{L^p(\mathbb{T}^3)} \le C_{r}\,\|u\|_{W^{1,r}(\mathbb{T}^3)}, \qquad C_{r} = \Big(1+3^{-\frac{1}{r-1}}\Big(\frac{r}{3-r}\Big)^{\frac{r}{r-1}}\Big)^{\frac{r-1}{r}}.
\]
\end{Theorem}
\noindent In order to prove this, we need several auxiliary results. We first establish a Loomis--Whitney-type inequality on $\mathbb{T}^3$.
\begin{Lemma}\label{lem:LW}
Let $F_1,F_2,F_3\in L^1(\mathbb{T}^2)$ be nonnegative, where $F_i$ is a function of the two coordinates other than $x_i$, then
\[
\int_{\mathbb{T}^3} F_1(x_2,x_3)^{1/2}F_2(x_1,x_3)^{1/2}F_3(x_1,x_2)^{1/2}\,dx \le \|F_1\|_{L^1}^{1/2}\|F_2\|_{L^1}^{1/2}\|F_3\|_{L^1}^{1/2}.
\]
\end{Lemma}
\begin{proof}
Fix $x_3$ and apply Cauchy--Schwarz in the variables $(x_1,x_2)$:
\[
\int_{\mathbb{T}^2} F_1^{1/2}F_2^{1/2}F_3^{1/2}\,dx_1\,dx_2 \le \Big(\int_{\mathbb{T}^2} F_1(x_2,x_3)F_2(x_1,x_3)\,dx_1\,dx_2\Big)^{1/2}\Big(\int_{\mathbb{T}^2} F_3(x_1,x_2)\,dx_1\,dx_2\Big)^{1/2}.
\]
By Fubini, the first factor splits as $\big(\int_0^1 F_1(x_2,x_3)\,dx_2\big)^{1/2}\big(\int_0^1 F_2(x_1,x_3)\,dx_1\big)^{1/2}$. Now integrate over $x_3\in[0,1]$ and apply Cauchy--Schwarz once more to the product of the two remaining $x_3$-dependent factors:
\[
\int_0^1 \Big(\int_0^1 F_1\,dx_2\Big)^{1/2}\Big(\int_0^1 F_2\,dx_1\Big)^{1/2} dx_3 \le \Big(\int_{\mathbb{T}^2} F_1\,dx_2 dx_3\Big)^{1/2}\Big(\int_{\mathbb{T}^2} F_2\,dx_1 dx_3\Big)^{1/2} = \|F_1\|_{L^1}^{1/2}\|F_2\|_{L^1}^{1/2}.
\]
Multiplying both sides with the constant factor $\|F_3\|_{L^1}^{1/2}$ yields the claim.
\end{proof}
\begin{Lemma}[Interpolation inequalities]\label{lem:genInterp}
For $u\in W^{1,r}(\mathbb{T}^3)$ and $v\in L^r(\mathbb{T}^3)$,
\[
\|u\|_{L^s}^s \le X\,Y^{s-1}, \qquad \int_{\mathbb{T}^3} |u|^{s-1}|v|\,dx \le Y^{s-1}\|v\|_{L^r}.
\]
\end{Lemma}
\begin{proof}
Write $|u|^s = |u|\,|u|^{s-1}$ and apply H\"older with conjugate exponents $r$ and $r'$,
\[
\int_{\mathbb{T}^3} |u|^s\,dx \le \Big(\int_{\mathbb{T}^3} |u|^r\,dx\Big)^{1/r}\Big(\int_{\mathbb{T}^3} |u|^{(s-1)r'}\,dx\Big)^{1/r'} = X\,Y^{s-1},
\]
using \eqref{eq:identity} in the last step. The second inequality is the same H\"older pair applied to $|u|^{s-1}|v|$.
\end{proof}
\begin{Theorem}\label{thm:d3general}
For all $u\in W^{1,r}(\mathbb{T}^3)$,
\[
Y^3 \le \prod_{i=1}^3 \Big(X+\frac{p}{3}g_i\Big) =: P(X,g_1,g_2,g_3).
\]
\end{Theorem}
\begin{proof}
Set $x := (x_1,x_2,x_3)$, fix $i\in\{1,2,3\}$ and apply Lemma~\ref{lem:osc} to $h(x_i):=|u(x_1,x_2,x_3)|^s$. 
Since $|h'(x_i)| = s|u(x)|^{s-1}|\partial_i u(x)|$, we obtain the pointwise estimate
\[
|u(x)|^s \le \int_0^1 |u|^s\,dx_i + \frac{s}{2}\int_0^1 |u|^{s-1}|\partial_i u|\,dx_i =: F_i(\hat x_i),
\]
where $\hat x_i$ denotes the two coordinates other than $x_i$. Multiplying the three bounds for $i=1,2,3$ gives $|u(x)|^{3s} \le F_1(\hat x_1)F_2(\hat x_2)F_3(\hat x_3)$. Since $3s=2p$, taking a square root yields
\[
|u(x)|^p \le F_1(\hat x_1)^{1/2}F_2(\hat x_2)^{1/2}F_3(\hat x_3)^{1/2}.
\]
Integrating over $\mathbb{T}^3$ and applying Lemma~\ref{lem:LW},
\begin{equation}\label{eq:d3star}
Y^p \le \|F_1\|_{L^1}^{1/2}\|F_2\|_{L^1}^{1/2}\|F_3\|_{L^1}^{1/2}.
\end{equation}
By Lemma~\ref{lem:genInterp} with $v=\partial_i u$,
\[
\|F_i\|_{L^1} = \int_{\mathbb{T}^3}|u|^s\,dx + \frac{s}{2}\int_{\mathbb{T}^3}|u|^{s-1}|\partial_i u|\,dx \le XY^{s-1} + \frac{s}{2}Y^{s-1}g_i = Y^{s-1}\Big(X+\frac{s}{2}g_i\Big).
\]
Substituting into \eqref{eq:d3star} and using $\frac{3(s-1)}{2}=p-\frac32$, leads to
\[
Y^p \le Y^{p-3/2}\prod_{i=1}^3 \Big(X+\frac{s}{2}g_i\Big)^{1/2}.
\]
If $Y=0$, the claim is trivial. Otherwise, we divide by $Y^{p-3/2}$ and square. 
Since $s/2=p/3$, this infers the desired claim.
\end{proof}
\noindent As $P$ is homogeneous of degree $3$, we observe
\begin{align*}
    Y^3 \leq P(X,g_1,g_2,g_3) = P\Big(\frac{X}{N},\frac{g_1}{N},\frac{g_2}{N},\frac{g_3}{N}\Big) N^3,
\end{align*}
with $N^r = X^r + g_1^r + g_2^r + g_3^r \Leftrightarrow 1 = (\frac{X}{N})^r + (\frac{g_1}{N})^r + (\frac{g_2}{N})^r + (\frac{g_3}{N})^r$. 
Thus, $\|u\|_{L^p} \leq C \|u\|_{W^{1,r}}$ will hold for every $C>0$ with
\[
\max\big\{P(X,g_1,g_2,g_3): X,g_1,g_2,g_3\ge0,\ X^r+g_1^r+g_2^r+g_3^r=1\big\} \leq C^3.
\]
\begin{Proposition}\label{prop:general}
Let $a:=\frac{p}{3}$ and $b:=a\,3^{-1/r}$. Then
\[
\max\big\{P(X,g_1,g_2,g_3): X,g_1,g_2,g_3\ge0,\ X^r+ g_1^r + g_2^r + g_3^r = 1\big\} = \big(1+b^{r'}\big)^{3/r'},
\]
is attained if and only if $g_1=g_2=g_3=:g_*$ with
\[
g_*^r = \frac{b^{r'}}{3\big(1+b^{r'}\big)}, \qquad X_*^r = \frac{1}{1+b^{r'}}.
\]
\end{Proposition}
\begin{proof}
We combine three elementary inequalities, each with a sufficient sharpness condition. \\
\noindent \emph{Step 1:} For reals $a_1,a_2,a_3\ge0$, we have $a_1a_2a_3 \le \big(\frac{a_1+a_2+a_3}{3}\big)^3$, with equality if and only if $a_1=a_2=a_3$. 
Applying this with $a_i = X+ag_i \ge 0$, yields
\[
\prod_{i=1}^3 (X+ag_i) \le (X+a\bar g)^3, \qquad \bar g := \frac{g_1+g_2+g_3}{3},
\]
with equality if and only if $g_1=g_2=g_3$. \\
\noindent \emph{Step 2:} By H\"older with conjugate exponents $r$ and $r'$, we find
\[
g_1+g_2+g_3 = (1,1,1)\cdot (g_1,g_2,g_3) \le 3^{1/r'}\Big(\sum_{i=1}^3 g_i^r\Big)^{1/r} =:3^{1/r'} G_r,
\]
with equality if and only if $g_1=g_2=g_3$. 
This gives $\bar g \le 3^{-1/r}G_r$, so $X+a\bar g \le X+bG_r$ with $b=a\cdot3^{-1/r}$. \\
\noindent \emph{Step 3:} By H\"older, again with conjugate exponents $r$ and $r'$, we find
\[
X+bG_r = (X,G_r) \cdot (1,b) \le (X^r+G_r^r)^{1/r}(1+b^{r'})^{1/r'} = (1+b^{r'})^{1/r'},
\]
with equality if and only if $(X^r,G_r^r)$ is proportional to $(1,b^{r'})$, i.e.
\[
X^r = \frac{1}{1+b^{r'}}, \text{ and } G_r^r = \frac{b^{r'}}{1+b^{r'}},
\]
additionally using $X^r+G_r^r = X^r+\sum_{i=1}^3 g_i^r = 1$.
Raising everything to the power of $3$ yields the claimed maximum with equality throughout if and only if $g_1=g_2=g_3$ and the above hold. Since $G_r^r=3g^r$, we find $g^r = \frac{b^{r'}}{3(1+b^{r'})}$.
\end{proof}
\begin{proof}[Proof of Theorem~\ref{thm:general}]
Both sides of Theorem~\ref{thm:d3general} and of the constraint $X^r+\sum_i g_i^r = N^r$ are homogeneous of degree $3$ in $(X,g_1,g_2,g_3)$. Normalizing and applying Proposition~\ref{prop:general} gives $Y^3\le(1+b^{r'})^{3/r'}$ whenever $X^r+\sum_i g_i^r=1$. Rescaling back and raising both sides to the power of $\frac13$ gives $Y\le C_{r}N$ with $C_{r}=(1+b^{r'})^{1/r'}$. 
Since $a=\frac{p}{3}=\frac{r}{3-r}$ and $\frac{r^\prime}{r} = \frac{1}{r-1}$, we have
\[
b^{r'} = \big(a\cdot3^{-1/r}\big)^{r'} = \Big(\frac{r}{3-r}\Big)^{r'}3^{-r'/r} = 3^{-\frac1{r-1}}\Big(\frac{r}{3-r}\Big)^{\frac{r}{r-1}},
\]
which yields the desired result.
\end{proof}
\noindent Specializing Theorem~\ref{thm:general} to $r=2$ leads to
\begin{Corollary}\label{cor:special}
For all $u\in H^1(\mathbb{T}^3)$ we have
\[
\|u\|_{L^6(\mathbb{T}^3)} \le C_{S}\|u\|_{H^1(\mathbb{T}^3)}, \qquad C_{S} = \Big(1+\frac43\Big)^{1/2} = \sqrt{\frac73} \le 1.52753.
\]
\end{Corollary}